\documentclass[11pt]{article}
\usepackage{graphicx}
\usepackage[left=0.5in,right=1in,top=0.7in,bottom=1in]{geometry}
\usepackage{amssymb,longtable,amsmath,amsthm}
\usepackage{epstopdf,tcolorbox}
\usepackage{amsmath,empheq}
\usepackage{authblk}
\newcommand{\R}{\mathbb{R}}
\renewcommand{\S}{\mathbb{S}}
\newcommand{\Z}{\mathbb{Z}}

\newcommand{\C}{\mathbb{C}}

\newcommand{\E}{{\cal E}}
\newcommand{\F}{{\cal F}}

\newcommand{\tR}{\tilde{R}}

\newcommand{\lec}{\lesssim}
\newcommand{\gec}{\gtrsim}

\newcommand{\p}{\partial}

\renewcommand{\k}{\kappa}
\newcommand{\e}{\epsilon}
\newcommand{\al}{\alpha}
\newcommand{\la}{\lambda}

\renewcommand{\th}{\theta}
\newcommand{\be}{\beta}

\newtheorem{prop}{Proposition}

\newtheorem{rem}{Remark}

\newtheorem{lem}{Lemma}
\newtheorem{thm}{Theorem}
\newtheorem{cor}{Corollary}

\begin{document}

\title{Co-rotational chiral magnetic skyrmions near harmonic maps}
\author{S. Gustafson
\thanks{\noindent
(corresponding author)
gustaf@math.ubc.ca,
University of British Columbia,
1984 Mathematics Rd., Vancouver, Canada V6T1Z2}
\quad Li Wang
\thanks{\noindent
lwang@math.ubc.ca,
University of British Columbia,
1984 Mathematics Rd., Vancouver, Canada V6T1Z2}
}
 
\date{}
\maketitle

\abstract{
Chiral magnetic skyrmions are topological solitons, of significant physical
interest, arising in ferromagnets described by a micromagnetic energy including a 
chiral (Dzyaloshinskii-Moriya) interaction term. 
We show that for small chiral interaction, the skyrmions on $\R^2$ with 
co-rotational symmetry are close to harmonic maps, and prove precise bounds 
on the differences. One application of these bounds is precise energy asymptotics.
Another (pursued in a separate work) is an alternate, quantitative proof of
the recent skyrmion stability result of Li-Melcher.
}

\medskip



\tableofcontents

\section{Introduction and Main Results}

\subsection{Magnetization configurations}

{\it Micromagnetics} is a widely used continuum theory for describing the static and 
dynamic behaviour of ferromagnets \cite{brown1963micromagnetics,landau1935lifschiz,kosevich1990magnetic}. 
The state of the ferromagnet (at a fixed time) is given by its {\it magnetization},
a constant-length (here normalized to one) vector field
\[
  \hat{m}(x) = \left[ \begin{array}{c} m_1(x) \\ m_2(x) \\ m_3(x) \end{array} \right]
  \qquad |\hat{m}(x)|^2 = m_1^2(x) + m_2^2(x) + m_3^2(x) \equiv 1.
\]
Here we consider an infinite, two-dimensional magnet, so that $\hat{m}(x)$ is defined
for all $x = (x_1, x_2) \in \R^2$. Geometrically, the magnetization is a map into
the unit $2$-sphere   
\begin{equation}  \label{unit}
  \hat m : \R^2 \to \S^2, \qquad \S^2 = \{ \hat m \in \R^3 \; | \; |\hat m| = 1 \}
  \subset \R^3.
\end{equation}

The behaviour of the magnet is determined by a  
{\it micromagnetic energy} functional $\E(\hat m)$ of the magnetization $\hat m$.
Equlibrium (static) configurations are critical points of this energy, so satisfy the
Euler-Lagrange equation
\begin{equation}  \label{E-L0}
  \E'(\hat m) = 0,
\end{equation}
while the dynamics is given by the {\it Landau-Lifshitz} 
(or {\it Landau-Lifshitz-Gilbert}) equation
\begin{equation}  \label{L-L}
  \p_t \hat m = \hat m \times \E'(\hat m) - \gamma \; \E'(\hat m)
\end{equation}
where $\gamma \geq 0$ is a damping parameter.

The contribution to this energy which reflects the 
ferromagnetic character of the material is the {\it exchange energy}
\[
  \E_e(\hat{m}) := \frac{1}{2} \int_{\R^2} |\nabla \hat{m}(x)|^2 dx.
\]
Configurations with finite exchange energy may be classified by topological
degree (or {\it Skyrmion number}) \cite{schoen1983boundary}
\begin{equation}  \label{number}
  N := \frac{1}{4 \pi} \int_{\R^2} \hat{m} \cdot 
  \left( \p_{x_2} \hat{m} \times \p_{x_1} \hat{m} \right) dx = 
  \frac{1}{4 \pi} \int_{\R^2} \p_{x_1} \hat{m} \cdot J \p_{x_2} \hat{m} \; dx \;\; 
  \in \Z,
\end{equation}
where here $J$ denotes the $\frac{\pi}{2}$ rotation (i.e. complex structure) on the tangent plane
\begin{equation}   \label{tangent}
  T_{\hat{m}} \S^2 = \{ \xi \in \R^3 \; | \; \hat{m} \cdot \xi = 0\}
\end{equation}
to the sphere at $\hat m$:
\begin{equation}  \label{J}
\begin{split}
  J = J_{\hat{m}} \; : \; & T_{\hat{m}} \S^2 \to T_{\hat{m}} \S^2 \\
  & \quad \xi \;\; \mapsto \; \hat{m} \times \xi .
\end{split}
\end{equation}
The degree $N$ determines the minimum possible exchange energy
by a classical ``completing-the-square" computation 
\cite{belavin1975metastable,kosevich1990magnetic}:
\begin{equation}  \label{topo0}
  \E_e(\hat{m}) = \frac{1}{2} \int_{\R^2} 
  \left( |\p_{x_1} \hat{m}|^2 +  |\p_{x_2} \hat{m}|^2 \right) dx =
  \frac{1}{2} \int_{\R^2} \left| (\p_{x_1} \mp J \p_{x_2}) \hat{m} \right|^2 dx 
  \pm 4 \pi N \geq 4 \pi |N|.
\end{equation}
This lower bound is attained by magnetizations given by rational 
functions of $z = x_1 + i x_2$ or $\bar{z} = x_1 - i x_2$
when represented via stereographic projection
$w(z,\bar z) = \frac{m_1(x) + i m_2(x)}{1 + m_3(x)} \in \C$.

\subsection{Micromagnetic energy}

A large physics literature (see, e.g., \cite{
bogdanov1989thermodynamically,
bocdanov1994properties,
bogdanov1994thermodynamically,  
komineas2015skyrmion})
has confirmed that for certain {\it chiral} magnets,
an important spin-orbit coupling effect is well modelled by  
a {\it Dzyaloshinskii-Moriya} contribution to the energy
\[
  \E_{DM}(\hat{m}) =  \frac{1}{2} \int_{\R^2} \hat{m} \cdot (\nabla \times \hat{m}) dx
  = \int_{\R^2} \left( m_1 \partial_{x_2} m_3 - m_2 \partial_{x_1} m_3 \right) dx
\]
(the second expression comes from a formal integration by parts).
Following these studies, we include this effect, as well as an anisotropy
in the direction orthogonal to the plane of the magnet
\[
  \E_a(\hat{m}) = \frac{1}{2} \int_{\R^2} (1 - \hat m \cdot \hat k) dx
  = \frac{1}{2} \int_{\R^2} (1 - m_3^2) dx, \qquad
  \hat k = \left[ \begin{array}{c} 0 \\ 0 \\ 1 \end{array} \right],
\] 
and a Zeeman (constant external magnetic field along $\hat{k}$) energy: 
\[
  \E_z(\hat{m}) =  \int_{\R^2} (1 - \hat m \cdot \hat k) dx = 
  \int_{\R^2} (1 - m_3) dx.
\]
Combining these with the exchange energy, with coefficients, yields
the full micromagnetic energy we consider here:
\[
  \E_{k,\al,\be} (\hat{m}) = \E_e(\hat{m}) + \be \; k \; \E_{DM}(\hat{m}) 
  + \be^2 \al \; \E_a(\hat{m}) + \be^2 (1-\al) \; \E_z(\hat{m}), 
  \qquad k \in \R,  \;\; \be > 0, \;\; \al \leq 1.
\]
The form of the coefficients requires some explanation. 
First, we will combine the anisotropy and Zeeman energies and write
\[
  \E^{(\al)}(\hat{m}) := \al \E_a(\hat{m}) + (1- \al) \E_z(\hat{m}) =
  \int_{\R^2} \left( 1 - m_3 - \frac{1}{2} \al (1 - m_3)^2 \right) dx.
\]
This includes the special cases
\[
  \E^{(1)} = \E_a, \;\; \E^{(0)} = \E_z,
\]
but in general allows both {\it easy-axis} ($\al > 0$) and {\it easy-plane}
($\al < 0$) type anisotropies (though, in the latter case, balanced by a
sufficiently strong Zeeman term).  
Note that
\begin{equation}  \label{alpha}
  \al \leq 1 \; \implies \; \E^{(\al)} \geq \E^{(1)} = \E_a \geq 0.
\end{equation}
The full energy is then written as
\begin{equation}  \label{fullen}
  \E_{k,\al,\be}(\hat m) = \E_e(\hat m) + \be \; k \; \E_{DM}(\hat m) 
  + \be^2 \; \E^{(\al)}(\hat m) ,  \qquad k \in \R,  \;\; \be > 0, \;\; \al \leq 1.
\end{equation}
Second, note that $\be$ is purely a scaling parameter: for $\la > 0$,
\begin{equation}  \label{law}
  \E_{k,\al,\be}(\hat m(\la \; \cdot)) = \E_e(\hat m) + \frac{\be}{\la} \; k \; 
  \E_{DM}(\hat m) + \frac{\be^2}{\la^2} \; \E^{(\al)}(\hat m)
  = \E_{k, \al, \frac{\be}{\la}}(\hat m),
\end{equation}
so in particular we may arrange $\be = 1$ by taking $\la = \beta$ to give
\[
  \E_{k,\al,\be}(\hat m(\be \; \cdot)) =  \E_{k,\al,1}(\hat m).   
\]

Because of the unit length constraint~\eqref{unit} on $\hat m$, the Euler-Lagrange equation~\eqref{E-L0} for critical points of this energy functional is
\begin{equation}  \label{E-L1}
\begin{split}
  0 &= \E'_{k,\al,\be}(\hat m) =  P_{T_{\hat m} \S^2} \left[   
  -\Delta \hat m + \be \; k \; \nabla \times \hat m 
  + \be^2 \left(   \al- 1 - \al m_3 \right) \hat k \right] \\
  &= -\Delta \hat m - |\nabla \hat m|^2 \hat m + \be \; k
  \left( \nabla \times \hat m - (\hat m \cdot (\nabla \times \hat m)) \hat m \right)
  + \be^2 \left(  \al- 1 - \al m_3 \right) \left( \hat k - m_3 \hat m \right)
\end{split}
\end{equation}
where 
\[
  P_{T_{\hat m} \S^2} : \xi \mapsto \xi - (\hat m \cdot \xi) \xi
\]
is the orthogonal projection from $\R^3$ onto the tangent plane~\eqref{tangent}.
Though the energy functional appears to be quadratic in the magnetization $\hat m$,
the Euler-Lagrange equation~\eqref{E-L1} is clearly nonlinear, as a 
consequence of the geometric constraint~\eqref{unit}. 

In the physics literature, energies of this type were used to predict the
existence of stable (energy minimizing), topologically non-trivial, spatially localized
solutions of~\eqref{E-L1} called {\it skyrmions} (or {\it chiral magnetic skyrmions})
\cite{bogdanov1989thermodynamically,
bogdanov1994thermodynamically,
kiselev2011chiral}. 
Such  configurations were later observed 
experimentally \cite{romming2013writing,romming2015field}  and are of potential technological importance 
\cite{kiselev2011chiral,
fert2013skyrmions,
sampaio2013nucleation,
nagaosa2013topological,
iwasaki2013current,
buttner2015dynamics,
woo2016observation}. 

Mathematically, it was shown in 
\cite{melcher2014chiral,doring2017compactness} that the energy~\eqref{fullen} 
has a global minimizer under the constraint that the skyrmion number~\eqref{number} is fixed at $N=1$, rigorously establishing the existence of chiral skyrmions.

\subsection{Symmetric reduction}

Much of the physics literature 
(e.g. \cite{komineas2015skyrmion,
bogdanov1994thermodynamically,
kiselev2011chiral,
bogdanov1999stability} )
concerns chiral skyrmion configurations
with {\it co-rotational} symmetry. Here we explain the reduction to this symmetry class.
It seems not to be known if the global minimizer of 
\cite{melcher2014chiral, doring2017compactness} has this
symmetry.  

It is easily verified that the energy functionals
$\E_e$, $\E_a$, and $\E_z$ (hence also $\E^{(\al)}$) are each separately invariant under
\begin{itemize}
\item
spatial rotations: \; $\hat{m}(x) \mapsto \hat{m}(e^{\phi \tilde{R}} x)$, \quad
$\tilde{R} = \left[ \begin{array}{cc} 0 & -1 \\ 1 & 0 \end{array} \right]$, \quad
$e^{\phi\tilde{R}} = \left[ \begin{array}{cc} \cos \phi & -\sin \phi \\ \sin \phi & \cos \phi \end{array} \right]$, \quad $\phi \in \R$;
\item
target rotations fixing $\hat k$: \; $\hat{m}(x) \mapsto e^{ \phi R } \hat{m}(x)$, \;
$R = \left[ \begin{array}{ccc} 0 & -1 & 0 \\ 1 & 0 & 0 \\ 0 & 0 & 0 
\end{array} \right]$, \;
$e^{\phi R} = \left[ \begin{array}{ccc} \cos \phi & -\sin \phi & 0 \\ 
\sin \phi & \cos \phi & 0 \\ 0 & 0 & 1 \end{array} \right]$;
\item
spatial reflections: \; $\hat{m}(x) \mapsto \hat{m}(\pm \tilde F \hat m)$,
\quad $\tilde{F} = \left[ \begin{array}{cc} 1 & 0 \\ 0 & -1 \end{array} \right]$;
\item
target reflections fixing $\hat k$: \; 
$\hat m(x) \mapsto \pm F \hat m(x)$, \quad
$F = \left[ \begin{array}{ccc} 1 & 0 & 0 \\ 0 & -1 & 0 \\ 0 & 0 & 1 
\end{array} \right]$.  
\end{itemize}
The chiral energy term $\E_{DM}$ breaks each of these symmetries, but
the following lemma, proved in Section~\ref{symm}, shows
that it retains invariance under combined spatial and target rotations
\begin{equation}  \label{equivariant}
  \hat{m}(x) \mapsto e^{-\phi R} \hat{m}( e^{\phi \tR} x ), \qquad 
  \phi \in \R,
\end{equation}
and combined reflections
\begin{equation}  \label{corot}
  \hat{m}(x) \mapsto - F \hat{m} ( \tilde F \; x).
\end{equation}
\begin{lem}  \label{invar}
We have 
\begin{equation}  \label{symmetry}
   \E_{DM}\left( e^{-\phi R} \hat{m}( e^{\phi \tR} \cdot ) \right) =  \E_{DM}(\hat{m})
   \qquad \forall \phi \in \R, 
\end{equation}
and
\begin{equation}  \label{symmetry2}
  \E_{DM} \left( -F \hat{m} ( \tilde{F} \; \cdot ) \right) = \E_{DM}(\hat{m}).
\end{equation}
\end{lem}

We refer to magnetization vectors $\hat{m}$ which are invariant under the 
transformation~\eqref{equivariant} as {\it equivariant}, and those which are
additionally invariant under~\eqref{corot}  as {\it co-rotational}. 
The latter take the special form
\begin{equation} \label{corotational}
  \hat{m}(x) = \left[ \; -\sin(u(r)) \sin(\theta), \; \sin(u(r)) \cos(\theta), \; \cos(u(r)) \; \right],
  \qquad u : [0,\infty) \to \R,
\end{equation}
where $(r,\theta)$ are polar coordinates on $\R^2$
(up to a reflection $\hat m \mapsto e^{\pi R} \hat m$, which reverses the sign of 
$\E_{DM}(\hat{m})$ and so can be absorbed by change of coefficient 
$k \mapsto -k$).
By Lemma~\ref{invar}, the full energy $\E_{k,\al,\be}$ is invariant under
both~\eqref{equivariant} and~\eqref{corot}, and so a critical point of
$\E_{k,\al,\be}$ restricted to co-rotational maps~\eqref{corotational}
will be a critical point of $\E_{k,\al,\be}$, hence a solution of the
Euler-Lagrange equation~\eqref{E-L1} -- see Proposition~\ref{differentiability}.
\begin{rem}
The co-rotational class~\eqref{corotational} is however {\it not} preserved by the 
dynamical Landau-Lifshitz equation~\eqref{L-L}, while the the larger equivariant class is.
This is because the reflection~\eqref{corot} does not commute with the 
complex structure~\eqref{J}, while the rotations~\eqref{equivariant} do. 
\end{rem}

A straightforward computation shows that the full energy~\eqref{fullen} of a co-rotational magnetization~\eqref{corotational} produces the following functional of the profile $u$:
\begin{equation}  \label{symmen}
  E_{k,\al,\be}(u) := \frac{1}{2\pi} \E_{k,\al,\be}(\hat{m}) =
  E_e(u) + \be \; k \; E_{DM}(u) + \be^2 \; E^{(\al)} (u) 
\end{equation}
where
\[  
\begin{split}
  &E_e(u) = \frac{1}{2} \int_0^\infty \left( u_r^2 + \frac{1}{r^2} \sin^2(u(r)) \right) r dr \\
  &E_{DM}(u) = \int_0^\infty \sin^2(u(r)) u_r \; r dr \\
  &E^{(\al)}(u) = \int_0^\infty \left( 1 - \cos(u(r)) - \frac{1}{2} \al (1-\cos(u(r)))^2  
  \right)  \; r dr.
\end{split}
\]
Another straightforward computation shows that for co-rotational magnetizations
\[
  \p_{x_1} \hat m \cdot J_{\hat m} \p_{x_2} \hat m =
  -\frac{1}{r} \sin u \; u_r = \frac{1}{r} ( \cos u )_r
\]
so by~\eqref{number}, the Skyrmion number of a co-rotational map is
\begin{equation}  \label{number2}
  N = \frac{1}{4\pi} \; 2 \pi \; \int_0^\infty ( \cos u )_r \; dr =
  \frac{1}{2} \cos u(r) \big|^{r=\infty}_{r=0}.
\end{equation}
By the scaling law~\eqref{law}, we have 
\begin{equation}  \label{law2}
  E_{k,\al,\be}(u(\la \cdot)) = E_{k,\al,\frac{\be}{\la}}(u),
\end{equation}

\subsection{Main results}

We are interested in minimizing profiles $u$ of the reduced energy~\eqref{symmen},
especially those satisfying boundary conditions
\begin{equation}  \label{BC0}
  u(0) = \pi, \;\; u(\infty) = 0 \;\; \implies \;\; N = 1
\end{equation}
by~\eqref{number2}. 
Such minimizers produce solutions of the form~\eqref{corotational} to the full 
Euler-Lagrange equations~\eqref{E-L1}. These are chiral magnetic skyrmions
solutions with co-rotational symmetry and Skyrmion number $N=1$. 

\cite{li2018stability} showed the existence of minimizers of~\eqref{symmen}
in this topological class for $0 < k < 1$, in the case $\al = 0$.
They also showed that for $k \ll 1$ the minimizers are unique and monotone.

Our approach is to treat the minimization problem for~\eqref{symmen}, for $k \ll 1$, as
a perturbation of the minimization problem for the $k = \be = 0$ limit energy 
$E_{0,\al,0} = E_e$.
The minimizers of the exchange energy $E_e$ with the boundary conditions~\eqref{BC0} are well-known to be 
\cite{belavin1975metastable,kosevich1990magnetic}
\[
  u(r) = Q(r/s) \; \mbox{ for some } s > 0, \quad
  Q(r) = \pi - 2 \tan^{-1}(r).
\]
The {\it bubble} $Q$ corresponds, via~\eqref{corotational}, to a degree-one harmonic map
$\R^2 \to \S^2$. Note that this is a one-parameter family, due to the scale invariance of $E_e$.  

Our main result makes precise the sense in which the skyrmion profiles
are perturbations of the bubble $Q$. To state it, we introduce the family of Banach spaces
\[
  X^p := \{ \xi : [0,\infty) \to \R \; | \; \| \xi \|_{X^p} < \infty \}, \quad
  \| \xi \|_{X^p} := \left\| \xi_r \right\|_{L^p} + \left\| \frac{\xi}{r} \right\|_{L^p}, \quad
  1 \leq p \leq \infty
\]
where 
\[
  \| f \|_{L^p} := \left( \int_0^\infty f^p(r) \; r dr \right)^{\frac{1}{p}}
\] 
denotes the $L^p$ norm of the radially-symmetric function $f(|x|)$ on $\R^2$.
The case $p=2$,
\[
  X = X^2 = \{ \xi : [0,\infty) \to \R \; | \; \| \xi \|_X := \| \xi_r \|_{L^2} +
  \left\| \frac{\xi}{r} \right\|_{L^2} < \infty \}, 
\]
plays a key role as our ``energy space".
\begin{thm} \label{mainthm}
Let $A > 0$ be given. There is $k_0 > 0$ and $C > 0$ such that for
each $0 < k \leq k_0$, $-A \leq \al \leq 1$, $\be > 0$, 
there is a unique minimizer $v_{k,\al, \be}$ of $E_{k,\al, \be}$
in the class $Q + X$. The corresponding map
\[
  \hat m_{k,\al, \be} = 
  \left[ \; -\sin(v_{k,\al, \be}) \sin(\theta), \; \sin(v_{k,\al, \be}) \cos(\theta), \; 
  \cos(v_{k,\al, \be}) \; \right]
\]
is a smooth solution of the Euler-Lagrange equations~\eqref{E-L1} with 
Skyrmion number $N=1$. Moreover, for a particular choice of scaling
$\be = \be(k) \to 0$ as $k \to 0$, satisfying
\begin{equation}  \label{relationThm}
  k = \be \left( 2 \log \left( \frac{1}{\be} \right) + O(1) \right),
\end{equation}
the skyrmion profile satisfies the estimates
\begin{equation}  \label{estimatesThm}
  \| v_{k,\al,\be} - Q \|_X \leq C \be, \qquad
  \| v_{k,\al,\be} - Q \|_{X^\infty} \leq C \be^2 \log \left( \frac{1}{\be} \right).
\end{equation}
Finally, the skyrmion energy behaves as
\begin{equation}  \label{energyThm}
  E_{k,\al,\be}(v_{k,\al,\be}) = E_e(Q) - \frac{k^2}{2 \log \left( \frac{1}{k} \right)}
  \left(1 + O \left( \frac{1}{\log \left( \frac{1}{k} \right)} \right) \right).   
\end{equation}
\end{thm}
\begin{rem}
By the scaling relation~\eqref{law2}, for any $\tilde \be > 0$,
$v_{k,\al, \tilde \be} \left( \frac{\be}{\tilde \be} r \right) = v_{k, \al, \be}(r)$,
so the estimates~\eqref{estimatesThm} give the $k \to 0$ behaviour
of all the skyrmion profiles.
\end{rem}

Compared to \cite{li2018stability}, we consider a more general energy functional,
involving $E^{(\al)}$ versus just $E^{(0)}$.
More importantly, our perturbative approach gives the precise 
description~\eqref{estimatesThm} of the skyrmion profile for $k \ll 1$, which has   
several advantages. For example, it allows us to compute the precise energy
asymptotics~\eqref{energyThm}
(see also~\cite{komineas2019profile}, where energy asymptotics are found
using formal asymptotic arguements).
Also, in the forthcoming paper \cite{li2020stability}, this information is used
to prove the stability of these skyrmion solutions (against general perturbations), with
precise estimates on the spectral gap. Stability is already obtained by
\cite{li2018stability} in the case $\al = 0$, but by a very different method. They use
the monotonicity of the skyrmion profile, which relies on an ODE argument.
It is unclear how to adapt this argument to the more complicated energy functionals
we consider here, or if it could be used to compute the spectral gap. 

There are two main complications in treating minimizers of $E_{k,\al,\be}$ for small $k$ 
as perturbations of the minimizers $Q(\cdot/s)$ of $E_{0,\al,0} = E_e$. Firstly, the
unperturbed minimizers have slow spatial decay: $Q \not\in L^2$.
Second, the unperturbed problem is scale invariant. As a consequence, 
the linear operator
\[
  H = -\Delta + \frac{1}{r^2} - \frac{8}{(1+r^2)^2}
\]
which appears when linearizing the Euler-Lagrange equation for $E_e$ around $Q$
(see Section~\ref{eqn}) has a resonance at the threshold of its continuous spectrum:
\[
  H h = 0, \qquad h(r) = \frac{d}{ds} Q \left(\frac{r}{s} \right) \big |_{s=1} = 
  \frac{2r}{1 + r^2} \not\in L^2.
\]
The equation for the correction to $Q$ involves the resolvent
$(H + \be^2)^{-1}$, and an important ingredient in our analysis 
is various estimates of this resolvent acting on spaces of functions with slow spatial decay.

In particular, the resolvent becomes singular as $\be \to 0$ due to the resonance.
To prove the first estimate in~\eqref{estimatesThm}, we must  remove this singularity
by carefully adjusting the scaling $s$ in $Q(r/s)$, which results in the relation~\eqref{relationThm}. This kind of
{\it Lyapunov-Schmidt reduction in the presence of a resonance} is similar to the 
analysis in \cite{coles2017solitary} of solitons of perturbations of the energy-critical 
nonlinear Schr\"odinger equation in 3D, though the specific estimates are quite
different for our 2D problem.  

Another key ingredient is to exploit the factorized structure of the linear operator
\[
  H = F^*F, \qquad F = h \partial_r \frac{1}{h} = 
  \partial_r + \frac{r^2-1}{r(r^2+1)}.
\]   
Applying $F$ to the equation for correction $\xi$ to $Q$ produces an equation
for $F \xi$ whose linear operator
\[
  \hat H = F F^* = -\Delta + \frac{4}{r^2(r^2+1)}
\]
has no threshold resonance or eigenvalue, and so good estimates for its
resolvent $(\hat H + \be^2)^{-1}$, without singularity as $\be \to 0$, can be proved.
This is how the second estimate of~\eqref{estimatesThm} is established. 

\subsection{Organization}

The existence of minimizing skyrmion profiles, and their
convergence to the bubble $Q$ as $k \to 0$, are shown in Section~\ref{sec:exist}. 
Much of this is similar to~\cite{melcher2014chiral,doring2017compactness,li2018stability},
but since we have a more general energy functional it is included,
though briefly, with technical details left to appendix 
Sections~\ref{functional}-~\ref{monotoneprofile}.

The main argument establishing the perturbation 
estimates~\eqref{relationThm} and~\eqref{estimatesThm} 
is in Section~\ref{sec:estimates}. This argument relies heavily on various
resolvent estimates which are important, but for readability are left to the appendix
Section~\ref{sec:resolvent}.

Finally, uniqueness is shown in Section~\ref{uniqueness}, 
by a variant of the perturbation estimates in Section~\ref{sec:estimates},
and the energy asymptotics~\eqref{energyThm} are proved
in Section~\ref{sec:energy}. 

{\bf Remarks on notation}:

${\bf \bullet}$ 
notation like $L^p$ and $H^s$ will generally refer to spaces of radial profile 
functions, e.g.,
$\| f \|_{L^p}^p = \int_0^\infty f^p(r) \; r dr$. 
For Lebesgue or Sobolev norms/spaces of functions
on $\R^2$ we use $L^p(\R^2)$ and $H^s(\R^2)$;

${\bf \bullet}$ 
both $( \cdot, \; \cdot )$ and $\langle \cdot, \; \cdot \rangle$ denote the standard 
$L^2$ inner-product: 
$(f, g) = \langle f, g \rangle = \int_0^\infty f(r) g(r) \; r dr$;

${\bf \bullet}$ as usual, $A \lec B$ means $A \leq C B$ for some constant $C$
independent of any relevant parameters.

\section{Existence and Isotropic Limit of Co-rotational Skyrmions}
\label{sec:exist}
\subsection{Minimization problem and function space}

To produce a skyrmion solution, we wish to minimize the energy
$E_{k,\al,\be}(u)$ among profile functions $u(r)$, in an appropriate 
function space, corresponding to topologically non-trivial magnetization configurations. 

To identify the natural function space, first observe that in the 
co-rotational setting, we have 
the following localized version of the elementary topological lower bound~\eqref{topo0}:
for any $0 \leq r_1 < r_2 \leq \infty$,
\begin{equation}  \label{topo}
\begin{split}
  E_e^{[r_1, r_2]}(u) &:= \frac{1}{2} \int_{r_1}^{r_2}
  \left( u_r^2 + \frac{1}{r^2} \sin^2 u \right) r dr 
  =  \int_{r_1}^{r_2}
  \left( \pm \frac{1}{r} (\cos u)_r + \frac{1}{2}(u_r \pm \frac{1}{r} \sin u)^2 \right) r dr \\
  &= \pm(\cos u(r_2) - \cos u(r_1)) + \frac{1}{2} \int_{r_1}^{r_2} 
  (u_r \pm \frac{1}{r} \sin u)^2  r dr \\
  &\geq  \pm (\cos u(r_2) - \cos u(r_1)).
\end{split}
\end{equation}
Note that by one-dimensional Sobolev embedding, if $E_e(u) < \infty$, then
$u \in C((0,\infty))$. Moreover, using~\eqref{topo}, the limits
$\lim_{r \to 0+, \infty} \cos u(r)$, and hence  
$u(0) = \lim_{r \to 0+} u(r)$, $u(\infty) = \lim_{r \to \infty} u(r)$ exist,
and (by finite energy again) are multiples of $\pi$:
\begin{equation} \label{cont}
  E_e(u) < \infty \;\; \implies \;\;
  u \in C([0,\infty]), \quad u(0), u(\infty) \in \pi \Z.
\end{equation} 
Finiteness of $E_z$ further requires $u(\infty) \in 2 \pi \Z$, 
and we may then assume $u(\infty) = 0$
by the shift $u \mapsto u - u(\infty)$ which leaves energy unchanged. 
Further, by $u \mapsto -u$, $k \mapsto -k$, 
is is enough to consider $n \geq 0$. 
This leads to the family of function spaces
\[
  X_n := \{ u : [0,\infty) \to \R \; | \;  
  E_e(u) < \infty, \;\; \lim_{r \to 0+} u(r) = n \pi, \; 
  \lim_{r \to \infty} u(r) = 0 \}
  \subset C([0,\infty]), \qquad n = 0, 1, 2, \ldots
\]
on which the exchange energy $E_e$ is well-defined.
By~\eqref{number2}, the Skyrmion number is
\[
  u \in X_n \;\; \implies \;\; N = \frac{1}{2} (1 - \cos(n \pi) ) = 
  \left\{ \begin{array}{cc} 1 & n \mbox{ odd } \\
  0 & n \mbox{ even }  \end{array} \right. .
\]

For $n=0$ we denote $X := X_0$. 
Since $|\sin u| \leq |u| \leq C |\sin u|$ for $|u| \leq \frac{\pi}{2}$,
\[
  X := X_0 = \{ u : [0,\infty) \to \R \; | \;  u_r \in L^2, \; \frac{u}{r} \in L^2 \},
\] 
is a Banach space with norm
\[
  \| u \|_X^2 := \int_0^\infty \left( u_r^2 + \frac{u^2}{r^2} \right) r dr.
\]
By writing $u^2(r)$ as the integral of its derivative
and using H\"older's inequality, 
we obtain an elementary embedding inequality
\begin{equation}  \label{embedding}
  u \in X \; \implies \; \| u \|_{L^\infty} \leq  \| u \|_X
\end{equation}
which is used frequently below.

It is easy to check that
$X_n = \bar u_n + X_0$ for any fixed $\bar u_n \in X_n$,
and so $X_n$ is an affine Banach space.

To ensure finiteness of $E^{(\al)}$ we add the condition $u \in L^2$:
since $1 - \cos(u) \leq \min( \frac{1}{2} u^2, 2 )$,
\[
  0 \leq E^{(\al)}(u) \leq \int_0^\infty (1 + |\al|) \frac{1}{2} u^2(r) r \; dr = 
  \frac{1}{2} (1 + |\al|) \| u \|_{L^2}^2.
\]
This also makes $E_{DM}$ finite via H\"older's inequality:
\begin{equation}  \label{DMbound}
\begin{split}
  \left| E_{DM}(u) \right| &\leq \int_0^\infty |\sin(u(r))| |u_r| r \; dr \leq
  \left[ \int_0^\infty \sin^2(u(r)) r \; dr \right]^{\frac{1}{2}} 
  \left[ \int_0^\infty u_r^2(r) r \; dr \right]^{\frac{1}{2}} \\
  &\leq
  2\sqrt{E_a(u)} \sqrt{E_e(u)} \leq 2\sqrt{E^{(\al)}(u)} \sqrt{E_e(u)} 
  < \infty.
\end{split}
\end{equation}

So finally our variational problem is:
\begin{equation}
\label{VarProb0}
  e_{k,\al,\be}^{(n)} := \inf \{ \; E_{k,\al,\be}(u) \; | \; u \in X_n \cap L^2 \},
  \qquad k \in \R, \; \al \leq 1, \; \be > 0, \qquad n \in \{0, 1, 2, \ldots \}.
\end{equation}
$X_0 \cap L^2$ is a Banach space, while for $n \not= 0$,
$X_n \cap L^2 = \bar v_n + X_0 \cap L^2$ for any fixed 
$\bar v_n \in X_n \cap L^2$ is an affine Banach space.
A simple variant of~\eqref{embedding} shows that 
functions in $X_n \cap L^2$ (unlike functions in $X_n$) have a
decay rate:
\begin{equation} \label{embedding2}
  u \in X_n \cap L^2 \subset H^1 \; \implies \; 
  u^2(r) \leq \frac{1}{r} \| u \|_{H^1}^2 \leq 
  \frac{1}{r} \left(2 E_e(u) + \| u \|_{L^2}^2 \right).
\end{equation}

A straightforward argument given in Section~\ref{functional}
shows that solving the variational problem~\eqref{VarProb0} 
in $X_n \cap L^2$ produces a solution of the Euler-Lagrange equation:
\begin{prop} \label{differentiability}
$E_{k,\al,\be}$ is a (Fr\'echet) differentiable function on $X_n \cap L^2$.
If $v \in X_n \cap L^2$ solves~\eqref{VarProb0}, then 
$v \in C^\infty((0,\infty))$ satisfies the Euler-Lagrange equation
\begin{equation}  \label{E-L}
  0 = E_{k,\al,\be}'(v) = -\Delta_r v + \frac{1}{2 r^2} \sin(2v) 
  -k \; \be \; \frac{1}{r} \sin^2(v) + \be^2 \; \left( 1 - \al(1-\cos(v)) \right) \sin(v).
\end{equation} 
Moreover, $v(r)$ and $v_r(r)$ decay exponentially as $r \to \infty$,
and the Pohozaev-type relation
\begin{equation}  \label{poho0}
  \be k E_{DM}(v) + 2 \be^2 E^{(\al)}(v) = 0
\end{equation}
is satisfied. Finally, the map $\hat m : \R^2 \to \S^2$ given by
\begin{equation}   \label{mfromv}
  \hat m(x) =  \left[ \; -\sin(v(r)) \sin(\theta), \; \sin(v(r)) \cos(\theta), \; \cos(v(r)) \; \right]
\end{equation}
satisfies the full Euler-Lagrange equation~\eqref{E-L1}.
\end{prop}

Taking $\la = \be$ in the scaling relation~\eqref{law2}, we have
\begin{equation}  \label{minscale}
  E_{k,\al,\be}(u(\be \cdot)) = E_{k,\al,1}(u), \qquad
  e_{k,\al,\be}^{(n)} =  e_{k,\al,1}^{(n)}, 
\end{equation}
and so we may fix $\be = 1$ and consider the minimization problem
\begin{equation}  \label{VarProb}
  e_{k,\al,1}^{(n)} := \inf \{ \; E_{\k,\al,1}(u) \; | \; u \in X_n \cap L^2 \},
  \qquad k \in \R, \; \al \leq 1.
\end{equation}

\subsection{Lower bounds}

To begin the study of~\eqref{VarProb}, 
we investigate if the energy is bounded from below.
First, using~\eqref{DMbound} and Young's inequality yields
\begin{equation} \label{basiclower}
\begin{split}
  E_{k,\al,1}(u) &\geq E_e(u) + E^{(\al)}(u) - 
  \min \left( k^2 E_e(u) + E^{(\al)}(u), \;  E_e(u) + k^2 E^{(\al)}(u) \right) \\
  &= \max \left( (1 - k^2) E_e(u), \; (1-k^2) E^{(\al)}(u) \right).
\end{split}
\end{equation}

Second:
\begin{lem}
If $u \in X_n$ satisfies $u(r) = j \pi$ for some $r \in [0,\infty)$, $j \in \Z$, then
$E_e(u) \geq 2 (|n-j| + |j|)$. Therefore,
\begin{equation} \label{bog2}
  u \in X_n \; \implies \; E_e(u) \geq 2 |n|
\end{equation}  
and
\begin{equation} \label{uniupper}
  u \in X_n \; \implies \; \| u \|_{L^\infty} \leq \left( \frac{1}{2} E_e(u) + 1 \right) \pi.
\end{equation}  
\end{lem}
\begin{proof}
If $r \not= 0$ and $j \not= n$, by continuity of $u$, 
there are non-intersecting sub-intervals  
$(r_m, \; r_{m+1}) \subset [0, \; r]$,  $m = 1, 2, \ldots, |n-j|$
with $|u(r_m) - u(r_{m+1})|  = \pi$. By~\eqref{topo},
$E_e^{[r_m, \; r_{m+1}]} (u) \geq 2$, and so
\[
  E_e^{[0, \; r]} (u) \geq 2|n-j|.
\]
Similarly, if $j \not= 0$, there are non-intersecting sub-intervals  
$(r_m, \; r_{m+1}) \subset [r, \; \infty)$,  $m = 1, 2, \ldots, |j|$,
$r_{|j| + 1} = \infty$, with $|u(r_m) - u(r_{m+1})|  = \pi$. By~\eqref{topo},
$E_e^{[r_m, \; r_{m+1}]} (u) \geq 2$, and so
\[
  E_e^{[r, \; \infty)} (u) \geq 2|j|.
\]
The first statement of the lemma follows from summing the two bounds above.
Then~\eqref{bog2} follows from taking $r=0$ and $j=n$.
For~\eqref{uniupper}: if  
$\| u \|_{L^\infty} > \left( \frac{1}{2} E_e(u) + 1 \right) \pi$, then
by continuity of $u$, for some $r \in [0,\infty)$, $u(r) = j \pi$ where 
$j \in \Z$ with $|j| \leq  \frac{1}{2} E_e(u) + 1 < |j| + 1$. Then by the first statement,
$E_e(u) \geq 2 |j| > E_e(u)$, a contradiction.
\end{proof}

Combining~\eqref{basiclower} and~\eqref{bog2} gives:
\begin{equation}  \label{lb}
  u \in X_n, \; |k| \leq 1 \;\; \implies \;\; e_{k,\al,1}^{(n)} \geq 2|n|(1 - k^2).
\end{equation}

We are not sure if the condition $|k| \leq 1$ is necessary for 
boundedness from below of the energy, but we do have
the following partial converse to~\eqref{lb}, proved in Section~\ref{unbounded}:
\begin{prop} \label{unbound}
\begin{equation}  \label{nolb}
  |k| > 4 - \frac{3}{2} \al \;\; \implies \;\; e_{k,\al,1}^{(n)} = -\infty.
\end{equation} 
\end{prop}

\subsection{Upper bounds}

Here we specialize to the case $n=1$, and show an upper-bound
which ensures compactness of minimizing sequences.
For this, we use the well-known explicit minimizers of 
the exchange energy $E_e$ on $X_1$, which saturate the inequality~\eqref{bog2}:
\[
  E_e(u) = 2, \; u \in X_1 \;\; \implies \;\;
  u(r) = Q(r/s) \; \mbox{ for some } s > 0, \quad
  Q(r) = \pi - 2 \tan^{-1}(r).
\]
This follows from~\eqref{topo}, since equality in 
$E_e(u) \geq 2$ holds if and only if $u_r + \frac{1}{r}\sin(u) = 0$
(almost everywhere) on $(0, \; \infty)$, and
the only solutions to this ode in $X_1$  are the given ones.

\begin{prop} \label{X1}
$e_{k,\al,1}^{(1)} \leq 2$. Moreover,
\begin{equation} \label{nobub} 
  k > 0 \; \implies \; e_{k,\al,1}^{(1)} < 2.
\end{equation}
\end{prop}
\begin{proof}
$Q \not\in L^2$, so to use it as a test function we cut it off: set
\[
  u_R(r) := Q(r) \phi_R(r), \quad 
  \phi_R(r) := \phi(r/R), \quad R > 0,
\]
where $\phi \geq 0$ is a standard smooth cut-off function with 
$\phi(r) = 1$ for $r \leq 1$ and $\phi(r) = 0$ for $r \geq 2$.
Then $u_R \in X_1 \cap L^2$, and as $R \to \infty$ it is easy to check that
\[
  E_e(u_R) = 2 + O(R^{-2}), \qquad
  E^{(\al)}(u_R) = O(\log R), \qquad
  E_{DM}(u_R) = -2 + o(1),
\]
using~\eqref{h3/r}. 
Rescaling, $u^s_R(r) := u_R(r/s)$, $s > 0$, we get
\[
  E_{k,\al.1}(u^s_R) = 2 - 2 k s + O(R^{-2}) + s \; o(1)  + s^2 \; O(\log R), 
\]
so
\[
  \frac{1}{s} \left[ E_{k,\al,1}(u^s_R) - 2 \right] = 
  - 2 k + \frac{1}{s} \; O( R^{-2} ) + o(1) + s \; O(\log R)  .  
\]
Then if we take $R \to \infty$ and $s = 1/R \to 0$, we see
\[
  \frac{1}{s} \left[ E_{\k,\al,1}(u^s_R) -2 \right] \to - 2 k \; , 
\]  
which proves both statements of the proposition.
\end{proof}

\begin{rem}
For $n \geq 2$, 
a test function built of a sequence of $n$ re-scaled, shifted, cut-off 
bubbles $Q$, connecting $u = n \pi$ at $r=0$ to $u = 0$ at $r = \infty$, with 
(small) length scales whose ratios are diverging, shows that   
$e_{k,1,\al}^{(n)} \leq 2 n$.
\end{rem}

\begin{rem}
Note that~\eqref{lb}, \eqref{nolb}, and~\eqref{nobub}
still leave open some questions about the minimal energy
$e_{k,\al,1}^{(n)}$. For example:
\begin{enumerate}
\item 
is $e_{k, \al,1}^{(n)} = -\infty$ for $1 < |k| \leq 4 - \frac{3}{2}\al$?
\item
for $n=1$ and $-1 < k < 0$, is $e_{k,\al,1}^{(1)} = 2$ or $< 2$?
\end{enumerate}
\end{rem}

\subsection{Existence of minimizer}

We now show that for $n=1$ and $0 < k < 1$,
lower bound~\eqref{lb} and upper bound~\eqref{nobub}
imply existence of a minimizer for problem~\eqref{VarProb}:
\begin{thm}  \label{minexist}
Let $n=1$, $k \in (0,1)$, $\al \leq 1$. 
There is $v \in X_1 \cap L^2$ with $E_{k,\al,1}(v) = e_{k,\al,1}^{(1)}$.
\end{thm}
\begin{rem}
Theorem~\ref{minexist}, together with Proposition~\ref{differentiability},
prove the existence part of Theorem~\ref{mainthm}.
\end{rem}

The proof is an elementary concentration-compactness-type argument,
with condition~\eqref{nobub} used to exclude concentration.
This in the spirit of~\cite{melcher2014chiral}, but much 
simpler because of the co-rotational symmetry. For this reason,
we give the details in the appendix, Section~\ref{existence}.

For the case $\al=0$ (no anisotropy), it is shown in~\cite{li2018stability} that
for $k \ll 1$, the minimizing profile $v(r)$ is monotone.
For completeness, we prove
in Section~\ref{monotoneprofile} of the appendix
that monotonicity holds also in the opposite case
$\al = 1$ (no external field):
\begin{prop} \label{monotonicity}
For $\al = 1$ and $0 < k \ll 1$, a minimizing profile $v(r)$ is monotonically
decreasing.
\end{prop}  

By applying the scaling~\eqref{minscale} to the minimizer provided by
Theorem~\ref{minexist}, we also have:
\begin{cor} \label{existcor}
Let $n=1$, $k \in (0,1)$, $\al \leq 1$, $\be > 0$. 
There is $v \in X_1 \cap L^2$ with $E_{k,\al,\be}(v) = e_{k,\al,\be}^{(1)}$.
\end{cor}
By Proposition~\ref{differentiability}, any such minimizer satisfies the Euler-Lagrange equation~\eqref{E-L}, as well as the Pohozaev-type relation~\eqref{poho0}.

\subsection{Convergence to $Q$, after rescaling}

\begin{thm}  \label{convergetoQ}
Consider a family of minimizers $\{ u_k \} \subset X_1$,
$E_{k,\al_k,\be_k}(u_k) = e^{(1)}_{k,\al_k,\be_k}$, 
indexed by  $(0,1) \ni k \to 0$, with $\al_k \leq 1$ and $\be_k > 0$.
Then there are $r_k > 0$ such that $v_k(r) := u_k (r_k \; r)$ satisfies
\begin{equation}  \label{convergence}
  \| v_k - Q \|_X \to 0
\end{equation}
as $k \to 0$. Moreover 
$E_{k, \al_k, \hat \be_k}(v_k) = e_{k, \al_k, \hat \be_k}^{(1)}$
with $\hat \be_k := r_k \be_k$, and $(1 + |\al_k|) \hat \be_k \to 0$.
\end{thm}

\begin{rem}
The existence of such minimizers is ensured by Corollary~\ref{existcor}.
\end{rem}

\begin{proof}
By rescaling $u_k(r) \mapsto u_k(\be_k r)$ we may assume $\be_k = 1$.
By~\eqref{basiclower}, ~\eqref{bog2}, and~\eqref{lb},
\[
  2 > e^{(1)}_{k,\al_k,1} = E_{k,\al_k,1}(u_k) 
  \geq (1-k^2) E_e(u_k) \geq 2 (1 - k^2) \to 2,
\]
which together with~\eqref{poho0} shows
\begin{equation}  \label{eto2}
  E_{k,\al_k,1}(u_k) 
  = e_{k,\al_k,1}^{(1)} \to 2, \qquad
  E_e(u_k) \to 2, \qquad
  E^{(\al_k)}(u_k) \to 0. 
\end{equation}

Since $u_k \in X_1$ is continuous, there is $r_k > 0$ such that 
$u_k(r_k) = \frac{\pi}{2}$. Set 
\[
  v_k(r) := u_k(r_k r), 
\]
so
\[
  v_k(1) = \frac{\pi}{2}, \;\; E_e(v_k) \to 2, \quad r_k^2 E^{(\al_k)}(v_k)  \to 0.
\] 

The proof of~\eqref{convergence} is a standard variational argument, together with
the use of~\eqref{topo}.
If~\eqref{convergence} fails, then along some subsequence
$v_j := v_{k_j}$, $k_j \to 0$, 
we have $\| v_j - Q \|_X \geq \e$ for some $\e > 0$.
We will contradict this.
Since $\| (v_j)_r \|_{L^2}^2 \leq 2 E_e(v_j) \lec 1$, 
and using the $1D$ (compact) Sobolev embedding
$H^1_{dr}([r_1, r_2]) \subset C([r_1, r_2])$ for any $0 < r_1 < r_2 < \infty$,
there is a further subsequence (still denoted $v_j$) converging uniformly 
on compact subintervals of $(0,\infty)$ to a continuous function $v(r)$, with
$v(1) = \frac{\pi}{2}$. Moreover, 
$(v_j)_r$ converges weakly in $L^2_{r dr}(0,\infty)$ to $v_r$.
By weak lower-semi-continuity of the $L^2$ norm, and Fatou's lemma
\begin{equation}  \label{enbound}
  E_e(v) \leq \liminf E(v_j) = 2.
\end{equation}
By~\eqref{cont}, $v \in C([0,\infty])$ with $v(0)$, $v(\infty) \in \pi \Z$.

We next show $v(0) = \pi$, $v(\infty) = 0$, and $v = Q$.
Apply~\eqref{topo} to $v_j$ on each of the intervals $[0, 1]$ and $[1, \infty]$, 
using~\eqref{enbound}, to conclude.
\begin{equation} \label{partition}
  1 \leq E_e^{[0,1]}(v_j) \leq 1 + o(1), \qquad 
  1 \leq E_e^{[1,\infty)}(v_j) \leq 1 + o(1).
\end{equation}
Then by weak lower-semi-continuity and Fatou,
\begin{equation}  \label{enbound2}
  E_e^{[0,1]}(v) \leq \liminf E^{[0,1]}_e(v_j) = 1, \qquad
  E_e^{[1,\infty)}(v) \leq \liminf E^{[1,\infty)}_e(v_j) = 1.
\end{equation}
Applying~\eqref{topo} for $v_j$ on $[0, r]$ and $[r, 1]$ for $r \leq 1$, using~\eqref{partition}, gives
\begin{equation}  \label{bound1}
  \frac{\pi}{2} - o(1) \leq v_j(r) \leq \pi + o(1), \quad 0 \leq r \leq 1,
\end{equation}
which implies
\[
  \frac{\pi}{2} \leq v(r) \leq \pi, \quad 0 \leq r \leq 1,
\]
and in particular $v(0) = \pi$.
Similarly, applying~\eqref{topo} on $[1, r]$ and $[r, \infty)$ for $r \geq 1$ gives
\begin{equation} \label{bound2}
  -o(1) \leq v_j(r) \leq \frac{\pi}{2} + o(1), \quad 1 \leq r < \infty,
\end{equation}
which implies
\[
  0 \leq v(r) \leq \frac{\pi}{2}, \quad 1 \leq r < \infty,
\]
and in particular $v(\infty) = 0$. 
Then~\eqref{topo} on $[0,\infty)$, with~\eqref{enbound} forces
$v_r + \frac{1}{r} \sin v = 0$ a.e. in $(0,\infty)$.
The unique solution of this ODE with $v(1) = \frac{\pi}{2}$ is $v = Q$.

So
\[
  w_j := v_j - Q \to 0 \; \mbox{ uniformly on compact subintervals of } \; (0,\infty),
  \; \mbox{ and } \; (w_j)_r \to 0 \; \mbox{ weakly in } L^2.
 \]
It remains to show the convergence is strong in $X$.
This will be a consequence of convergence of the energy: 
$2 = E_e(Q) = \lim E_e(v_j)$. The energy relation
\[
  v_j = Q + w_j \quad \implies \quad
  E_e(v_j) = E_e(Q) + E_e(w_j) 
  + \int_0^\infty \left( Q_r (w_j)_r +
  \frac{1}{r^2} \cos (Q + w_j) \sin Q \sin w_j \right) r dr   
\]
is an elementary consequence of trig identities.
Letting $j \to \infty$ here, using the weak convergence of $(w_j)_r$, the 
local uniform convergence of $w_j$, and the facts that
$Q_r \in L^2_{r dr}$, $\frac{1}{r^2} \sin Q \in L^1_{r dr}$, we find
\[
  \lim_{j \to \infty} E_e(w_j) = 0.
\]
Now by~\eqref{bound1} and~\eqref{bound2},
\[
  -\frac{\pi}{2} - o(1) \leq w_j(r) \leq \frac{\pi}{2} + o(1),
\]
we have $w_j^2 \lec \sin^2 w_j$ and so 
\[
  \| w_j \|_X^2 \lec E_e(w_j) \to 0,
\]
which completes the contradiction argument.

By~\eqref{minscale}, $v_k \in X_1$ minimizes $E_{k,\al_k,\hat \be_k}$ with
$\hat \be_k := r_k$. Finally,
\[
  E^{(\al)}(u) \geq
  \int_0^\infty \left( (1 - \cos u) + \al_- \frac{1}{2} (1 - \cos u)^2 \right) r dr 
  \geq \frac{1}{2}(1 + \al_-) \int_0^\infty (1 - \cos u)^2 r dr,
\]
where $\al_- := \max(0,-\al)$, so  by~\eqref{eto2},
\[
  0 \leftarrow \hat\be_k^2 E^{(\al_k)}(v_k) \geq
  \hat \be_k^2 (1 + (\al_k)_-)
  \frac{1}{2} \int_0^\infty (1 - \cos v_k)^2 \; r dr .
\]
By Fatou and~\eqref{h^2/r^2},
\[
   \liminf\limits_{k \to 0} \int_0^\infty (1 - \cos(v_k))^2 r dr \geq
   \int_0^\infty (1 - \cos(Q))^2 r dr = 2 > 0,
\]
so $\hat \be_k^2(1 + (\al_k)_-) \to 0$, and 
since $\al_k \leq 1$, $\hat \be_k^2 (1 + |\al_k|) \to 0$.
\end{proof}

\section{Convergence Rate Estimates }  \label{sec:estimates}

Our main result says that by a more refined re-scaling 
of a sequence of minimizers as in Theorem~\ref{convergetoQ},
we can sharpen the convergence~\eqref{convergence} 
to give quantitative estimates:
\begin{thm}  \label{estimates}   
Fix $A \geq 0$. Let $\{ v_k \} \subset X_1$ be a family of minimizers
$E_{k_j,\al_k,\hat\be_k}(u_k) = e^{(1)}_{k,\al_k,\hat\be_k}$, 
indexed by $(0,1) \ni k \to 0$, with
$-A \leq \al_k \leq 1$, $0 < \hat \be_k \to 0$, 
such that $\| v_k - Q \|_X \to 0$ as $k \to 0$.
Then for all $k$ sufficiently small, there is $\mu_k = 1 + o(1)$, with
\begin{equation}  \label{relation}
  k = \be_k \left( 2 \log \left( \frac{1}{\be_k} \right) + O(1) \right),
  \qquad \be_k := \mu_k \hat \be_k \to 0,
\end{equation}
such that if we write
\[
  v_k(\mu_k r) = Q(r) + \xi_k(r),
\]
then
\begin{equation}  \label{xiest2}
  \| \xi_k \|_X \lec \be_k,
\end{equation}
and
\begin{equation}  \label{deriv}
  \| \xi_k \|_{X^\infty} \lec \be_k^2 \log^2 \left( \frac{1}{\be_k} \right).
\end{equation}
\end{thm}
\begin{rem}
The existence of such minimizers is ensured by Corollary~\ref{existcor}
and Theorem~\ref{convergetoQ}.
\end{rem}
\begin{rem}
Theorem~\ref{estimates} proves~\eqref{relationThm} and~\eqref{estimatesThm}
of Theorem~\ref{mainthm}.
\end{rem}

\begin{proof}
The proof occupies the next several subsections.
It is based on a careful analysis of the Euler-Lagrange equation~\eqref{E-L}
with $\be \ll 1$ and $k \ll 1$.
Having fixed a family $v_k$, $\al_k$, $\hat \be_k$ as in the statement 
of the theorem, we will drop the subscript $k$ from the notation for simplicity.

\subsection{Equation for the difference}  \label{eqn}

We write
\begin{equation}  \label{o(1)}
  v(r) = Q(r) + \hat \xi(r), \qquad \| \hat \xi \|_X = o(1),
\end{equation}
and reorganize the Euler-Lagrange equation~\eqref{E-L} as
an equation for~$\hat \xi$:
\begin{equation}  \label{pde}
  (H + \hat\be^2) \; \hat \xi  = s + N,
\end{equation}
where we denote
\[
  h := \sin(Q) = \frac{2r}{r^2+1}, \quad
  \hat{h} := \cos(Q) = \frac{r^2-1}{r^2+1}
\]
\[  
  H = -\Delta_r + \frac{1}{r^2} - W(r), \qquad W = 2 \frac{h^2}{r^2},
\]
\begin{equation}  \label{source}
  s := s_{k,\al,\hat\be}(r) 
  := k \hat\be \frac{1}{r} h^2 + \hat\be^2 (\al \frac{h}{r} - 1) h
  = -\hat\be^2 h + (k + \hat\be \al) \hat\be \frac{h^2}{r}
\end{equation}
and
\begin{equation}  \label{Nform}
\begin{split}
  N = N_{k,\al,\hat\be}(\hat \xi) &= \frac{1}{2r^2} \left(  2 \hat{h} h (1-\cos(2 \hat\xi)) 
  + (2 h^2-1)(\sin(2 \hat\xi) - 2 \hat\xi) \right) \\
  & \quad + k \hat\be \frac{1}{r}(\sin^2(Q+\hat\xi) - \sin^2(Q)) \\ 
  & \quad + \hat\be^2 \left(  (1-\al + \al \cos(Q)) \sin(Q) + \hat\xi  
  - (1-\al + \al \cos(Q + \hat\xi))\sin(Q+\hat\xi) \right)
\end{split}
\end{equation}
so that
\begin{equation}  \label{Nbound}
  |N| \lec \frac{1}{r^2} ( h \hat\xi^2 + |\hat\xi|^3 )
  + k \hat\be \frac{1}{r}( h |\hat\xi| + \hat\xi^2) + 
  \hat\be^2 \left( (1 + |\al|) (\frac{h}{r} |\hat\xi| + |\hat\xi|^3 + h \hat\xi^2) + |\al| h \hat\xi^4  \right).
\end{equation}

\subsection{Resolvent estimates}

Re-writing equation~\eqref{pde} as
\begin{equation} \label{pde2} 
  \hat \xi = (H + \hat\be^2)^{-1} \left( s + N \right),
  \qquad s = s_{k,\al,\hat\be} = -\hat\be^2 h + (k + \hat\be \al) \hat\be \frac{h^2}{r},
  \quad N = N_{k,\al,\hat\be}(\hat \xi)
\end{equation}
it is clear we need to understand the 
mapping properties of the resolvent $(H + \be^2)^{-1}$ as $\be \to 0$,
which are dominated by the presence of a threshold resonance for the 
$\be=0$ limit $H$:
\[
  H h = 0, \qquad h \in L^p, \; p > 2, \quad h \not\in L^2.
\] 
This means we need an orthogonality condition to avoid singular behaviour 
of $(H + \be^2)^{-1}$ as $\be \to 0$. With such a condition, we find
$(H + \be^2)^{-1}$  satisfies the same estimates as the free resolvent
\[
  R_0(\be) := \left( -\Delta_r + \frac{1}{r^2} + \be^2 \right)^{-1}.
\] 
Precisely,
\begin{equation}  \label{basic}
  g \perp R_0(\be) Wh \;\; \implies \;\;
  \| (H + \be^2)^{-1} g \|_X \lec \| R_0(\be) g \|_{X}.
\end{equation}
This estimate is a consequence of the factorization
\begin{equation}  \label{factor1}
  (H + \be^2)^{-1} = ( I - R_0(\be) W )^{-1} R_0(\be)
\end{equation}
and the estimate
\begin{equation}  \label{basic2}
  f \perp Wh \;\; \implies \;\;
  \| (I - R_0(\be) W)^{-1} f \|_X \lec \| f \|_{X},
\end{equation}
which is proved in Section~\ref{fullresolvent}.
The bound~\eqref{basic} will be used together with 
estimates on the free resolvent which are proved in Section~\ref{freeresolvent}:
\begin{equation}  \label{p=1}
  \| R_0(\be) g \|_X \lec \| r g \|_{L^2},
\end{equation}
\begin{equation}  \label{p=0}
  \| R_0(\be) g \|_X \lec \frac{1}{\be} \| g \|_{L^2},
\end{equation}
\begin{equation}  \label{p=-1}
  \| R_0(\be) g \|_X \lec \frac{1}{\be^2} \| g \|_X.
\end{equation}
In addition, we use a refinement of~\eqref{p=0} for the case $g = h \not\in L^2$:
 \begin{equation}  \label{refined}
  \| R_0(\be) h \|_X \lec \frac{1}{\be}.
\end{equation}
This is proved in Section~\ref{refinement}.

\subsection{Reparameterization}

To apply estimate~\eqref{basic} to equation~\eqref{pde2}
would require the orthogonality condition
\[
  0 = \langle s_{k,\al,\hat\be} + N_{k,\al,\hat\be}(\hat\xi), R_0(\hat\be) W h \rangle,
\]
which does not hold a priori. We first re-express this condition:
\begin{lem}  \label{orthequiv}
If $\hat\xi$ satisfies~\eqref{pde2}, then
\begin{equation}  \label{orth2}
\langle s + N, R_0(\hat\be) W h \rangle = 0 \;\; \iff \;\;
\langle W \hat\xi, R_0(\hat\be) h \rangle = 0. 
\end{equation}
\end{lem}
\begin{proof}
This is a computation using the equation~\eqref{pde2} for $\hat\xi$, the equation for the resonance eigenfunction
\[
  0 = H h = (-\Delta + \frac{1}{r^2} - W) h \;\; \implies \;\; R_0(0) (W h) = h, 
\] 
and the resolvent identity 
\begin{equation}  \label{resid}
  R_0(\gamma) - R_0(\be) = (\be^2 - \gamma^2) R_0(\be) R_0(\gamma), 
  \qquad \be, \gamma \geq 0
\end{equation}
with $\gamma = 0$:
\[
\begin{split}
  \langle s + N, R_0(\hat\be) W h \rangle &=
  \langle (H + \hat\be^2)\hat\xi, R_0(\hat\be) W h \rangle =
  \langle \hat\xi, (-\Delta + \frac{1}{r^2} + \hat\be^2  - W) R_0(\hat\be) W h \rangle \\ &=
  \langle \hat\xi, W h - W R_0(0) W h + W(R_0(0) - R_0(\hat\be)) Wh \rangle \\ &=
  \hat\be^2 \langle \hat\xi, W R_0(\hat\be) R_0(0) Wh \rangle =
  \hat\be^2 \langle W \hat\xi, R_0(\hat\be) h \rangle.
\end{split}
\]
\end{proof}
To impose this condition we will reparameterize:
\[
  v(r) = Q(r) + \hat\xi(r) = Q(r/\mu) + \xi(r/\mu),
  \qquad \mu = 1 + o(1).
\]
It follows that the new perturbation $\xi$
solves the Euler-Lagrange equation~\eqref{pde2} with rescaled parameter:
\begin{equation}  \label{pde3}
  (H + \be^2) \xi = s_{k,\al,\be}
  + N_{k,\al,\be}(\xi), \qquad
  \be := \mu \hat \be \to 0.
\end{equation}
The idea is to choose $\mu$ so as to enforce
the orthogonality condition corresponding to~\eqref{orth2}:
\begin{equation} \label{neworth2}
  \langle W \xi, R_0(\be) h \rangle = 0.
\end{equation}
\begin{prop} \label{reparam}
For all $k$ sufficiently small, 
there exists $\mu = 1 + o(1)$ so that~\eqref{neworth2} holds.
\end{prop}
\begin{proof}
We have
\[
  \xi(r) = v( \mu r) - Q(r) = Q(\mu r) - Q(r) + \hat\xi(\mu r).
\]
Now since
\[
  \frac{\p}{\p \mu} Q(\mu r) = r Q'(\mu r) = - \frac{1}{\mu} h(\mu r), \qquad
  \frac{\p^2}{\p \mu^2} Q(\mu r) = \frac{1}{\mu} r h^2(\mu r),
\]
by Taylor's theorem we have
\[
  Q(\mu r) - Q(r) = (1-\mu) h(r) + (\mu-1)^2 \zeta_\mu(r), \qquad 
  \zeta_\mu(r) = \frac{1}{2\mu^*} r h^2(\mu^* r), \quad \mu^* = \mu^*(r) \in (1, \mu),
\]
and so we express condition~\eqref{neworth2} as
\begin{equation} \label{neworth3}
  0 = (1 - \mu) \langle W h, R_0(\be) h \rangle +
  (\mu-1)^2 \langle W \zeta_\mu, R_0(\be) h \rangle +
  \langle W \hat\xi(\mu \cdot), R_0(\be) h \rangle.
\end{equation}
To estimate these terms, we use a key estimate
of inner-products involving the free resolvent acting on the slowly
decaying function $h$, proved in Section~\ref{freeinners}:
\begin{lem}  \label{key}
Let $g \in L^1$ be a (radial) function: 
\begin{enumerate}
\item 
if $r^q g \in L^\infty$ for some $q > 3$, then as $\be \to 0+$,
\begin{equation} \label{choice}
  \left|  \langle g , R_0(\be) h \rangle -   \left( \int_0^\infty r g(r) \; r dr \right) 
  \log \left( \frac{1}{\be} \right) \right|
  \lec \| g \|_{L^1 \cap r^{-q} L^\infty},
\end{equation}
and in particular
\begin{equation}  \label{innerbound}
  | \langle g , R_0(\be) h \rangle | \lec 
  \log \left( \frac{1}{\be} \right)\| g \|_{L^1 \cap r^{-q} L^\infty}  , 
\end{equation}
and
\begin{equation}  \label{innercomp}
  \langle W h , R_0(\be) h \rangle = 4 \log \left( \frac{1}{\be} \right)
  + O(1);
\end{equation}
\item
if $r ^3 g \in L^\infty$, then
\begin{equation}  \label{innerbound2}
  | \langle g , R_0(\be) h \rangle | \lec 
  \log^2 \left( \frac{1}{\be} \right) \| g \|_{L^1 \cap r^{-3} L^\infty}  . 
\end{equation}
\end{enumerate}
\end{lem}
We apply~\eqref{innercomp},
\eqref{innerbound} for $q=4$ with $g = W \zeta_\mu$, 
and with $g = W \xi(\mu \cdot)$ to~\eqref{neworth3}, noting that
\[
  \| W \zeta_\mu \|_{L^1 \cap r^{-4} L^\infty} \leq \| W \|_{L^1 \cap r^{-4} L^\infty}
  \| \zeta_\mu \|_{L^\infty}  \lec \frac{1}{(\mu^*)^2} \lec 1
\]  
for $\mu$ near $1$, and
\[
  \| W \hat\xi(\mu \cdot) \|_{L^1 \cap r^{-4} L^\infty}  \lec 
  \| W \|_{L^1 \cap r^{-4} L^\infty} \| \hat\xi (\mu \cdot) \|_{L^\infty} 
  \lec \| \hat\xi \|_{L^\infty} \lec \| \hat\xi \|_X = o(1)
\]
by~\eqref{o(1)}.  
After dividing through by $\log \left( \frac{1}{\hat \be} \right)$,
this yields
\[
  (4 + o(1))(\mu - 1) + O((\mu-1)^2) + o(1) = 0 .
\]
By the intermediate value theorem, there is a solution
$\mu = 1 + o(1)$ of this equation for all $k$ small enough.
\end{proof}

\subsection{Remainder estimate}
\label{remainder}

By Proposition~\ref{reparam} and Lemma~\ref{orthequiv}
(for $k$ small enough), the rescaled remainder
\[
  \xi(r) = v(\mu r) - Q(r), \quad \| \xi \|_X \to 0 
\]
satisfies equation~\eqref{pde3} and the orthogonality condition
\begin{equation}  \label{orth1}
  \langle s_{k,\al,\be} + N_{k,\al,\be}(\xi), R_0(\be) W h \rangle = 0.
\end{equation}
So using~\eqref{basic}, we have
\begin{equation}  \label{xiest}
  \| \xi \|_X \lec \| R_0(\be) \left( s_{k,\al,\be} + N_{k,\al,\be}(\xi) \right) \|_X.
\end{equation}
We will use this bound to show estimate~\eqref{xiest2} for  $\| \xi \|_X$.

We begin with the contributions from 
\begin{equation}  \label{sourcenew}
  s_{k,\al,\be} = -\be^2 h + (k + \be \al) \be \frac{h^2}{r}.
\end{equation}
We use~\eqref{refined} for the main term $-\be^2 h$,  
and~\eqref{p=1} for the second term, since $r \frac{h^2}{r} = h^2 \in L^2$, to get
\begin{equation}  \label{sest}
  \| R_0(\be) s_{k,\al,\be} \|_X \lec  \be + (k + \be |\al|) \be \lec \be.
\end{equation}

Next we will bound the contributions coming from the nonlinear terms
$N_{k,\al,\be}(\xi)$.
From the pointwise estimate~\eqref{Nbound}, and recalling that 
$\| \xi \|_{L^\infty} \lec \| \xi \|_X \ll 1$, we have
\begin{equation}  \label{Nsplit}
  N_{k,\al,\be}(\xi) = N_1 + N_2 + N_3,
\end{equation}
where
\begin{equation}  \label{Nbounds}
\begin{split}
  & N_1 \lec h \frac{\xi^2}{r^2} + \frac{|\xi|^3}{r^2} 
  + \left( k + \be(1 + |\al|) \right) \be h \frac{|\xi|}{r}  \\
  & N_2 \lec  k \be \frac{\xi^2}{r} + \be^2(1 + |\al|) h \xi^2 \\
  & N_3 \lec \be^2(1 + |\al|) |\xi|^3.
\end{split}
\end{equation}
Using~\eqref{p=1}:
\begin{equation} \label{N1}
\begin{split}
  \| R_0(\be) N_1 \|_X  &\lec \| r \; N_1 \|_{L^2} 
  \lec \| r h \|_{L^\infty} \left\| \frac{\xi}{r} \right\|_{L^2}^2 + 
  \| \xi \|_{L^\infty}^2 \left\| \frac{\xi}{r} \right\|_{L^2} 
  +  \left( k + \be(1 + |\al|) \right) \be \| r h \|_{L^\infty} \left\| \frac{\xi}{r} \right\|_{L^2} \\
  &\lec \left( \| \xi \|_X +  (k + \be(1 + |\al|))\be \right) \| \xi \|_X 
  \lec \left( \| \xi \|_X + \be \right) \| \xi \|_X = o(1) \| \xi \|_X.
\end{split}
\end{equation}
Using~\eqref{p=0}:
\begin{equation}  \label{N2}
\begin{split}
  \| R_0(\be) N_2 \|_X  &\lec \frac{1}{\be} \| N_2 \|_{L^2} \lec \frac{1}{\be} \left(
  k \be \| \xi \|_{L^\infty} \| \frac{\xi}{r} \|_{L^2} + 
  \be^2(1 + |\al|) \| r h \|_{L^2} \| \xi \|_{L^\infty} \| \frac{\xi}{r} \|_{L^2}  \right) \\
  &\lec \left( k + \be(1 + |\al|) \right) \| \xi \|_X^2 
  \ll \| \xi \|_X^2 = o(1) \| \xi \|_X.
\end{split}
\end{equation}
Using~\eqref{p=-1}: 
\begin{equation}   \label{N3}
  \| R_0(\be) N_3 \|_X  \lec \frac{1}{\be^2} \| N_3 \|_{X} 
  \lec (1 + |\al|) \| \xi \|_{L^\infty}^2 \| \xi \|_X \lec \| \xi \|_X^3 = o(1) \| \xi \|_X,
\end{equation}
since $|\al| \lec 1 + A \lec 1$.

Combining~\eqref{N1}-\eqref{N3} with~\eqref{Nsplit},~\eqref{sest}, 
and~\eqref{xiest}, we get
\[
  \| \xi \|_X \lec \be + o(1) \| \xi \|_X,
\]
and~\eqref{xiest2} of Theorem~\ref{estimates} follows
for $k$ sufficiently small.

\subsection{Parameter relation}

Here we will use~\eqref{orth1} to determine the leading-order
relationship~\eqref{relation} between $k$ and $\be$. 

Begin with the contributions from the source term $s_{k,\al,\be}$.
By~\eqref{innercomp} of Lemma~\ref{key}, we have
\begin{equation}  \label{source1}
  \langle h, R_0(\be) W h \rangle =
  4 \log \left( \frac{1}{\be} \right) + O(1).
\end{equation}
Using the resolvent identity~\eqref{resid} with $\gamma = 0$, and
$R_0(0) (Wh) = h$, we have
\[
  R_0(\be) W h = h - \be^2 R_0(\be) h, 
\]
so
\begin{equation} \label{source2}
  \langle \frac{h^2}{r}, \; R_0(\be) W h \rangle =
  \langle \frac{h^2}{r}, \; h \rangle  - \be^2 
  \langle \frac{h^2}{r}, \; R_0(\be) h \rangle =
  2 + O \left( \be^2 \log^2 \left( \frac{1}{\be} \right) \right) 
\end{equation}
by using~\eqref{h3/r}, as well as~\eqref{innerbound2} of Lemma~\ref{key} 
with $g = \frac{h^2}{r} \in L^1 \cap r^{-3} L^\infty$.
Combining~\eqref{source1} and~\eqref{source2} yields
\begin{equation}  \label{source3}
  \langle s_{k,\al,\be}, \; R_0(\be) W h \rangle =
  -4 \be^2 \log \left( \frac{1}{\be} \right) 
  + 2 k \be + O(\be^2).
\end{equation}

To bound the contributions of the nonlinear terms we will use 
the decomposition~\eqref{Nsplit}, the basic estimate 
\[
  \left| \langle N_j, \; R_0(\be) W h \rangle \right| =
  \left| \langle R_0(\be) N_j, \; W h \rangle \right| 
  \lec \| R_0(\be) N_j \|_{L^\infty} \| W h \|_{L^1} \lec
  \| R_0(\be) N_j \|_{X},
\]
the estimates~\eqref{N1}-\eqref{N3}, and~\eqref{xiest2}, to conclude:
\begin{equation} \label{Ncont}
  \left| \langle N_{k,\al,\be}(\xi) \; R_0(\be) W h \rangle \right| \lec \be^2. 
\end{equation}

Inserting~\eqref{source3} and~\eqref{Ncont} into~\eqref{orth1} 
shows~\eqref{relation} of Theorem~\ref{estimates},
for $k$ sufficiently small.

\subsection{Derivative estimates}
\label{regests}

In this section, we complete the proof of Theorem~\ref{estimates}
by establishing the higher-regularity estimate~\eqref{deriv}.

First, we show:
\begin{lem} 
$\frac{\xi}{r} \in L^\infty$, with
\begin{equation}  \label{xi/r}
  \left\| \frac{\xi}{r} \right\|_{L^\infty} \lec \be.
\end{equation}
\end{lem}
\begin{proof}
Define
\[
  \bar\xi(r) := \frac{1}{\be} \xi(r), \qquad
  M(\delta) := \sup_{\delta < r < 1} \frac{|\bar \xi(r)|}{r}  , \qquad 0 < \delta < 1.
\]
By~\eqref{xiest2}, $\| \bar \xi \|_{L^\infty} \lec \| \bar \xi \|_X \lec 1$,
so we have the crude bound
\begin{equation}  \label{crude}
  M(\delta) \lec \frac{1}{\delta},
\end{equation}
and it suffices to show $M(\delta) \lec 1$ uniformly for $\delta < 1$.
Re-write the Euler-Lagrange equation~\eqref{pde3} as
\begin{equation} \label{ode2}
  -\bar \xi_{rr}  - \frac{1}{r} \bar \xi_r + \frac{1}{r^2} \bar \xi  = g, \qquad  
  g := \left( \frac{2h^2}{r^2} - \be^2 \right) \bar \xi + \frac{1}{\be} (s + N).
\end{equation}
We will single out one term from the nonlinear terms $N$ which needs to 
be treated carefully, and write
\[
  g = g_1 + g_2, \qquad |g_2| \lec \be^2 \frac{|\bar \xi|^3}{r^2},
\]
where using~\eqref{Nbounds} and~\eqref{sourcenew} it is easily checked that
\[
  \| \bar \xi \|_{X} \lec 1 \; \implies \; \left\| \frac{g_1}{r} \right\|_{L^1_{\leq 1}} \lec 1.
\]

We claim the representation formula
\begin{equation}  \label{rep}
  2 \left( \frac{\bar \xi(r)}{r} - \bar \xi(1) \right) = 
  \frac{1}{r^2} \int_0^r y g(y) \; y dy + \int_r^1 \frac{1}{y} g(y) \; y dy
  - \int_0^1 y g(y) \; y dy
\end{equation}
holds.
Indeed, since $r$ and $1/r$ are fundamental solutions of~\eqref{ode2} with $g=0$,
by the variation of parameters formula and ODE uniqueness,
\begin{equation}  \label{rep2}
  2 \frac{\bar \xi(r)}{r} = \frac{1}{r^2} \left( a - \int_r^1 y g(y) \; y dy \right)
  + b + \int_r^1 \frac{1}{y} g(y) \; y dy
\end{equation}
for some $a, b \in \R$. Now for $r \leq 1$,
\begin{equation}  \label{M1}
\begin{split}
  \left| \int_r^1 \frac{1}{y} g(y) \; y dy \right| &\lec 
  \left\| \frac{g_1}{r} \right \|_{L^1_{\leq 1}} + 
  \be^2 \int_r^1 \frac{1}{y} \frac{|\bar \xi|^3}{y^2} \; y dy \\
  & \lec 1 +  \be^2 \| \bar \xi \|_X^2 M(r) \lec 1 + \be^2 M(r) \lec \frac{1}{r}
\end{split}
\end{equation}
using~\eqref{crude},
and so since $\bar \xi \in L^\infty$ we must have
\[
  a = \int_0^1 y g(y) \; y dy, \quad \mbox{ and so } \quad
  2 \frac{\bar \xi(r)}{r} = \frac{1}{r^2} \int_0^r y g(y) \; y dy
  + b + \int_r^1 \frac{1}{y} g(y) \; y dy,
\]
and then the formula~\eqref{rep} follows from plugging in $r=1$ to find $b$.

Now using~\eqref{M1} again, as well as
\[
\begin{split}
  \left| \int_0^r y g(y) \; y dy \right| &\lec 
  r^2 \left \| \frac{g_1}{r} \right \|_{L^1_{\leq 1}} + 
  \be^2 \int_0^r y \frac{|\bar \xi(y)|^3}{y^2} \; y dy \\
  &\lec r^2 + \be^2 \int_0^r |\bar \xi(y)| dy,
\end{split}
\]
in~\eqref{rep}, we arrive at
\[
  \frac{|\bar \xi(r)|}{r} \lec 1 + \be^2 M(r) +
  \frac{\be^2}{r^2} \int_0^r |\bar \xi(y)| dy
  \lec 1 + \be^2 M(\delta) + \frac{\be^2}{r^2} \int_0^r |\bar \xi(y)| dy
\]
if $\delta \leq r \leq 1$.
Taking supremum over such $r$ gives
\[
  M(\delta) \lec 1 + \be^2 M(\delta) + 
  \be^2 \sup_{\delta \leq r \leq 1} \frac{1}{r^2} \int_0^r |\bar \xi(y)| dy,
\]
and so, since $\be \ll 1$,
\begin{equation}  \label{start}
  \frac{|\bar \xi(\delta)|}{\delta} \leq M(\delta) \leq C \left( 1 + 
  \be^2 \sup_{\delta \leq r \leq 1} \frac{1}{r^2} \int_0^r \frac{|\bar \xi(y)|}{y}
  \; y dy \right)
\end{equation}
for some constant $C$.
To finish the argument, we will iterate this relation, beginning with the
crude bound~\eqref{crude}, i.e. 
\begin{equation}  \label{avest}
   \frac{|\bar \xi(r)|}{r} \leq \frac{C}{r} \; \implies \;
   \sup_{\delta \leq r \leq 1} \frac{1}{r^2} \int_0^r \frac{|\bar \xi(y)|}{y} \; y dy \leq
   C \sup_{\delta \leq r \leq 1} \frac{1}{r^2} \int_0^r dy = \frac{C}{\delta},
\end{equation}
so by~\eqref{start},
\[
  \frac{|\bar \xi(\delta)|}{\delta} \leq C \left( 1 + 
  \be^2 \frac{C}{\delta} \right).
\]
We can use this improved estimate in place of the crude one in~\eqref{avest},
and return again to~\eqref{start} to even further improve the estimate.
Iterating in this way $k$ times yields
\[
   \frac{|\bar \xi(\delta)|}{\delta} \leq C \left( \sum_{j=0}^{k-1} (C \be^2)^j 
   + \frac {(C \be^2)^k}{\delta}  \right)
   \lec 1 + \frac {(C \be^2)^k}{\delta} ,
\]
and the Lemma is proved by taking $k \to \infty$. 
\end{proof}

The idea behind the proof of estimate~\eqref{deriv} 
is to exploit the factorized structure of the linearized operator
$H$ appearing in the Euler-Lagrange equation~\eqref{pde3} for $\xi$,
\[
  H = F^* F, \qquad F = \p_r + \frac{1}{r} \hat h = h \p_r \frac{1}{h}, 
  \qquad F^* = -\partial_r + \frac{1}{r}(\hat{h } - 1),
\]
by applying the first-order factor $F$ to~\eqref{pde3}, to obtain an equation for $F \xi$:
\begin{equation}  \label{pde'}
  (\hat H + \be^2) \eta = F s + F N, \qquad
  \eta := F \xi,
\end{equation}
where
\[
  \hat H = F F^* 
  = -\Delta_{r}  + \hat{V}, \qquad  \hat{V}:= \frac{4}{r^2(r^2 + 1)}.
\]
Equation~\eqref{pde'} is better than equation~\eqref{pde3} in two ways.
First, since $F h = 0$, the term $-\be^2 h$ in $s_{k,\al,\be}$ which
was the main term in the bound $\| \xi \|_X \lec \be$, is completely absent from the
source term $F s_{k,\al,\be}$ in~\eqref{pde'}. More precisely, 
\[
  F h = 0, \quad F \frac{h^2}{r} = -\frac{h^2}{r^2}(1 + \hat h) \;\; \implies \;\;
  F s = -(k + \be \al) \be \frac{h^2}{r^2}(1 + \hat h) ,
\]
which is well-localized. In particular, defining the space $L^{1,\log}$ by the norm
\[
  \| f \|_{L^{1,\log}} := \int_0^\infty \log(2 + r) |f(r)| \; r dr,
\]
(this slight refinement of $L^1$ is needed in Proposition~\ref{L1} below), 
and observing that $\frac{h^2}{r^2}(1 + \hat h) \in L^{1,\log}$, 
\begin{equation}  \label{Fs}
  \| F s \|_{L^{1,\log}} \lec k \be \lec \be^2 \log \left( \frac{1}{\be} \right).
\end{equation}
by~\eqref{xiest2}.
Second, $\hat H$ (unlike $H$) is a {\it positive} operator on $L^2$
(with domain $D(\hat H) = D(-\Delta_r + \frac{4}{r^2})$), and in particular has 
no zero-energy resonance or eigenvalue. 
So a uniform in $\be$ bound for the resolvent $(\hat H + \be^2)^{-1}$, 
without any orthogonality condition, is possible:
\begin{prop} \label{L1}
For any $f \in L^{1,\log} \cap L^2 \cap C((0,\infty))$, 
\begin{equation} \label{L1est}
  \| (\hat H + \be^2)^{-1} f \|_{L^\infty \cap r^2 L^1_{\leq 1}}
  \lec \| f \|_{L^{1,\log}} .
\end{equation}   
\end{prop}  
This proved in Section~\ref{fullresolvent'}.

Combining~\eqref{L1est} with~\eqref{Fs} gives
\begin{equation}  \label{RFs}
  \| (\hat H + \be^2)^{-1} Fs \|_{L^\infty \cap r^2 L^1_{\leq 1}}
  \lec \be^2 \log \left( \frac{1}{\be} \right).
\end{equation}

From the expression for $N = N_{\be,\k,\al}(\xi)$ we can derive a
pointwise estimate of $|F N| \lec |N_r| + \frac{1}{r} |N|$
analogous to~\eqref{Nbound}, using $\| \xi \|_{L^\infty} \ll 1$ 
to simplify slightly:
\[ 
  N = N^{(1)} + N^{(2)}
\]
with
\[
  |(N^{(1)})_r| + \frac{1}{r} |N^{(1)}| \lesssim 
  \frac{|\xi|}{r} \left( \frac{h}{r} + \frac{|\xi|}{r} \right) \left( |\xi_r| + \frac{|\xi|}{r} \right) 
  + (k  + \be (1 + |\al|) ) \be \frac{h}{r} \left( |\xi_r| + \frac{|\xi|}{r}  \right)  
\]
and
\[
   |(N^{(2)})_r| + \frac{1}{r} |N^{(2)}| \lesssim 
   k \be \frac{|\xi|}{r} \left( |\xi_r| + \frac{\xi}{r} \right)
   + \be^2 (1 + |\al|) |\xi| (h + \xi^2) \left( |\xi_r| + \frac{|\xi|}{r} \right).
\]

Using the pointwise bound above, $F N^{(1)} \in L^{1,\log}$, with
\[
\begin{split}
  \| F N^{(1)} \|_{L^{1,\log}} &\lec
  \| \xi \|_X^2 \left( \left \| \frac{h}{r} \right \|_{L^{\infty,\log}} + 
  \left\| \frac{\xi}{r} \right \|_{L^{\infty,\log}} \right) +
  (k   + \be (1 + |\al|) ) \be \| \xi \|_X \left\| \frac{h}{r} \right \|_{L^{2,\log}} \\
  &\lec \be^2 \left( 1 +  \left\| \frac{\xi}{r} \right \|_{L^\infty} +  
  \| \xi \|_{L^\infty} \right ) \lec  \be^2,
\end{split}
\]
using $h/r \in L^{\infty,\log} \cap L^{2,\log}$, 
$\| \xi \|_{L^\infty} \lec \| \xi \|_X \lec \be$, \eqref{xi/r}, and 
\[
  \left\| \frac{\xi}{r} \right\|_{L^{\infty,\log}} = \sup_r \log(2+r) \frac{|\xi(r)|}{r} \lec
  \sup_{r \leq 1} \frac{|\xi(r)|}{r} + \left( \sup_{r \geq 1} \frac{\log(2+r)}{r} \right)
  \| \xi \|_{L^\infty} \lec   \left\| \frac{\xi}{r} \right \|_{L^\infty} +  
  \| \xi \|_{L^\infty}.
\]
Then by~\eqref{L1est},
\begin{equation}  \label{RFN1}
  \| (\hat H + \be^2)^{-1} F N^{(1)} \|_{L^\infty \cap L^1_{\leq 1}} \lec \be^2.
\end{equation}

Since the terms of $F N^{(2)}$ decay too slowly
to lie in $L^{1,\log}$, we also consider the resolvent
$(\hat{H} + \beta)^{-1}$ acting on $L^2$ functions, and 
get a bound with some loss in $\be$: 
\begin{prop}  \label{L2prop}
For $f \in L^2$, 
\begin{equation}   \label{boundL2}
  \| (\hat{H} + \beta^2)^{-1} f \|_{L^\infty \cap r^2 L^1_{\leq 1}}
  \lesssim \frac{1}{\beta} \log^{\frac{3}{2}}(\frac{1}{\beta}) \| f \|_{L^2}.
\end{equation}
\end{prop}
This is proved in Section~\ref{fullresolvent'}.

Using the pointwise bound above, $F N^{(2)} \in L^2$, with
\[
  \| F N^{(2)} \|_{L^2} \lec
  k \be \left\| \frac{\xi}{r} \right\|_{L^\infty} \| \xi \|_X
  + \be^2(1 + |\al|) \| \xi \|_{L^\infty}( 1 + \| \xi \|_{L^\infty}^2) \| \xi \|_X
  \lec \be^4 \log \left( \frac{1}{\be} \right),
\]
using~\eqref{relation}, \eqref{xiest2}, and~\eqref{xi/r}. 
Then by~\eqref{boundL2},
\begin{equation}  \label{RFN2}
  \| (\hat H + \be^2)^{-1} F N^{(2)} \|_{L^\infty \cap L^1_{\leq 1}} \lec 
   \be^3 \log^{\frac{5}{2}} \left( \frac{1}{\be} \right).
\end{equation}
Using~\eqref{RFs}, \eqref{RFN1} and~\eqref{RFN2} in~\eqref{pde'} gives
\begin{equation}  \label{etaest}
  \| \eta \|_{L^\infty \cap L^1_{\leq 1}} \lec \be^2 \log \left( \frac{1}{\be} \right).
\end{equation}

Finally, we by recover the desired estimate~\eqref{deriv} for $\xi$
from the estimate~\eqref{etaest} for $\eta = F \xi$ using:
\begin{lem} 
\begin{equation} \label{recover}
  \| \xi \|_{X^\infty} \lec
  \log \left( \frac{1}{\be} \right) \| \eta \|_{L^\infty \cap r^2 L^1_{\leq 1}}.
\end{equation}
\end{lem}
\begin{proof}
Solving the first order equation $\eta = F \xi = h \left( \frac{\eta}{h} \right)_r$ gives 
\begin{equation}  \label{xisplit}
  \xi = c h + \xi^\sharp, \qquad
  \xi^\sharp(r) = h(r) \int_1^r \frac{\eta(s)}{h(s)} ds,
\end{equation}
for some $c = c(\eta) \in \R$. 

The estimate
\begin{equation}  \label{xibar}
  \| \xi^\sharp \|_{X^\infty}
  \lec \| \eta \|_{L^\infty \cap r^2 L^1_{\leq 1}},
\end{equation}
follows from the asymptotic behaviour of $h$ 
\[
  h(r) \sim r, \;\; h'(r) \sim 1 \; \mbox{ for } \; r \leq 1, \qquad
  h(r) \sim \frac{1}{r}, \;\; h'(r) \sim \frac{1}{r^2} \; \mbox{ for } \;  r \geq 1,
\]
and the estimates: for $r \leq 1$,
\[
  \left| \int_1^r \frac{\eta(s)}{h(s)} ds \right| \lec
  \int_r^1 \frac{|\eta(s)|}{s^2} \; s ds \lec \| \eta \|_{r^2 L^1_{\leq 1}},
\]
and for $r \geq 1$,
\[
  \left| \int_1^r \frac{\eta(s)}{h(s)} ds \right| \lec
  \| \eta \|_{L^\infty} \int_1^r s ds \lec  r^2 \| \eta \|_{L^\infty}. 
\]

To bound the constant $c$ in~\eqref{xisplit}, we use
the orthogonality condition~\eqref{orth1}, in the form~\eqref{orth2}:
\begin{equation}  \label{orth3}
  0 = \langle \xi, \; W R_0(\be) h \rangle =
  c \langle h,\; W R_0(\be) h \rangle + \langle \xi^\sharp, \; W R_0(\be) h\rangle.
\end{equation}
Since
\[
  \|  W \xi^\sharp \|_{L^1 \cap r^{-3} L^\infty} \leq
  \| r W \|_{L^1 \cap r^{-3} L^\infty} \left\| \frac{\xi^\sharp}{r} \right\|_{L^\infty} 
  \lec \left\| \xi^\sharp \right\|_{X^\infty}, 
\]
we may apply~\eqref{innerbound2}, as well as~\eqref{innercomp}
in~\eqref{orth3} to get
\[
  \left( 4 \log \left( \frac{1}{\be} \right) + O(1) \right) c  =
  O \left( \log^2 \left( \frac{1}{\be} \right)  \left\| \xi^\sharp \right\|_{X^\infty} \right)
\]
and so using~\eqref{xibar},
 \[
   |c| \lesssim \log\left (\frac{1}{\be}\right )  \left\| \xi^\sharp \right\|_{X^\infty} \lec
   \log\left (\frac{1}{\be}\right ) \| \eta \|_{L^\infty \cap r^2 L^1_{\leq 1}}.
 \]
This, together with~\eqref{xisplit} and~\eqref{xibar}, gives~\eqref{recover}.
\end{proof}

%
%
%

Combining~\eqref{recover} and~\eqref{etaest} shows~\eqref{deriv} and
completes the proof of Theorem~\ref{estimates}.
\end{proof} 

\subsection{Uniqueness}  \label{uniqueness}

\begin{thm}  \label{unique}
Fix $A \geq 0$. 
There is $k_0 > 0$ such that for all $k \in (0,k_0]$,
$\al \in [-A, 1]$ and $\be > 0$, there is a unique minimizer of
$E_{k,\al,\be}$ in $X_1$.
\end{thm}
\begin{rem}
Theorem~\ref{unique} proves the uniqueness part of Theorem~\ref{mainthm}.
\end{rem}

\begin{proof}
The existence of minimizers is shown in Theorem~\ref{minexist}.
If the uniqueness fails, there are two families
$\{ u_k^{(\nu)} \} \subset X_1$, $\nu = 1, 2$, of minimizers
$E_{k,\al_k, \hat\be_k}(u_k^{(\nu)}) = e_{k,\al_k, \hat \be_k}$,
$\al_k \in [-A,1]$, $\hat \be_k > 0$,
which disagree along a sequence $(0,1) \ni k_j \to 0$:
$u_{k_j}^{(1)} \not= u_{k_j}^{(2)}$.

Applying Theorems~\ref{convergetoQ} and~\ref{estimates} to these
families of minimizers we have, dropping the subscripts $k$,
\begin{equation} \label{size0}
  u^{(\nu)}( \la^{(\nu)} r ) = Q(r) + \xi^{(\nu)}(r), \qquad
  \| \xi^{(\nu)} \|_X \lec \be^{(\nu)}, \qquad \nu = 1,2,
\end{equation}
for some $\la^{(\nu)} > 0$ with
\begin{equation}  \label{size}
  \be^{(\nu)} = \la^{(\nu)} \hat \be \to 0, \qquad
  k = \be^{(\nu)} \left( 2 \log \left( \frac{1}{\be^{(\nu)}} \right) + O(1) \right)
\end{equation}
where the orthogonality conditions~\eqref{orth1}
\begin{equation} \label{orth4}
  \F^{(\nu)} := R_0(\be^{(\nu)}) \left( s_{k,\al,\be^{(\nu)}} + N_{k,\al,\be^{(\nu)}}  
  (\xi^{(\nu)}) \right)  \quad \perp \; Wh
\end{equation}
hold, and the Euler-Lagrange equations~\eqref{pde3}, using~\eqref{factor1},
\begin{equation}  \label{pde4}
  \xi^{(\nu)} = (H + \be^{(\nu)} )^{-1} \left( s_{k,\al,\be^{(\nu)}} + N_{k,\al,\be^{(\nu)}}
  (\xi^{(\nu)}) \right)
  = I(\be^{(\nu)}) \F^{(\nu)}, \qquad I(\be) := \left( I - R_0(\be) W \right)^{-1}
\end{equation}
hold.
We will use difference estimates for $\xi^{(1)} - \xi^{(2)}$ and $\be^{(1)} - \be^{(2)}$
to show that for $k$ small enough, $\xi^{(1)} = \xi^{(2)}$ and $\be^{(1)} = \be^{(2)}$,
which shows $u^{(1)} = u^{(2)}$ and proves the uniqueness theorem.
Without loss of generality, we may assume
\begin{equation}  \label{size2}
  \be^{(2)} \lec \be^{(1)},
\end{equation}
otherwise we may restrict to a subsequence for which
$\be^{(1)} \lec \be^{(2)}$, and reverse the roles of $\nu=1$ and $\nu=2$ below.

From~\eqref{pde4},
\begin{equation}  \label{tri1}
\begin{split}
  \| \xi^{(1)} - \xi^{(2)} \|_X &\lec 
  \left\| I(\be^{(1)}) \left( \F^{(1)} - \F^{(2)} \right) \right\|_X + 
  \left\| \left( I(\be^{(1)}) -  I(\be^{(2)}) \right) \F^{(2)} \right\|_X \\
  &\lec \|  \F^{(1)} - \F^{(2)} \|_X + 
  \left\| \left( I(\be^{(1)}) -  I(\be^{(2)}) \right) \F^{(2)} \right\|_X,
\end{split}
\end{equation}
where the second line follows from \eqref{basic2}, since 
$\F^{(1)} - \F^{(2)} \perp Wh$ by~\eqref{orth4}.

From~\eqref{orth4},
\begin{equation}  \label{tri2}
\begin{split}
  \|  \F^{(1)} - \F^{(2)} \|_X &\lec 
  \| R_0( \be^{(1)} ) \left( s_{k,\al,\be^{(1)}} - s_{k,\al,\be^{(2)}} \right) \|_X  \\
  & \quad + 
  \left\| R_0( \be^{(1)} ) \left(  N_{k,\al,\be^{(1)}}(\xi^{(1)}) - 
  N_{k,\al,\be^{(2)}}(\xi^{(2)}) \right) \right\|_X  \\
  & \quad + \left\|  \left( R_0(\be^{(1)}) - R_0(\be^{(2)}) \right) 
  \left( s_{k,\al,\be^{(2)}} + N_{k,\al,\be^{(2)}}(\xi^{(2)}) \right) \right\|_X.
\end{split}
\end{equation}
For the first term in~\eqref{tri2}, note
\[
  s_{k,\al,\be^{(1)}} - s_{k,\al,\be^{(2)}}  = \left( (\be^{(2)})^2 - (\be^{(1)})^2 \right) h
  +  \al \left( (\be^{(1)})^2 -  (\be^{(2)})^2 \right) \frac{h^2}{r},
\]
and so by~\eqref{refined} and~\eqref{p=1},
\begin{equation}  \label{first}
   \| R_0( \be^{(1)} ) \left( s_{k,\al,\be^{(1)}} - s_{k,\al,\be^{(2)}} \right) \|_X   
   \lec  \left|\be^{(1)} - \be^{(2)} \right| \left( 1 + \frac {\be^{(2)}}{\be^{(1)}} \right)
   \lec  \left|\be^{(1)} - \be^{(2)} \right|. 
\end{equation}
For the third term in~\eqref{tri2}, the resolvent identity~\eqref{resid},
\eqref{p=-1}, together with the estimates~\eqref{sest}
and~\eqref{N1}-\eqref{N3} show
\begin{equation}  \label{third}
\begin{split}
  \left\|  \left( R_0(\be^{(1)}) - R_0(\be^{(2)}) \right) 
  \left( s_{k,\al,\be^{(2)}} + N_{k,\al,\be^{(2)}}(\xi^{(2)}) \right) \right\|_X 
  &\lec  \left| \be^{(1)} - \be^{(2)} \right| \left( \be^{(1)} + \be^{(2)} \right)
  \frac{1}{(\be^{(1)})^2} \be^{(2)}  \\
  &\lec  \left| \be^{(1)} - \be^{(2)} \right|.
\end{split}
\end{equation}
From the form~\eqref{Nform} of the nonlinear terms, and using also~\eqref{size0},
\eqref{size} and~\eqref{size2}, we can get difference estimate versions 
of~\eqref{Nbounds}:
\begin{equation}  \label{diffsplit}
  N_{k,\al,\be^{(1)}}(\xi^{(1)}) - N_{k,\al,\be^{(2)}}(\xi^{(2)}) = 
  \left( M_1 + M_2 + M_3 \right)  ( \xi^{(1)} - \xi^{(2)} )
  + \left( L_1 + L_2 + L_3 \right) ( \be^{(1)} - \be^{(2)} ),
\end{equation}
where
\[
\begin{split}
  & |M_1| \lec \frac{1}{r^2} \left(  (h + |\xi^{(1)}| + |\xi^{(2)}|)(|\xi^{(1)}| + |\xi^{(2)}|)
  \right) + k \be^{(1)} \frac{h}{r} \\
  & |M_2| \lec k \be^{(1)} \left( \frac{1}{r} +  h \right) 
  ( |\xi^{(1)}| + |\xi^{(2)}| ) \\
  & |M_3| \lec (\be^{(1)})^2  ( (\xi^{(1)})^2 + (\xi^{(2)})^2 ) \\
  & |L_1| \lec k \frac{h}{r} (  |\xi^{(1)}| + |\xi^{(2)}| ) \\
  & |L_2| \lec k \left( \frac{1}{r} + h \right) ( (\xi^{(1)})^2 + (\xi^{(2)})^2 )  \\
  & |L_3| \lec \be^{(1)} (  |\xi^{(1)}|^3 + |\xi^{(2)}|^3 )
\end{split}
\]
Then using~\eqref{p=1}, \eqref{p=0}, and~\eqref{p=-1},
\[
\begin{split}
  \| R_0(\be^{(1)}) M_1 ( \xi^{(1)} - \xi^{(2)} ) \|_X &\lec
  \| r^2 M_1 \|_{L^\infty} \| \xi^{(1)} - \xi^{(2)}  \|_X \\ &\lec
  \left( \| \xi^{(1)} \|_X +  \| \xi^{(2)} \|_X + k \be^{(1)}  \right) \| \xi^{(1)} - \xi^{(2)}  \|_X
  \\ &\lec \be^{(1)} \| \xi^{(1)} - \xi^{(2)}  \|_X, \\
  \| R_0(\be^{(1)}) M_2 ( \xi^{(1)} - \xi^{(2)} ) \|_X &\lec
  \frac{1}{\be^{(1)}} \| r M_2 \|_{L^\infty} \| \xi^{(1)} - \xi^{(2)}  \|_X \lec
  k  \be^{(1)} \| \xi^{(1)} - \xi^{(2)}  \|_X, \\
  \| R_0(\be^{(1)}) M_3 ( \xi^{(1)} - \xi^{(2)} ) \|_X &\lec
 \frac{1}{(\be^{(1)})^2} \| M_3 \|_{L^\infty} \| \xi^{(1)} - \xi^{(2)}  \|_X 
 \lec (\be^{(1)})^2 \| \xi^{(1)} - \xi^{(2)}  \|_X, \\
 \| R_0(\be^{(1)}) L_1 \|_X &\lec \| r L_1 \|_{L^2}
 \lec k  \left( \| \xi^{(1)} \|_X +  \| \xi^{(2)} \|_X \right) \lec k \be^{(1)}, \\
  \| R_0(\be^{(1)}) L_2 \|_X &\lec \frac{1}{\be^{(1)}} \| L_2 \|_{L^2} 
  \lec \frac{1}{\be^{(1)}} k  \left( \| \xi^{(1)} \|_X +  \| \xi^{(2)} \|_X \right) 
  \left( \| \xi^{(1)} \|_{L^\infty} +  \| \xi^{(2)} \|_{L^\infty} \right) \\
  & \lec k \be^{(1)}, \\
  \| R_0(\be^{(1)}) L_3 \|_X &\lec \frac{1}{(\be^{(1)})^2} \| L_3 \|_X
  \lec (\be^{(1)})^2.
\end{split}
\]
Using all of these in~\eqref{diffsplit}, for the second term in~\eqref{tri2} we get
\begin{equation}  \label{second}
  \left\| R_0( \be^{(1)} ) \left(  N_{k,\al,\be^{(1)}}(\xi^{(1)}) - 
  N_{k,\al,\be^{(2)}}(\xi^{(2)}) \right) \right\|_X   \lec
  \be^{(1)}  \| \xi^{(1)} - \xi^{(2)}  \|_X +
  k \be^{(1)} |\be^{(1)} - \be^{(2)}|.
\end{equation}
Combining~\eqref{first}, \eqref{third}, and~\eqref{second} we get an estimate
of the first term in~\eqref{tri1}:
\begin{equation}   \label{Fdiff}
  \|  \F^{(1)} - \F^{(2)} \|_X \lec
  |\be^{(1)} - \be^{(2)}| + \be^{(1)} \| \xi^{(1)} - \xi^{(2)}  \|_X.
\end{equation}

For the second term in~\eqref{tri1}, by resolvent identities
\[
\begin{split}
  \left( I(\be^{(1)}) - I(\be^{(2)}) \right) \F^{(2)} &= 
  \left( (\be^{(2)})^2 - (\be^{(1)})^2 \right)
  I(\be^{(1)}) R_0(\be^{(1)}) R_0(\be^{(2)}) W I(\be^{(2)}) \F^{(2)} \\
  &=   \left( (\be^{(2)})^2 - (\be^{(1)})^2 \right)
  I(\be^{(1)}) R_0(\be^{(1)}) R_0(\be^{(2)}) W \xi^{(2)},
\end{split}
\]
using~\eqref{pde4} in the last step.
The difficulty here is that $I(\be^{(1)})$ acts on a function which is not
$\perp W h$, and so bahaves badly as $\be^{(1)} \to 0$. Precisely,
\begin{equation}  \label{basic3}
  f \in X \;\; \implies \;\; 
  \|  I(\be) f  \|_X \lec \frac{1}{\be^2 \log \left(\frac{1}{\be}\right) }
  \left| \langle W h, \; f \rangle \right| + \| f \|_X
\end{equation} 
is proved in Section~\ref{fullresolvent}. Applying this,
\begin{equation}  \label{Idiff}
\begin{split}
  \left\|  \left( I(\be^{(1)}) - I(\be^{(2)}) \right) \F^{(2)} \right \|_X &\lec
  \left| \be^{(1)} - \be^{(2)} \right| \be^{(1)} \Big(
   \frac{1}{(\be^{(1)})^2 \log \left(\frac{1}{\be^{(1)}}\right) }
  \left| \left\langle R_0(\be^{(1)}) W h, \; R_0(\be^{(2)}) W \xi^{(2)} \rangle \right) \right|
  \\  & \qquad \qquad \qquad \qquad \qquad
  +  \|  R_0(\be^{(1)}) R_0(\be^{(2)}) W \xi^{(2)}\|_X \Big).
\end{split}
\end{equation}
Use~\eqref{p=-1} and~\eqref{p=1} for the second term in~\eqref{Idiff}:
\[
  \| R_0(\be^{(1)}) R_0(\be^{(2)}) W \xi^{(2)}\|_X \lec
  \frac{1}{(\be^{(1)})^2} \| r W \xi^{(2)} \|_{L^2} \lec
  \frac{1}{(\be^{(1)})^2} \| r^2 W \|_{L^\infty}  \| \xi^{(2)} \|_{X}
  \lec \frac{\be^{(2)}}{(\be^{(1)})^2}.
\] 
For the first term, we use an $L^2$ estimate for the free resolvent
acting on well-localized functions, proved in Section~\ref{refinement},
\begin{equation} \label{refined2}
  \| R_0(\be) f \|_{L^2} \lec \log^{\frac{1}{2}} \left( \frac{1}{\be} \right) \| r f \|_{L^1}
  + \| f \|_{L^1} \lec \log^{\frac{1}{2}} \left( \frac{1}{\be} \right) 
  \| (1 + r) f \|_{L^1}
\end{equation}
to get
\[
\begin{split}
  \left| \left\langle R_0(\be^{(1)}) W h, \; R_0(\be^{(2)}) W \xi^{(2)} \right\rangle \right|
  &\lec \log^{\frac{1}{2}} \left( \frac{1}{\be^{(1)}} \right)
  \log^{\frac{1}{2}} \left( \frac{1}{\be^{(2)}} \right) \| (1+r)  W h \|_{L^1}
  \| r(1+r) W \|_{L^2} \| \xi^{(2)} \|_X \\
  &\lec \log^{\frac{1}{2}} \left( \frac{1}{\be^{(1)}} \right)
  \log^{\frac{1}{2}} \left( \frac{1}{\be^{(2)}} \right) \be^{(2)}.
\end{split}
\]
Using the last two estimates in~\eqref{Idiff} gives
\[
\begin{split}
  \left\|  \left( I(\be^{(1)}) - I(\be^{(2)}) \right) \F^{(2)} \right \|_X 
  &\lec  \left| \be^{(1)} - \be^{(2)} \right| \be^{(1)} \left(
   \frac{ \log^{\frac{1}{2}} \left( \frac{1}{\be^{(1)}} \right)
  \log^{\frac{1}{2}} \left( \frac{1}{\be^{(2)}} \right) \be^{(2)}}
  {(\be^{(1)})^2 \log \left(\frac{1}{\be^{(1)}}\right) }
  + \frac{\be^{(2)}}{(\be^{(1)})^2} \right) \\
  &\lec  \left| \be^{(1)} - \be^{(2)} \right|  \left(
  \frac{\be^{(2)}  \log \left( \frac{1}{\be^{(2)}} \right)}
  {\be^{(1)}  \log \left( \frac{1}{\be^{(1)}} \right)  } 
  + \frac{\be^{(2)}}{\be^{(1)}} \right)
  \lec  \left| \be^{(1)} - \be^{(2)} \right|
\end{split}
\]
by~\eqref{size2}. Combining this with~\eqref{Fdiff}, we complete
the estimate of $\xi^{(1)} - \xi^{(2)}$ in~\eqref{tri2}:
\[
  \| \xi^{(1)} - \xi^{(2)} \|_X \lec \left| \be^{(1)} - \be^{(2)} \right|
  + \be^{(1)} \| \xi^{(1)} - \xi^{(2)} \|_X
\]
and so for sufficiently large $j$,
\begin{equation}  \label{xidiff}
  \| \xi^{(1)} - \xi^{(2)} \|_X \lec \left| \be^{(1)} - \be^{(2)} \right|.
\end{equation}

It remains to estimate $\be^{(1)} - \be^{(2)}$.
For this we use the orthogonality conditions~\eqref{orth4} re-written as
\[
  \gamma(\be^{(\nu)}) = S^{(\nu)},
  \qquad 
  \gamma(\be) := \be \left( R_0(\be) W h, \; h \right), \;\; \mbox{ and}
\]
\[
  S^{(\nu)} := 
  (k + \be^{(\nu)} \al)
  \left( R_0(\be^{(\nu)})  Wh, \frac{h^2}{r} \right) +
  \frac{1}{\be} \left( R_0(\be^{(\nu)}) W h,
   N_{k,\al,\be^{(\nu)}}(\xi^{(\nu)}) \right).
\]
We need a monotonicity estimate for $\gamma(\be)$,
proved in Section~\ref{freeinners}:
\begin{lem} \label{monotone}
\[
  | \gamma(\be^{(1)}) - \gamma(\be^{(2)}) | \gec
  \log \left( \frac{1}{\be^{(1)}} \right) | \be^{(1)} - \be^{(2)} |
\]
\end{lem}.
We also need more difference estimates:
\begin{lem}  \label{morediff}
\[
  |S^{(1)} - S^{(2)}| \lec |\be^{(1)} - \be^{(2)}| + \| \xi^{(1)} - \xi^{(2)} \|_X.
\]
\end{lem}
The proof is below.
Applying these two lemmas to the previous relation gives
\[
  |\be^{(1)} - \be^{(2)}| \lec \frac{1}{\log \left( \frac{1}{\be^{(1)}} \right)}
  \left(  |\be^{(1)} - \be^{(2)}| + \| \xi^{(1)} - \xi^{(2)} \|_X \right),
\]
so for $k$ sufficiently small,
\[
  |\be^{(1)} - \be^{(2)}| \lec \frac{1}{\log \left( \frac{1}{\be^{(1)}} \right)}
  \| \xi^{(1)} - \xi^{(2)} \|_X.
\]

Combining this with~\eqref{xidiff}, we see that $\xi^{(1)} = \xi^{(2)}$
and $\be^{(1)} = \be^{(2)}$ for $k$ sufficiently small. 
It follows that $u^{(1)} = u^{(2)}$ for $k$ sufficiently small.
This contradiction proves Theorem~\ref{unique}.
\end{proof}

{\it Proof of Lemma~\ref{morediff}}:
by resolvent identity,~\eqref{diffsplit},~\eqref{p=1},~\eqref{p=0},~\eqref{refined2},
\eqref{p=-1}, the estimates of Section~\ref{remainder} and~\eqref{second},
\[
\begin{split}
  & |S^{(1)} - S^{(2)}| \lec
  |\be^{(1)} - \be^{(2)}|  \left[  \; \left| \left( R_0(\be^{(1)})  Wh, \frac{h^2}{r} \right) \right|
  + \be^{(2)} (\be^{(1)} + \be^{(2)}) \left|  \left( R_0(\be^{(2)}) R_0(\be^{(1)})  Wh, 
  \frac{h^2}{r} \right) \right| \right. \\
  & \quad \left. +
  \frac{\be^{(1)}+\be^{(2)}}{\be^{(1)}}  \left|
  \left( Wh, R_0(\be^{(1)}) R_0(\be^{(2)}) N_{k,\al,\be^{(2)}}(\xi^{(2)}) \right) \right|
  + \frac{1}{\be^{(1)} \be^{(2)}} 
  \left| \left(  Wh, R_0(\be^{(2)}) N_{k,\al,\be^{(2)}}(\xi^{(2)}) \right) \right|
  \right] \\ 
  & \quad \quad + \frac{1}{\be^{(1)}}
   \left| \left(  W h, R_0(\be^{(1)}) N_{k,\al,\be^{(1)}}(\xi^{(1)}) - N_{k,\al,\be^{(2)}}(\xi^{(2)})
   \right) \right| \\
   & \qquad \qquad \quad \lec |\be^{(1)} - \be^{(2)}|  \left[
   1 + \be^{(2)} (\be^{(1)} + \be^{(2)}) \frac{1}{\be^{(2)}} \log \left( \frac{1}{\be^{(1)}}
   \right) 
   + \frac{\be^{(1)}+\be^{(2)}}{\be^{(1)}} \frac{(\be^{(2)})^2}{(\be^{(1)})^2}
   +  \frac{(\be^{(2)})^2}{\be^{(1)} \be^{(2)}} \right] \\
   & \quad \qquad \qquad \qquad \qquad +   
   \frac{1}{\be^{(1)}}
   \left( \be^{(1)}  \| \xi^{(1)} - \xi^{(2)}  \|_X + \be^{(1)} |\be^{(1)} - \be^{(2)}|
   \right) . \\
   & \qquad \qquad \quad \lec |\be^{(1)} - \be^{(2)}| +  \| \xi^{(1)} - \xi^{(2)}  \|_X
\end{split}
\]
as required.
$\Box$

\subsection{Energy asymptotics and $L^p$ estimates} 
\label{sec:energy}

To complete the proof of Theorem~\ref{mainthm},
it remains to prove the energy asymptotics~\eqref{energyThm}. 
By~\eqref{law2}, we may compute for the special case
\[
  k = k(\be) = \be \left( 2 \log \left( \frac{1}{\be} \right) + O(1) \right)
\]
so that by Theorem~\ref{estimates}, the minimizer $v_{k,\al,\be}$ satisfies
\[
  v_{k,\al,\be} = Q + \xi, \quad \xi = \xi_{k,\al,\be}, \quad
  \| \xi  \|_X \lec \be.
\] 
This energy-space bound alone is not sufficient
to estimate the D-M energy $E_{DM}(v_{k,\al,\be})$, so
at the same time, we obtain $L^p$ estimates of $\xi$ by interpolating between 
the energy space estimate, and the 
$L^2$ (and weak $L^2$) information provided by the 
anisotropy/Zeeman energy term $E^{(\al)}(v_{k,\al,\be})$.

Begin with the D-M energy:
\[
\begin{split}
  \left| E_{DM}(Q + \xi) - E_{DM}(Q) \right| &= \left|
  \int_0^\infty \left( \sin^2(Q + \xi) (Q_r + \xi_r) - \sin^2 Q \; Q_r \right) r dr \right| \\
  & = \left|  \int_0^\infty \left( h^2 \xi_r + ((1-2h^2) \sin^2 \xi + 2 h \hat h \cos \xi \sin \xi) 
  (-\frac{h}{r} + \xi_r)  \right) r dr  \right| \\
  &\lec \int_0^\infty \left( (h^2 + \xi^2) |\xi_r| + \frac{h}{r}( \xi^2 + h |\xi| ) \right) r dr \\
  &\lec \left( \| h \|_{L^4}^2 + \| \xi \|_{L^4}^2 \right) \| \xi_r \|_{L^2}
  + \left( \| h \|_{L^4}^2 + \| h \|_{L^\infty} \| \xi \|_{L^\infty} \right) \| \frac{\xi}{r} \|_{L^2} \\
  &\lec  \left( 1 + \| \xi \|_{L^4}^2 + \| \xi \|_X \right) \| \xi \|_X
  \lec (1 + \| \xi \|_{L^4}^2 ) \be,
\end{split}
\]
while by~\eqref{h3/r},
\[
  E_{DM}(Q) = \int_0^\infty \sin^2 Q \; Q_r \; r dr =
  -\int_0^\infty \frac{h^3}{r} \; r dr  =  -2,
\]
so that 
\begin{equation}  \label{DMasy}
  E_{DM}(v_{k,\al,\be}) = E_{DM}(Q + \xi) = -2 + O( (1 + \| \xi \|_{L^4}^2 ) \be ).
\end{equation}

Next, the Zeeman-anisotropy energy:
since $v_{k,\al,\be}$ is a critical point of $E_{k,\al,\be}$, 
the Pohozaev relation~\eqref{poho0} holds,
\begin{equation}  \label{Ani-est}
  \be^2 E^{(\al)} (v_{k,\al,\be}) = -\frac{1}{2} \be k(\be) E_{DM}(Q)
  =  \be k(\be) \left( 1 + O((1 + \| \xi \|_{L^4}^2 ) \be ) \right).
\end{equation}
We can extract $L^2$ information from this:
\[
  \| \sin(Q + \xi) \|_{L^2}^2 = \int_0^\infty \sin^2(Q + \xi) \; r dr = 
  2 E_a(v_{k,\al,\be}) \leq 2 E^{(\al)}(v_{k,\al,\be}) 
  \lec \log \left( \frac{1}{\be} \right) \left( 1 + \be \| \xi \|_{L^4}^2 \right).
\] 
By trig identity,
\[
  \sin \xi + h = \sin(Q + \xi) + h (1-\cos(\xi))  + (1-\hat{h}) \sin(\xi),
\]
and so
\begin{equation}  \label{sin+h}
\begin{split}
  \| \sin \xi + h \|_{L^2} &\lec \| \sin(Q + \xi) \|_{L^2} + \| r h \|_{L^\infty} \| \xi \|_{L^\infty} \| \xi \|_X
  + \| 1 - \hat {h} \|_{L^2} \| \xi \|_{L^\infty} \\
  &\lec \log^{\frac{1}{2}} \left( \frac{1}{\be} \right) \left( 1 + \sqrt{\be} \| \xi \|_{L^4} \right)
  + \be^2 + \be \lec \log^{\frac{1}{2}} \left( \frac{1}{\be} \right) \left( 1 + \sqrt{\be} 
  \| \xi \|_{L^4} \right).
\end{split}
\end{equation}
Since $\| \xi \|_{L^\infty} \lec \be \ll 1$,
\[
  |\xi(r)| \lec | \sin \xi(r) | \leq |\sin \xi(r) + h(r)| + h(r),
\]
and since $h \in L^{2,w}$,
\[
  \| \xi \|_{L^{2,w}} \lec \| \sin \xi + h \|_{L^2} + \| h \|_{L^{2,w}} \lec 
   \log^{\frac{1}{2}} \left( \frac{1}{\be} \right) \left( 1 + \sqrt{\be} \| \xi \|_{L^4} \right) + 1 \lec 
   \log^{\frac{1}{2}} \left( \frac{1}{\be} \right) \left( 1 + \sqrt{\be} \| \xi \|_{L^4} \right).
\]
Simple interpolation with $\| \xi \|_{L^\infty} \lec \be$ yields
\begin{equation}  \label{interp}
  \| \xi \|_{L^p} \leq \left( \frac{p}{p-2} \right)^{\frac{1}{p}}
  \| \xi \|_{L^{2,w}}^{\frac{2}{p}} \| \xi \|_{L^\infty}^{1-\frac{2}{p}} \lec 
  \left( \frac{p}{p-2} \right)^{\frac{1}{p}} \log^{\frac{1}{p}}\left( \frac{1}{\be} \right) 
  \left( 1 + \sqrt{\be} \| \xi \|_{L^4} \right)^{\frac{2}{p}} \be^{1-\frac{2}{p}},
  \qquad p > 2.
\end{equation}
In particular, the case $p=4$
\[
  \| \xi \|_{L^4} \lec \log^{\frac{1}{4}} \left( \frac{1}{\be} \right)
  \left(1 + \sqrt{\be} \| \xi \|_{L^4} \right)^{\frac{1}{2}} \be^{\frac{1}{2}}
  \lec \log^{\frac{1}{4}} \left( \frac{1}{\be} \right) \be^{\frac{1}{2}} +
  \log^{\frac{1}{4}} \left( \frac{1}{\be} \right) \be^{\frac{3}{4}} \| \xi \|_{L^4}^{\frac{1}{2}}
\]
shows that
\begin{equation}  \label{L4}
  \| \xi \|_{L^4} \lec \log^{\frac{1}{4}} \left( \frac{1}{\be} \right) \be^{\frac{1}{2}}.
\end{equation}
With this, we can return to~\eqref{interp} to get
\[
  \| \xi \|_{L^p} \lec  \left( \frac{p}{p-2} \right)^{\frac{1}{p}} 
  \log^{\frac{1}{p}}\left( \frac{1}{\be} \right) \be^{1-\frac{2}{p}}, \qquad p > 2,
\]
and~\eqref{sin+h} to get
\[
   \| \sin \xi + h \|_{L^2} \lec \log^{\frac{1}{2}} \left( \frac{1}{\be} \right).
\]
Then
\[
  \| \xi - \sin \xi \|_{L^2} \lec \| \xi^3 \|_{L^2} = \| \xi \|_{L^6}^3
  \lec \log^{\frac{1}{2}} \left( \frac{1}{\be} \right) \be^2,
\]
so
\[
  \| \xi + h \|_{L^2}
  \leq \| \sin \xi + h \|_{L^2} + \| \xi - \sin \xi \|_{L^2} \lec 
  \log^{\frac{1}{2}} \left( \frac{1}{\be} \right) (1 + \be^2) \lec 
  \log^{\frac{1}{2}} \left( \frac{1}{\be} \right).
\]

Finally, we can return to the energy computation.
For the exchange energy: since $Q$ minimizes $E_e$
(among finite-energy configurations with the given boundary conditions),
\[
\begin{split}
  0 \leq E_e(Q+\xi) - E_e(Q) &= E_e(Q+\xi) - E_e(Q) - \langle E_e'(Q), \xi \rangle \\
  &= \frac{1}{2} \int_0^\infty \left(
  \xi_r^2 + \frac{1}{r^2} (\sin^2(Q + \xi) - \sin^2 Q - 2 \sin Q \cos Q \; \xi ) \right) r dr  \\
  &= \frac{1}{2} \int_0^\infty \left( \xi_r^2 + \frac{1}{r^2} (
  (1-2 h^2 ) \sin^2 \xi + 2 h \hat{h} (\cos \xi \sin \xi - \xi) ) \right) r dr \\
  & \lec \int_0^\infty \left( \xi_r^2 + \frac{1}{r^2} ( \xi^2 + \xi^3 ) \right) r dr
  \lec \| \xi \|_X^2(1 + \| \xi \|_{L^\infty}) \lec \|\xi \|_X^2(1 + \|\xi\|_X) \lec \be^2,
\end{split}
\] 
so that
\begin{equation}  \label{exasy}
  E_e(v_{k,\al,\be}) = E_e(Q+\xi) = E_e(Q) + O(\be^2) = 2 + O(\be^2).
\end{equation}
Combining~\eqref{exasy}, \eqref{DMasy}, \eqref{Ani-est}
and~\eqref{L4}: 
\[
\begin{split}
  E_e(Q) - E_{k,\al,\be}(v_{k,\al,\be}) &= 
  -E_e(v_{k,\al,\be}) + 2 - \be k(\be) E_{DM}(v_{k,\al,\be}) 
  - \be^2 E^{(\al)}(v_{k,\al,\be}) \\
  &= -E_e(v_{k,\al,\be}) + 2 - \frac{1}{2} \be k(\be) E_{DM}(v_{k,\al,\be})
  = O(\be^2) + \be k(\be) \left( 1 + O(\be) \right) \\
  &= \be k(\be) + O(\be^2)
  = 2 \be^2 \log \left( \frac{1}{\be} \right) + O(\be^2)
\end{split}
\]
which implies~\eqref{energyThm}.

This completes the proof of Theorem~\ref{mainthm}. 
$\Box$

\section{Appendices}

\subsection{Symmetries}
\label{symm}

Here we Prove Lemma~\ref{invar}:
\begin{proof}
For~\eqref{symmetry}, just compute:
\[
  \nabla \times \left( \hat{m}(e^{\phi \tR} x) \right) =
  \left[ \begin{array}{c} \p_2 \; m_3(e^{\phi \tR} x) \\ -\p_1 \; m_3(e^{\phi \tR} x)
  \\ \p_1 m_2(e^{\phi \tR} x) - \p_2 m_1(e^{\phi \tR} x) 
  \end{array} \right] = \left[ \begin{array}{c}
  -\sin(\phi) (m_3)_1 + \cos(\phi) (m_3)_2 \\
  -\cos(\phi) (m_3)_1 - \sin(\phi) (m_3)_2 \\
  \cos(\phi) (m_2)_1 + \sin(\phi)(m_2)_2 + \sin(\phi) (m_1)_1 - \cos(\phi)(m_1)_2 
  \end{array} \right]
\]
and so 
\[
\begin{split}
  \hat{m}(e^{\phi \tR} x) \cdot  \nabla \times \left( \hat{m}(e^{\phi \tR} x) \right) &=
  \sin(\phi) \left( - m_1(m_3)_1 - m_2 (m_3)_2 + m_3(m_2)_2 + m_3(m_1)_1
   \right) \\ & + \cos(\phi) \left( m_1(m_3)_2 - m_2(m_3)_1 
   + m_3(m_2)_1 - m_3(m_1)_2 \right) 
\end{split}
\]
Since the change of variables $x \mapsto e^{\phi \tR} x$
has unit Jacobian, after integration we have
\begin{equation} \label{first2}
\begin{split}
  \E_{DM}\left( \hat{m}(e^{\phi \tR} \cdot ) \right) &=
  \sin(\phi) \int \left( - m_1(m_3)_1 - m_2 (m_3)_2 + m_3(m_2)_2 + m_3(m_1)_1
   \right) \; dx\\ & + \cos(\phi) \int \left( m_1(m_3)_2 - m_2(m_3)_1 
   + m_3(m_2)_1 - m_3(m_1)_2 \right) \; dx
  \end{split}
\end{equation}
On the other hand, compute 
\[
  \nabla \times \left( e^{\phi R} \hat{m} \right) =
  \left[ \begin{array}{c} \p_2 m_3 \\ -\p_1 m_3 \\
  \p_1 \left( \sin(\phi) m_1 + \cos(\phi) m_2 \right) -
  \p_2 \left( \cos(\phi) m_1 - \sin(\phi) m_2 \right) \end{array} \right] 
\]
so
\[
\begin{split}
  \left( e^{\phi R} \hat{m} \right) \cdot \nabla \times 
  \left( e^{\phi R} \hat{m} \right)     
  &= \left( \cos(\phi) m_1 - \sin(\phi) m_2 \right) (m_3)_2 -
  \left( \sin(\phi) m_1 + \cos(\phi) m_2 \right) (m_3)_1 \\
  & + m_3 \left( \sin(\phi)((m_1)_1 + (m_2)_2) + \cos(\phi)
  ((m_2)_1 - (m_1)_2) \right)
\end{split}
\]
and so after integration,
\begin{equation} \label{second2}
\begin{split}
  \E_{DM} \left( e^{\phi R} \hat{m} \right) &=
  \sin(\phi) \int \left(
  -m_2(m_3)_2 - m_1 (m_3)_1 + m_3 (m_1)_1 + m_3 (m_2)_2 \right) \; dx \\
  &+ \cos(\phi) \int \left( m_1(m_3)_2 - m_2 (m_3)_1 +
  m_3(m_2)_1 - m_3(m_1)_2 \right) \; dx  \\
  &= \E_{DM} \left( \hat{m}(e^{\phi \tR} \cdot ) \right)
\end{split}
\end{equation}
using~\eqref{first2} in the last step. Now replacing 
$\hat{m}$ with $e^{-\phi R} \hat{m}$ in~\eqref{second2}
yields~\eqref{symmetry}. 

For~\eqref{symmetry2}, it is easily checked that
\[
  \left( F \hat{m}  \cdot \nabla \times (F \hat{m})  \right)(x)
  = \left(\hat{m} \cdot \nabla \times \hat{m} \right)(x_1,-x_2),
\] 
and then~\eqref{symmetry2} follows from the fact that
$(x_1,x_2) \mapsto (x_1,-x_2)$ has unit Jacobian.
\end{proof}

\subsection{Differentiability of the energy}   \label{functional}

Here we prove Proposition~\ref{differentiability}.
\begin{proof}
Let $u \in X_n \cap L^2$ and $\xi \in X \cap L^2$.
Using simple trig identities and the elementary bounds
$|1 - \cos \xi| + |\xi - \sin \xi| \lec \xi^2$, we find:
\[
  \left| E_e(u + \xi) - E_e(u) - \int_0^\infty \left( u_r \xi_r  + \frac{1}{r^2} \sin u \cos u 
  \right) r dr \right| \lec \| \xi \|_X^2;
\]
\[
  \left| E_{DM}(u + \xi) - E_{DM}(u) + \int_0^\infty \frac{1}{r} \sin^2 u \; \xi \;  r dr 
  -\int_0^\infty \left( \sin^2 u \; \xi \; r \right)_r dr \right| 
  \lec \| u_r \|_{L^2} \| \xi \|_{L^2} \| \xi \|_{L^\infty} + \| \xi_r \|_{L^2} \| \xi \|_{L^2}
\]
and so since $u \in X_n \cap L^2$, $\xi \in X \cap L^2$ implies
$\sin^2 u \; \xi \; r \to 0$ as $r \to 0$ and $r \to \infty$ (using~\eqref{embedding2}), 
\[
  \left| E_{DM}(u + \xi) - E_{DM}(u) + \int_0^\infty \frac{1}{r} \sin^2 u \; \xi \;  r dr 
  \right|  \lec \| \xi \|_{X \cap L^2}^2;
\]
\[
  \left| E_z(u + \xi) - E_z(u) - \int_0^\infty \sin u \; \xi \; r dr \right| \lec \| \xi \|_{L^2}^2;
\]
\[
  \left| E_a(u + \xi) - E_a(u) - \int_0^\infty \sin u \cos u \; \xi \; r dr \right| \lec \| \xi \|_{L^2}^2.
\]
Combining these shows that $E_{k,\al,\be}$ is differentiable on $X_n \cap L^2$ with
Fr\'echet derivative at $u$ given by
\[
  X \cap L^2 \ni \xi \mapsto \int_0^\infty \left(
  u_r \xi_r + \left[ \frac{1}{r^2} \sin u \cos u - k \be \frac{1}{r} \sin^2 u  +
  \be^2 \left( \al \sin u \cos u + (1-\al) \sin u \right)
  \right] \xi \right) r dr.
\]
If $v$ minimizes $E_{k,\al,\be}$ on $X_n \cap L^2$, this derivative
vanishes for all $\xi \in X \cap L^2$. In particular, taking 
$\xi \in C_0^\infty((0,\infty))$, we see that $v$ satisfies the Euler-Lagrange 
equation~\eqref{E-L} in the sense of distributions.
It follows that $v_{rr} \in L^2_{loc}((0,\infty))$, so in fact 
$v \in H^2_{loc}$. Continuing in this way, $v \in H^k_{loc}$ for all $k$,
and so $v \in C^\infty((0,\infty))$ satisfies~\eqref{E-L} in the classical sense.

The exponential decay of $v$ is standard. Since $v \in X_n \cap L^2$, it
has some decay by~\eqref{embedding2}. The linear approximation to
the Euler-Lagrange equation~\eqref{E-L} around $v=0$ is
\[
  0 \approx -v_{rr} - \frac{v_r}{r} + \frac{v}{r^2} + \be^2 v 
\]
whose decaying fundamental solution is the modified Bessel
function $K_1(\be r)$, which  decays like $\frac{e^{-\be r}}{\sqrt{\be r}}$
as $r \to \infty$. See \cite{li2018stability} for details.

To obtain the Pohozaev-type identity~\eqref{poho0}, multiply
the Euler-Lagrange equation~\eqref{E-L} by $r v_r$
and integrate over an interval $[s \; R]$ with 
$0 < s < R < \infty$ to obtain
\begin{equation} \label{prepoho1}
  k \be \int_s^R e_{DM}(v) r dr + 2 \be^2 \int_s^R e^{(\al)}(v) r dr
  = \left[  -\frac{1}{2} (r v_r)^2 + \frac{1}{2} 
  \sin^2 v + \be^2 r^2 e^{(\al)}(v) \right] \Big |^R_s.
\end{equation}
Since $u \in X_n$, we have $\lim\limits_{r \to 0} v(r) = n \pi$ and so
$\lim\limits_{r \to 0} \sin v = \lim\limits_{r \to 0} r^2 e^{(\al)}(v) = 0$.
Since $v \in X_n \cap L^2$, $v_r \in L^2$, and by~\eqref{E-L} 
and the exponential decay,  $(r v_r)_r \in L^2$. 
So $r v_r \in X$ and in particular $\lim\limits_{r \to 0} r v_r = 0$.
Taking $s \to 0$ in~\eqref{prepoho1} gives
\begin{equation} \label{prepoho2}
  k \be \int_0^R e_{DM}(v) r dr + 2 \be^2 \int_0^R e^{(\al)}(v) r dr =
  -\frac{1}{2} R v_r^2(R) + \frac{1}{2} \sin^2 v(R) + \be^2 R^2 e^{(\al)}(v)(R).
\end{equation}
Taking $R \to \infty$ and using the exponential decay 
gives the desired relation~\eqref{poho0}.

Finally, we verify directly that~\eqref{mfromv} satisfies~\eqref{E-L1}. Let 
\[
 \hat e(r,\th) := [ -\cos(v(r)) \sin(\theta), \cos(v(r)) \cos(\theta), -\sin(v(r)) ].
\]
We notice that $|\hat e| = 1$,  and $\hat e \perp \hat m$, so 
$\hat e(r,\th) \in T_{\hat m(r,\th)} \S^2$. Compute
\[
  \hat{m}_r = v_r \hat{e}, \quad
  \hat{m}_{rr} = v_{rr} \hat{e} - v_r^2 \hat {m}, \quad
  \hat{m}_{\theta\theta}  = -\sin^2(v) \hat{m} - \sin(v) \cos(v) \hat{e},
\]
and so
\[
  P_{T_{\hat m} \S^2} \Delta\hat{m} = 
  P_{T_{\hat m} \S^2} \left( \hat{m}_{rr} + \frac{1}{r} \hat{m}_r + 
  \frac{1}{r^2} \hat{m}_{\theta\theta} \right)
  = \left( v_{rr}  + \frac{1}{r} v_r  - \frac{\sin(v) \cos(v)}{r^2}  \right) \hat{e}.
\]
Compute
\[
  \nabla \times \hat{m} = 
  [ -v_r \sin(v) \sin(\theta), v_r \sin(v) \cos(\theta), v_r \cos(v) + \frac{1}{r} \sin(v) ]
  = v_r \hat m + \frac{1}{r} \sin (v) \hat k,
\]
so
\[
   P_{T_{\hat m} \S^2} (\nabla \times \hat{m}) = 
   \frac{1}{r} \sin(v) P_{T_{\hat m} \S^2} \hat k,
\]
and
\[
   P_{T_{\hat m} \S^2} \hat k = \hat k - m_3 \hat m =
   -\sin(v) \hat e.
\]
Then by~\eqref{E-L1}
\[
\begin{split}
  \E'(\hat{m}) &= P_{T_{\hat m} \S^2} 
  \left[ -\Delta \hat m + \be \; k \; \nabla \times \hat m + 
  \be^2(\al -1 - \al m_3) \hat k \right] \\ &= \left(
  -v_{rr} - \frac{1}{r} v_r + \frac{\sin(v) \cos(v)}{r^2} 
  - \be \; k \; \frac{1}{r} \sin^2(v) 
  + \be^2 (1-\al + \al \cos(v)) \sin(v) \right) \hat e.
\end{split}
\]
So the full Euler-Lagrange equation~\eqref{E-L1} is satsified, 
since the reduced one~\eqref{E-L} is. 
\end{proof}

\subsection{Unboundedness from below of the energy}
\label{unbounded}

Here we prove Proposition~\ref{unbound}.
\begin{proof}
Without loss of generality, we may take $n \geq 0$.
First suppose $k > 0$.
Consider, for $M \gg t \gg 1$, $M$ an odd integer, 
a piecewise linear test function of the form
\[
  u(r) = \left\{ \begin{array}{cl} n \pi + t r & 0 \leq r \leq \frac{M-n}{t} \pi \\
  (M + \frac{M-n}{t}) \pi - r & \frac{M-n}{t} \pi \leq r \leq (\frac{M-n}{t} + M) \pi \\
  0 & r \geq (\frac{M-n}{t} + M) \pi
  \end{array} \right\} \in X_n.
\]
Begin with the exchange energy:
\[
\begin{split}
  \frac{1}{2} \int_0^\infty u_r^2 \; r dr &= 
  \frac{1}{2} t^2 \frac{1}{2} \frac{(M-n)^2}{t^2} \pi^2 + 
  \frac{1}{2} \frac{1}{2} \pi^2 \left( ( \frac{M-n}{t} + M )^2 - \frac{(M-n)^2}{t^2} \right) \\
  &= \frac{M^2 \pi^2}{4} \left( (1- n/M)^2 + 1 + \frac{2}{t}(1 - n/M) \right)
  = \frac{M^2 \pi^2}{2} \left( 1 + O(\frac{1}{t}) \right).
\end{split}
\]
The other part of the exchange is (mainly) bounded by the anisotropy:
using $\sin^2(n\pi + t r) = \sin^2(t  r) \leq \min ( t^2 r^2, \; 1 )$,  
\[
\begin{split} 
  \frac{1}{2} \int_0^\infty \frac{\sin^2(u)}{r^2} \; r dr &\leq
  \frac{1}{2} \int_0^{1/t} t^2 \; r dr + \frac{1}{2} \int_{1/t}^{M/t} \frac{dr}{r} +
  \frac{1}{2} \frac{t^2}{M^2} \int_{M/t}^\infty \sin^2(u) \; r dr \\
  & = O \left( 1 + \log M  + \frac{t^2}{M^2} E_a(u)  \right),
\end{split}
\]
and combining these gives
\[
  E_e(u) = \frac{M^2 \pi^2}{2} \left( 1 + O( \frac{1}{t}
  + \frac{t^2}{M^4} E_a(u) ) \right).
\]
Next the anisotropy: setting $S(y) = y/2 - \sin(2y)/4$, 
$\hat S(y) = y^2/4 - y \sin(2y)/4 + \sin^2(y)/4$, so that 
$S'= \sin^2(y)$ and $\hat S'= y \sin^2(y)$,
\[
\begin{split}
  E_a(u) &= \frac{1}{2} \int_0^\infty \sin^2(u(r)) \; r dr =
  \frac{1}{2} \int_0^{\frac{M-n}{t} \pi} \sin^2(tr) \; r dr +
  \frac{1}{2} \int_{\frac{M-n}{t} \pi}^{(\frac{M-n}{t} + M) \pi}
  \sin^2(r-\frac{M-n}{t}\pi) \; r dr \\
  &= \frac{1}{2 t^2}\hat S((M-n)\pi) + \frac{1}{2}\hat S(M \pi) + 
  \frac{(M-n)\pi}{2t} S(M \pi)
  = \frac{M^2 \pi^2}{8} \left( 1 + O(\frac{1}{t}) \right).
\end{split}
\]
The Zeeman term is similar: set 
$\tilde{S}_{\pm}(y) = y^2/2 \pm y\sin(y) \pm \cos(y) \mp 1$ 
so that $\tilde{S}'_{\pm} = y(1 \pm \cos(y))$,
\[
\begin{split}
  E_z(u) &= \int_0^\infty (1-\cos(u(r))) \; r dr =
  \int_0^{\frac{M-n}{t} \pi} (1 \pm \cos(tr)) \; r dr +
  \int_{\frac{M-n}{t} \pi}^{(\frac{M-n}{t} + M) \pi}
  (1 + \cos(r-\frac{M-n}{t}\pi) ) \; r dr \\
  &= \frac{1}{t^2} \tilde{S}_{\pm}((M-n)\pi) +\tilde{S}_+(M \pi) + 
  \frac{(M-n)\pi}{t} M \pi
  = \frac{M^2 \pi^2}{2} \left( 1 + O(\frac{1}{t}) \right).
\end{split}
\]
Finally the D-M term (which is the source of negative energy):
\[
\begin{split}
  E_{DM}(u) &= \int_0^\infty u_r \sin^2(u) \; r dr =
  2 t E_{a}^{[0, (M-n)\pi/t]} - 2 E_{a}^{[(M-n)\pi/t, (M + (M-n)/t)\pi]} \\
  &= -\frac{M^2 \pi^2}{4} \left( 1 + O(\frac{1}{t}) \right).
\end{split}
\]
So in total
\[
  E_{k,\al,1}(u) = \frac{M^2 \pi^2}{4} \left( 2 - k + \frac{\al}{2}  + 2(1-\al)
  + O( \frac{1}{t} + \frac{t^2}{M^2} ) \right) 
  =  \frac{M^2 \pi^2}{4} \left( 4 - \frac{3\al}{2} - k
  + O( \frac{1}{t} + \frac{t^2}{M^2} ) \right) 
\]
which $\rightarrow -\infty$  as $M \gg t \to \infty$, 
provided $k > 4 - \frac{3\al}{2}$. 

The case $k < -(4 - \frac{3\al}{2})$ is handled by a similar test function
with the slopes reversed:
\[
  u(r) = \left\{ \begin{array}{cl} n\pi - t r & 0 \leq r \leq \frac{M+n}{t} \pi \\
  -(M + \frac{M+n}{t}) \pi + r & \frac{M+n}{t} \pi \leq r \leq (\frac{M+n}{t} + M) \pi \\
  0 & r \geq (\frac{M+n}{t} + M) \pi
  \end{array} \right\} \in X_n.
\]
We omit the details.
\end{proof}

\subsection{Existence of a minimizer}   \label{existence}

Here we prove Theorem~\ref{minexist}.

\begin{proof}
By~\eqref{lb} and~\eqref{nobub}, 
\begin{equation} \label{I1}
  e_{k,\al,1}^{(1)} \in [2(1-k^2), 2) \subset (0, 2).
\end{equation} 
Let $\{ u_j \}_{j=1}^\infty \subset X_1$
be a minimizing sequence: $E_{k,\al,1}(u_j) \to e_{k,\al,1}$.   
Since $|k| < 1$, by~\eqref{basiclower} we have uniform bounds:
$E_e(u_j) \lec 1, \; E^{(\al)}(u_j) \lec 1$. We then also have a uniform 
pointwise bound, $\| u_j \|_\infty \lec 1$ by~\eqref{uniupper}.

A useful observation is that the co-rotational symmetry allows for easy 
uniform control of the $r \to \infty$ tail of minimizing sequences
(c.f. the classical Strauss lemma~\cite{strauss1977existence} giving 
compact embedding of $H^1$ into a Lebesgue space for radial
functions):
\begin{lem}
For $u \in X_n$,
\begin{equation}  \label{Strauss}
  1 - \cos(u(r)) + \frac{1}{2} \sin^2(u(r)) \leq \frac{4}{r} \sqrt{E^{(\al)}(u)} 
   \sqrt{E_e(u)}.
\end{equation}
\end{lem}   
\begin{proof}
Since $1-\cos(u(r)) + \frac{1}{2} \sin^2(u(r)) \to 0$ as $r \to \infty$, 
by the fundamental theorem of calculus, and Cauchy-Schwarz,
\[
\begin{split}
  1 - \cos(u(r)) + \frac{1}{2} \sin^2(u(r)) &= 
  \int_r^\infty \frac{1}{s} \sin(u(s)) (1 - \cos(u(s)) u_r(s) \; s ds \\
  &\leq \frac{2}{r} \| \sin(u) \|_{L^2_{r dr}} \| u_r \|_{L^2_{r dr}}
  \leq \frac{4}{r} \sqrt{E_a(u)} \sqrt{E_e(u)}
  \leq \frac{4}{r} \sqrt{E^{(\al)}(u)} \sqrt{E_e(u)}
\end{split}
\]
where the last inequality is from~\eqref{alpha}.
\end{proof}

It follows from~\eqref{Strauss} (and continuity) that there is $R$ such that
$|u_j(r)| \leq \pi/2$ for all $r \geq R$ and $j$. Then since 
$|u| \lec |\sin(u)|$ when $|u| \leq \pi/2$,
\begin{equation} \label{L2}
  \| u_j \|_{L^2}^2 = 
  \int_0^\infty u_j^2(r) \; rdr = \int_0^R u_j^2(r) \; r dr + \int_R^\infty u_j^2(r) \; r dr
  \lec R^2 \| u_j \|_\infty^2 + E^{(\al)}(u_j) \lec 1.
\end{equation}
From here we also get uniform decay as $r \to \infty$, just as in~\eqref{Strauss}:
\begin{equation} \label{smalltail}
  u_j^2(r) = -2 \int_r^\infty \frac{1}{s} u_j(s) (u_j)_r(s) \;s ds 
  \leq \frac{2}{r} \| u_j \|_2 \|(u_j)_r \|_2 \lec \frac{1}{r}.
\end{equation} 
Combining~\eqref{L2} with $E_e(u) \lec 1$, shows $ \| u_j \|_{H^1_{r dr}} \lec 1$, 
and so by the standard arguments (theorems of Alaoglu and 
Rellich-Kondrachov) 
there is a subsequence (which we still denote $\{ u_j \}$) such that    
\begin{equation} \label{converge}
  \exists \; H^1 \ni v \leftarrow u_j \quad
  \mbox{ weakly in } H^1_{r dr}; \mbox{ strongly in } 
  L^p([0,R))_{r dr} \; \forall \; R, \; 1 \leq p < \infty; \mbox{ and a.e.} 
\end{equation}
Moreover, since $E_e(u_j) \lec 1$, we have $\| \sin(u_j)/r \|_{L^2_{r dr}} \lec 1$
and we may assume (passing to a further subsequence if needed) that
\begin{equation} \label{converge2}
  \frac{\sin(u_j)}{r} \to \frac{\sin(v)}{r} \mbox{ weakly in } L^2_{r dr}.
\end{equation}
Additionally, we have uniform convergence away from the origin. This is because
on compact intervals $I \subset (0,\infty)$, $u_j$ are uniformly bounded 
in $H^1(I)_{dr}$ and so the one-dimensional Sobolev embedding gives
compactness in $L^\infty(I)$. Combining this observation with the small
tail estimate~\eqref{smalltail}:
\begin{equation} \label{uniform} 
  \forall \; R > 0, \quad \lim_{j \to \infty} \; \left( \sup_{r \geq R} |u_j(r) - v(r)| \right) = 0.
\end{equation}

As usual, we have weak lower semicontinuity of 
$\| u_r \|_{L^2}^2$, and so (by Fatou's Lemma) of $E_e$
and $E^{(\al)}$:
\begin{equation} \label{lsc}
  E_e(v) \leq \liminf E_e(u_j), \qquad
  E^{(\al)}(v) \leq \liminf E^{(\al)} (u_j).
\end{equation}
We have full convergence of the DM term, because of the extra
$\sin(u)$ factor, and the uniform decay~\eqref{Strauss}: 
\begin{lem}
\begin{equation} \label{strong}
  E_{DM}(v) = \lim_{j \to \infty} E_{DM}(u_j). 
\end{equation}
\end{lem}
\begin{proof}
By~\eqref{Strauss}, and H\"older, for any $u \in H^1_{r dr}$,
\[
  \int_R^\infty \sin^2(u(r)) |u_r| \; rdr \leq
  \sup_{r \geq R} |\sin(u(r)| \left[ \int_R^\infty \sin^2 u \; r dr \right]^{1/2} 
  \left[ \int_R^\infty u_r^2 \; r dr \right]^{1/2}
  \lec \frac{1}{\sqrt{R}} \sqrt{E^{(\al)}(u)} \sqrt{E_e(u)}.
\]
So given $\e > 0$, choose $R$ such that 
$\int_R^\infty \sin^2(u_j) |(u_j)_r| \; rdr + \int_R^\infty \sin^2(v) |v_r| \; rdr
< \e$. Then
\[
\begin{split}
  \frac{1}{k} [ E_{DM}(u_j) - E_{DM}(v) ] &= 
  \int_0^R (\sin^2(u_j) - \sin^2(v)) (u_j)_r \; rdr +
  \int_0^R \sin^2(v) ((u_j)_r - v_r) \; rdr \\
  & \qquad + \int_R^\infty \left( \sin^2(u_j) (u_j)_r - \sin^2(v) v_r \right) \; r dr
  =: I + II + III.
\end{split}
\]
Since $\sin^2$ is a Lipshitz function, and $u_j \to v$ strongly in $L^2$ on $B_R$, 
$\| \sin^2(u_j) - \sin^2(v) \|_{L^2(B_R)} \to 0$, and so by H\"older,
\[
  | I | \lec \| \sin^2(u_j) - \sin^2(v) \|_{L^2(B_R)} E_e^{1/2}(u_j)
  \to 0.
\]
Weak $H^1$ convergence of $u_j$ shows $II \to 0$. Finally, by choice
of $R$, $|III| < \e$, which was arbitrary.
\end{proof}
Combining~\eqref{lsc} and~\eqref{strong}, we have:
\[
  E_{k,\al,1}(v) \leq \liminf_{j \to \infty} E_{k,\al,1}(u_j) = e_{k,\al,1}^{(1)},
\]
and so it remains to show $v \in X_1$. 

Since $E_e(v) < \infty$, by~\eqref{cont} and~\eqref{smalltail},
$v \in C([0,\infty])$ with $\lim_{r \to \infty} v(r) = 0$, and so
\[
  v \in X_m \mbox{ for some } \;\; m \in \Z.
\]
It remains to show $m=1$.
\begin{lem}
\begin{equation} \label{jump}
  \liminf_{j \to \infty} E_e(u_j) \geq 2|m-1| + E_e(v).
\end{equation}
\end{lem}
\begin{proof}
Set $w_j = u_j - v$ and write $u_j = v + w_j$ so that
\begin{equation} \label{ensplit}
  E_e(u_j) = E_e(v) + E_e(w_j) + 
  \int_0^\infty  v_r (w_j)_r \; r dr + \int_0^\infty \frac{1}{2 r^2}
  (\sin^2(v + w_j) - \sin^2(v) - \sin^2(w_j)) \; r dr.
\end{equation}
By the weak convergence, the first integral on the right $\rightarrow 0$.
For the second integral, we first claim that
\begin{equation} \label{weak}
  \frac{\sin(w_j)}{r} \to 0 \mbox{ weakly in } L^2_{r dr}.
\end{equation}
To see this, note
\[
  \frac{\sin(w_j)}{r} = \cos(v) \left( \frac{\sin(u_j)}{r} - \frac{\sin(v)}{r} \right)
  + \frac{\sin(v)}{r} (\cos(v) - \cos(u_j)).
\]
The first term tends weakly to $0$ by~\eqref{converge2}, and the fact that 
$\cos(v)$ is bounded. The second term tends to $0$ strongly, since
for any $R > 0$,
\[
  \| \frac{\sin(v)}{r} (\cos(v) - \cos(u_j)) \|_2 \leq
  2 \| \frac{\sin(v)}{r} \|_{L^2(B_R)} + \| \cos(u_j) - \cos(v) \|_{L^\infty(B_R^c)}
  \| \frac{\sin(v)}{r} \|_2,
\]
the first term $\to 0$ as $R \to 0$ while the second $\to 0$ (for any fixed $R$)
as $j \to \infty$ by the uniform convergence~\eqref{uniform}.
By trig identities, the second integral on the right of~\eqref{ensplit} can be written
\[
  \int_0^\infty \cos(v) \frac{\sin(v)}{r} \cos(w_j) \frac{\sin(w_j)}{r} \; r dr  -
  \int_0^\infty \frac{\sin^2(v)}{r^2} \sin^2(w_j) \; r dr.
\]
The first integral $\to 0$ by~\eqref{weak} (and boundedness of $\cos(w_j)$
and $\cos(v)$, and $\sin(v)/r \in L^2$), while the second integral is bounded by
\[
  \int_{0}^R \frac{\sin^2(v)}{r^2} \; r dr + \sup_{r \geq R} \sin^2(w_j(r))
  \| \sin(v)/r \|_2,
\]
with the first term $\to 0$ as $R \to 0$, and the second $\to 0$ (for fixed $R$)
as $j \to \infty$ by~\eqref{uniform} (and trig identities).
Thus from~\eqref{ensplit} we see
\[
  \liminf_{j \to \infty} E_e(u_j) = E_e(v) +  \liminf_{j \to \infty} E_e(w_j).
\]
Finally, since $w_j(0) = (1-m) \pi$, the lower bound~\eqref{bog2} 
shows $E_e(w_j) \geq 2|m-1|$, and
the lemma follows.
\end{proof}
Combining~\eqref{nobub}, \eqref{lsc}, \eqref{strong}, \eqref{jump}, 
and~\eqref{lb} we get
\[
\begin{split}
  2 &> e^{(1)}_{k,\al,1}  = \liminf_{j \to \infty} E_{k,\al,1}(u_j) \geq
  2|m-1| + E_{k,\al,1}(v) \\ &\geq 2|m-1| + e_{k,\al,1}^{(m)}
  \geq 2|m-1| + 2(1-k^2)|m| \geq 2|m-1|
\end{split}
\]
from which $m=1$ follows, completing the proof of the theorem.
\end{proof}

\subsection{Monotonicity of the profile}
\label{monotoneprofile}

Here we prove Proposition~\ref{monotonicity}

\begin{proof}
With $\al = 1$ and $\be = 1$, the minimizer 
$v(r) \in C^\infty((0,\infty))$ satisfies the Euler-Lagrange equation~\eqref{E-L} 
\begin{equation}\label{EL}
  -v_{rr} - \frac{1}{r}v_r + \frac{1}{r^2}\sin(v)\cos(v) - \frac{k}{r} \sin^2(v) + \sin(v) \cos(v) = 0 
\end{equation}
and the boundary conditions 
\[
  \lim_{r \to 0+} v(r) = \pi, \quad \lim_{r \to \infty} v(r) = 0.
\] 
And by~\eqref{nobub}, 
\begin{equation} \label{nonbub2}
  E_{k,1,1}(v) = e^{(1)}_{k,1,1} < 2.
\end{equation}
\begin{lem} \label{bounds}
For $k$ sufficiently small, $0 \leq v(r) \le \pi$.
\end{lem}
\begin{proof}
We use proof by contradiction, and a maximum-principle argument.

First suppose $v$ has a local minimum at $r_1 \in (0,\infty)$ 
with $v(r_1) < 0$. Then $v_r(r_1) = 0$ and $v_{rr} (r_1) \ge 0$, and so
\begin{equation}  \label{1}
  (\frac{1}{r_1^2} + 1)\sin(v(r_1))\cos(v(r_1)) - \frac{k}{r_1} \sin^2(v(r_1)) \ge 0.
\end{equation}
Thus $\sin(v(r_1)) \cos(v(r_1)) \ge 0$, implying 
$v(r_1) \le -\frac{\pi}{2}$.
Then by the energy lower bounds~\eqref{basiclower} and~\eqref{topo},
\[
  E_{k,1,1}(v) \ge (1 - k^2) E_e(v) \geq 4(1-k^2) \geq 2 
\]
for $k \leq \frac{1}{\sqrt{2}}$,  contradicting~\eqref{nonbub2}.

On the other hand, suppose $v$ has a local maximum at $r_2 \in (0,\infty)$ with 
$v(r_2) > \pi$. Then $v_r(r_2) = 0$ and $v_{rr}(r_2)\le 0$ so
\begin{equation} \label{2}
  (\frac{1}{r_2^2} + 1)\sin(v(r_2))\cos(v(r_2)) - \frac{k}{r_2} \sin^2(v(r_2)) \le 0.
\end{equation}
If $v(r_2) \ge \frac{3}{2}\pi$, then again
by the lower bounds~\eqref{basiclower} and~\eqref{topo},
\[
  E_{k,1,1}(v) \ge (1 - k^2) E_e(v) \geq 4(1-k^2) \geq 2
\]
for $k \leq \frac{1}{\sqrt{2}}$, contradicting~\eqref{nonbub2}.
So we may assume $\pi < v(r_2) < \frac{3}{2} \pi$ and 
so $\sin(v(r_2)) < 0$ and $\cos(v(r_2)) < 0$. Then by~\eqref{2},
\[
  0 \geq \sin(v(r_2)) \cos (v(r_2)) \left[  
  \frac{1}{r_2^2} + 1 - \frac{k}{r_2} \tan(v(r_2)) \right],
\]
so 
\[
  \tan(v(r_2)) \geq \frac{1}{k} \left(  \frac{1}{r_2} + r_2 \right) \geq \frac{2}{k}
\]
and so $v(r_2) \ge \pi + \arctan(\frac{2}{k})$.
Then again by lower bounds~\eqref{basiclower} and~\eqref{topo},
\[
  E_{k,1,1}(v) \ge (1 - k^2) E_e(v) \geq (1-k^2)
  \left[ 2 \left| \cos\left(\pi + \arctan(\frac{2}{k})\right) + 1 \right| +  2  \right]
  \geq 2(1-k^2) \left[ 2 - \frac{k}{\sqrt{4 + k^2}} \right] \geq 2
\]
provided $k$ is small enough, contradicting~\eqref{nonbub2}.
\end{proof}

\begin{lem} \label{boundsmin}
If $v$ has a local minimum at $r_1 \in (0,\infty)$, then
$0 \le v(r_1) \le \frac{\pi}{2}$.
\end{lem}
\begin{proof}
Inequality~\eqref{1} holds, so  $\sin(v(r_1)) \cos(v(r_1))\ge 0$,
and by Lemma~\ref{bounds}, $0 \leq v(r_1) < \pi$. So 
$0 \leq v(r_1) \leq  \frac{\pi}{2}$.
\end{proof}

\begin{lem} \label{boundsmax}
If $v$ has a local maximum at $r_2 \in (0,\infty)$, 
then $\arctan(\frac{2}{k}) \leq v(r_2) \leq \pi$.
\end{lem}
\begin{proof}
By Lemma~\ref{bounds}, $0 < v(r_2) \leq \pi$.
If $\cos(v(r_2)) \le 0$, then
$\frac{\pi}{2} \leq v(r_2) \leq \pi$.
Otherwise, $\cos(v(r_2)) > 0$ and $\sin(v(r_2)) > 0$, 
and since~\eqref{2} holds,
$\frac{1}{r_{2}^2} + 1 - \frac{k}{r_2} \tan(v(r_2)) \leq 0$.  So
\[
  \tan(v(r_2))  \geq \frac{1}{k} \left( \frac{1}{r_2} + r_2 \right) \geq \frac{2}{k},
\]
and $v(r_2) \ge \arctan(\frac{2}{k})$.
\end{proof}

\begin{lem} \label{bounds3}
For $k$ sufficiently small, if 
$v$ has a local minimum at $r_1 \in (0,\infty)$, then
\begin{equation} \label{sinminimum}
  \sin(v(r_1))  = 1 - \mathcal{O}(k^2). 
\end{equation}
\end{lem}
\begin{proof}
$v$ must have a local maximum at some $r_2 > r_1$. 
By~\eqref{nonbub2} and the lower bounds~\eqref{basiclower} and~\eqref{topo},
 \[ 
   2 > E_{k,1,1}(v) \ge (1 - k^2 )E_e(v) \ge 
   (1 - k^2) [\cos(v(r_1)) + 1 + 2 ( \cos(v(r_1)) - \cos(v(r_2)) ) + 1 - \cos(v(r_1))]
 \]
 so
 \[
   \cos(v(r_1)) - \cos(v(r_2)) < \frac{1}{1-k^2} - 1 = \frac{k^2}{1-k^2}.
\]
By Lemmas~\ref{boundsmin} and~\ref{boundsmax},
\[
  0 \leq \cos(v(r_1)) < \frac{k^2}{1-k^2} + \cos(\arctan(\frac{2}{k}))
  = \frac{k^2}{1-k^2} +  \frac{k}{\sqrt{4 + k^2}} = O(k),
\]
and
\[
  \sin(v(r_1)) = \sqrt{1 - \cos^2(v(r_1))} = 1 - O(k^2).
\]

%
%
%
%
%
 \end{proof}
 
To complete the proof of Proposition~\ref{monotonicity},
we suppose $v$ has a local minimum at $r_1 \in (0,\infty)$,
and derive a contradiction.
By the localized version~\eqref{prepoho2} of the Pohozaev-type identity,
\begin{equation}  \label{poho2}
   2 E_a^{[0, r_1]}(v) + k E_{DM}^{[0,r_1]}(v) = \frac{1}{2} 
   (1 + r_1^2) \sin^2(v(r_1)),
 \end{equation}
since $v_r(r_1) = 0$.
By Lemma~\ref{boundsmin}, $\cos(u(r_1)) \ge 0$. So
using~\eqref{poho2},
\[
\begin{split}
  E_{k,1,1}^{[0,r_1]}(v) &= E_e^{[0,r_1]}(v) + 
  \frac{1}{2} (1 + r_1^2) \sin^2(v(r_1)) - E_a^{[0, r_1]}(v) \\ &\geq
  1 + \cos(v(r_1)) + \frac{1}{2} (1 + r_1^2) \sin^2(v(r_1)) - \frac{1}{4} r_1^2
 \end{split}
 \]
where we used~\eqref{topo} and
$E_a^{[0,r_1]}(v) \leq \frac{1}{2} \int_0^{r_1} r dr = \frac{1}{4} r_1^2$.
Then using again~\eqref{basiclower} and~\eqref{topo},
and that $\sin(v(r_1)) = 1- \mathcal{O}(k^2)$  by Lemma~\ref{bounds3}, 
\[
\begin{split}
  E_{k,1,1}(v) &= E_{k,1,1}^{[0,r_1]}(v) + E_{k,1,1}^{[r_1,\infty)}(v) \\
  &\ge 
  1 + \cos(v(r_1)) + \frac{1}{2} (1 + r_1^2) \sin^2(v(r_1)) - \frac{1}{4} r_1^2
  + (1-k^2)(1 - \cos(v(r_1)) \\
  &\ge 2 - k^2 + \frac{1}{2} (1 + r_1^2)(1 - O(k^2)) - \frac{1}{4} r_1^2 \\
  &\ge \frac{5}{2} - \mathcal{O}(k^2)  > 2 
\end{split}
\]
for $k$ sufficiently small, contradicting~\eqref{nonbub2}.
\end{proof}

\subsection{Resolvent estimates}
\label{sec:resolvent}

\subsubsection{Free resolvent estimates}
\label{freeresolvent}

Here we record some simple estimates for the free resolvent 
$R_0(\be) = (-\Delta_r + \frac{1}{r^2})^{-1}$,
which include~\eqref{p=1}, \eqref{p=0}, and~\eqref{p=-1}:
\begin{lem} \label{freeest}
For $f$ radial:
\begin{equation}  \label{freeest1}
  \|R_0(\be)  f \|_X \lec \be^{p-1} \| r^p f \|_{L^2}, \quad 0 \leq p \leq 1;
\end{equation}
\begin{equation}  \label{freeest2}
  \| R_0(\be)  f \|_X \lec \frac{1}{\be^2} \| f \|_{X} ;
\end{equation}
\begin{equation}  \label{freeest3}
  \| R_0(\be)  f \|_X \lec \be^{p-1} \left( \| r^{p+1} f_r \|_{L^2} + \| r^p f \|_{L^2}  \right),
  \quad -1 \leq p \leq 0.
\end{equation}
\end{lem}
\begin{proof}
These estimates follow from an elementary relation for functions on $\R^2$: for 
\[
  g(x), \; k(x), \mbox{ with } \; g = (-\Delta + \be^2)^{-1} k,
\]
\begin{equation}  \label{ibp}
  ( g, \; k) = ( g, \; -\Delta g) + \be^2 (g, \; g) = \| \nabla g \|_{L^2}^2 + \be^2 \| g \|_{L^2}^2.
\end{equation}
In particular,
\[
  \| \nabla g \|_{L^2}^2 \leq \left| \langle \frac{g}{|x|}, \; |x| k \rangle \right|
  \leq \left \| \frac{g}{|x|} \right\|_{L^2} \| |x| k \|_{L^2}. 
\] 
Then taking 
\[
  k(x) = e^{i \theta} f(r), \mbox{ so } \; g(x) = e^{i \theta} R_0(\be)  f,
\]
\[
  \| R_0(\be)  f \|_X^2 \lec \| \nabla g \|_{L^2}^2 
  \leq  \left \| \frac{1}{|x|} R_0(\be)  f \right\|_{L^2} 
  \| r f \|_{L^2} \leq \| R_0(\be)  f \|_X \| r f \|_{L^2},
\]
from which the $p=1$ case of~\eqref{freeest1} follows.
Similarly from~\eqref{ibp},
\[
  \be^2 \| g \|_{L^2}^2 + \| \nabla g \|_{L^2}^2 \leq |(g, k)| \leq \| g \|_{L^2} \| k \|_{L^2},
  \quad \implies \quad
  \| g \|_{L^2} \leq \frac{1}{\be^2} \| k \|_{L^2}
\]
and then,
\[
  \| R_0(\be)  f \|_X^2 \lec   \| \nabla g \|_{L^2}^2  
  \leq \frac{1}{\be^2} \| f \|_{L^2}^2
\]
from which the $p=0$ case of~\eqref{freeest1} follows.
The $0 < p < 1$ cases of~\eqref{freeest1} then follow from elementary interpolation.

For~\eqref{freeest2}, replacing $g$ by $\nabla g$ and $k$ by $\nabla k$ in~\eqref{ibp}
(since $\nabla$ commutes with $-\Delta +\be^2$) yields
\[
  \be^2 \| \nabla g \|_{L^2}^2 \leq 
  |\langle \nabla g, \nabla k \rangle | \leq \| \nabla g \|_{L^2} \| \nabla k \|_{L^2},
\]  
from which~\eqref{freeest2} follows.

Bound~\eqref{freeest3} follows from interpolation between the other two bounds.
\end{proof}

%

\subsubsection{Inner products with the free resolvent}
\label{freeinners}

First we prove Lemma~\ref{key}.
\begin{proof}
We use the following Fourier-Bessel transformation: for a real-valued, 
radial function $f$,
\[
  \tilde f(\rho) := 
  -i \widehat{e^{i \th} f(r)}(\xi)|_{\xi = (\rho,0)} = 
  \int_0^\infty J_1(\rho r) f(r) \; r \; dr,
\]
where $J_1$ is the order-$1$ Bessel function of the first kind. In particular 
\[
  ((-\Delta_r + \frac{1}{r^2} ) f)^{\tilde{}} (\rho) = \rho^2 \tilde f(\rho), \qquad 
  (R_0(\be)  f )^{\tilde{}} (\rho) = \frac{ \tilde f(\rho) }{\rho^2 + \be^2}.
\]
Then by Plancherel,
\begin{equation} \label{FBsplit}
  \langle g , R_0(\be)  h \rangle =
  \langle \tilde g , \frac{\tilde{h}}{\rho^2 + \be^2} \rangle =
  \int_0^1 \frac{\tilde g \tilde{h}}{\rho^2 + \be^2} \rho d \rho +
  \int_1^\infty \frac{ \tilde g \tilde{h}}{\rho^2 + \be^2} \rho d \rho.
\end{equation}
By standard estimates for the Fourier transform,
\begin{equation}  \label{g0}
  \| R_0(\be) g \|_{L^\infty} \lec \| g \|_{L^1},
\end{equation}
and
\begin{equation}  \label{highrho}
  h \in X \; \implies \; \rho \tilde h \in L^2,
\end{equation}
and we can easily bound the large $\rho$ integral in~\eqref{FBsplit}:
\begin{equation}  \label{largerho}
  \left| \int_1^\infty \frac{ \tilde g \tilde{h}}{\rho^2 + \be^2} \rho d \rho \right|
  \leq \int_1^\infty \frac{ |\tilde g | | \tilde h| }{\rho^2} \rho d \rho  
  \leq \| \tilde g \|_{L^\infty} \| \rho \tilde h \|_{L^2} \left( \int_1^\infty 
  \frac{1}{\rho^6} \rho d \rho \right)^\frac{1}{2} \lec \| g \|_{L^1}.
\end{equation}

Now we consider the small $\rho$ contribution. 
Suppose first $r^q g \in L^\infty$ for some $q > 3$.
We assume $q < 4$ (which would imply
the desired result for any higher $q$).
Choosing $0 < \e < q-3$,  using the pointwise estimate 
\[  
  |J_1(y) - \frac{1}{2} y| \lec y^{q-2-\e}
\] 
from Taylor's theorem, and setting
\[
  c := \int_0^\infty r g \; r dr, \qquad
  |c| \leq \| r^q g \|_{L^\infty} \int_0^\infty r^{1-q} \; r dr
  \lec \| g \|_{r^{-q} L^\infty},
\]
we have
\begin{equation}  \label{g1}
\begin{split}
  \left| \tilde g(\rho) -  \frac{c}{2} \rho \right| 
  &= \left| \int_0^\infty \left( J_1(\rho r) -  \frac{1}{2} \rho r \right) g(r) \; r dr  \right|
  \lec  \rho^{q-2-\e} \int_0^\infty r^{q-2-\e} |g(r)| \; r dr \\
  &\lec \rho^{q-2-\e} \left( \int_0^1 |g(r)| \; r dr + \| r^q g \|_{L^\infty} 
  \int_1^\infty r^{-2 - \e} \; r dr \right) \\
  &\lec \rho^{q-2-\e} \| g \|_{L^1 \cap r^{-q} L^\infty}.
\end{split}
\end{equation}

Now write 
\[
  h = 2 \chi_{\geq 1}(r) \frac{1}{r} + h^\#,
\]
and note that
\[
  h^\# := h - 2 \chi_{\geq 1}(r) \frac{1}{r} \; \in L^1 \cap L^2 
  \quad \implies \quad \tilde h^\# \; \in L^2 \cap L^\infty.
\] 
Moreover,
\[
  \left( \chi_{\geq 1}(r) \frac{1}{r} \right)^{\Tilde{ }}(\rho) = 
  \int_1^\infty J_1(\rho r) dr = \frac{1}{\rho} \int_\rho^\infty J_1(y) dy, 
\]
and since $\int_0^\infty J_1(y) dy = 1$, and $|J_1(y)| \lec y$, 
\[
  \left|  \left( \chi_{\geq 1}(r) \frac{1}{r} \right)^{\tilde{ }}(\rho) 
  - \frac{1}{\rho} \right| \lec \rho.
\]
Then
\begin{equation} \label{h1}
  | \tilde h (\rho) - \frac{2}{\rho} | \leq
  2  \left|  \left( \chi_{\geq 1}(r) \frac{1}{r} \right)^{\tilde{ }}(\rho) 
  - \frac{1}{\rho} \right| + | \tilde h^\# (\rho) | \lec 1
   \;\; \mbox{ for } \rho \leq 1,
\end{equation}
and in particular
\begin{equation} \label{h2} 
  | \tilde h(\rho) | \lec \frac{1}{\rho}  \;\; \mbox{ for } \rho \leq 1.
\end{equation}

We can now use estimates~\eqref{g1},~\eqref{h1} and~\eqref{h2}
in the small $\rho$ intergal of~\eqref{FBsplit} to get
\[
\begin{split}
  &\left| \int_0^1 \frac{\tilde g \tilde{h}}{\rho^2 + \be^2} \rho d \rho 
  - c \int_0^1 \frac{\rho \; d \rho}{\rho^2 + \be^2}  \right|
   \lec   \int_0^1 \frac{\rho \; d \rho}{\rho^2}
  \left( |\tilde g(\rho) - \frac{c}{2} \rho| |\tilde h(\rho)| + \frac{|c|}{2} \rho 
  |\tilde h(\rho) - \frac{2}{\rho} |  \right) \\
  & \qquad \qquad \qquad \qquad \lec  \int_0^1 \frac{\rho \; d \rho}{\rho^2}
  \left( \rho^{q-2-\e} \frac{1}{\rho} + \rho \right)
  \| g \|_{L^1 \cap r^{-q} L^\infty} \lec \| g \|_{L^1 \cap r^{-q} L^\infty} 
\end{split}
\]
Then since
\[
  \int_0^1 \frac{\rho d \rho}{\rho^2 + \be^2} = \log \left( \frac{1}{\be} \right) + O(\be^2),
\]
and recalling~\eqref{FBsplit} and~\eqref{largerho}, we have established~\eqref{choice}.

The bound~\eqref{innerbound} is an immediate consequence of~\eqref{choice},
while~\eqref{innercomp} follows from the computation, using~\eqref{h3/r}:
\[
  \int_0^\infty r W(r) h(r) = 2 \int_0^\infty \frac{h^3}{r} \; r dr = 4. 
\]

Finally, we turn to~\eqref{innerbound2}. So suppose now $r^3 g \in L^\infty$.
Since $|J_1(y)| \lec \min (y, \frac{1}{\sqrt{y}})$, we have, for $\rho \leq 1$,
\[
\begin{split}
  |\tilde g(\rho)| &= \left| \int_0^\infty J_1(r \rho) g(r) \; r dr \right| \lec
  \int_0^1 \rho r |g(r)| \; r dr + \int_1^{\frac{1}{\rho}} \rho r |g(r)| \; r dr 
  + \int_{\frac{1}{\rho}}^\infty \frac{1}{\sqrt{\rho}} \frac{1}{\sqrt{r}} |g(r)| \; r dr \\
  &\lec \rho \| g \|_{L^1} + \rho \| r^3 g \|_{L^\infty} \int_1^{\frac{1}{\rho}} \frac{dr}{r}
  + \frac{1}{\sqrt{\rho}} \| r^3 g \|_{L^\infty} \int_{\frac{1}{\rho}}^\infty 
  \frac{dr}{r^{\frac{5}{2}}}  \\
  &\lec \| g \|_{L^1} \; \rho + \| r^3 g \|_{L^\infty} \; \rho \log \left( \frac{1}{\rho} \right)
  + \| r^3 g \|_{L^\infty} \; \rho \lec \| g \|_{L^1 \cap r^{-3} L^\infty} \; 
  \rho \log \left( \frac{1}{\rho} \right).
\end{split}
\]
Using this, together with~\eqref{h2}, in the small $\rho$ contribution to~\eqref{FBsplit}
gives
\[
  \left| \int_0^1 \frac{\tilde{g}(\rho) \tilde{h}(\rho)}{\rho^2 + \be^2} \rho d \rho \right|
  \lec  \| g \|_{L^1 \cap r^{-3} L^\infty} 
  \int_0^1 \frac{ \log \left( \frac{1}{\rho} \right) }{\rho^2 + \be^2} \rho d \rho.
\]
By change of variable $\rho = \be y$,
\[
  \int_0^1 \frac{ \log \left( \frac{1}{\rho} \right) }{\rho^2 + \be^2} \rho d \rho
  = \log \left( \frac{1}{\be} \right) \int_0^{\frac{1}{\be}} \frac{y \; dy}{y^2 + 1}
  + \int_0^{\frac{1}{\be}} \log\left( \frac{1}{y} \right)  \frac{y \; dy}{y^2 + 1}
  \lec \log^2 \left( \frac{1}{\be} \right)
\]
which, together with ~\eqref{FBsplit} and~\eqref{largerho}, 
establishes~\eqref{innerbound2}.
\end{proof}

Next we prove the monotonicity estimate Lemma~\ref{monotone}.
\begin{proof}
We may assume $\be^{(2)} \lec \be^{(1)}$.
We have, by resolvent identity,
\begin{equation}  \label{mono1}
\begin{split}
  \gamma(\be^{(1)}) &- \gamma(\be^{(2)}) =
  \left( \be^{(1)} - \be^{(2)} \right) \left( R_0(\be^{(1)}) W h, h \right) +
  \be^{(2)} \left( (\be^{(1)})^2 - (\be^{(2)})^2 \right) \left ( R_0(\be^{(2)})
  R_0(\be^{(1)}) W h, h \right) \\
  &=  \left( \be^{(1)} - \be^{(2)} \right)\left( 4 \log\left( \frac{1}{\be^{(1)}} \right)
  + O(1)  +  \be^{(2)} \left( \be^{(1)} + \be^{(2)} \right) \left(  
  R_0(\be^{(1)}) W h, R_0(\be^{(2)}) h \right) \right)
\end{split}
\end{equation}
using~\eqref{innercomp}.
Using the Fourier-Bessel transform,
\[
  \left( R_0(\be^{(1)}) W h, R_0(\be^{(2)}) h \right) 
  = \int_0^\infty \frac{(Wh)^{\tilde{}}(\rho) \tilde h(\rho) \; \rho \; d \rho}
  {\left(\rho^2 + (\be^{(1)})^2 \right) \left(\rho^2 + (\be^{(2)})^2 \right)} 
\]
Since $r^5 Wh \in L^\infty$, we have $|Wh^{\tilde{}}(\rho)| \lec \rho$
from~\eqref{g0} and~\eqref{g1}. And since 
$\rho \tilde h(\rho) = Wh^{\tilde{}}(\rho)$, we have
$|\tilde h(\rho)| \leq 1/\rho$. Using these above, we have
\[
  \left| \left( R_0(\be^{(1)}) W h, R_0(\be^{(2)}) h \right) \right|
  \lec \int_0^\infty \frac{\rho \; d \rho}
  {\left(\rho^2 + (\be^{(1)})^2 \right) \left(\rho^2 + (\be^{(2)})^2 \right)}.
\]
Using $\rho^2 + (\be^{(2)})^2  \gec \rho \be^{(2)}$ and the change of variable
$\rho = \be^{(1)} y$, we get
\[
  \left| \left( R_0(\be^{(1)}) W h, R_0(\be^{(2)}) h \right) \right|
  \lec \frac{1}{\be^{(1)}}  \frac{1}{\be^{(2)}} \int_0^\infty \frac{dy}{y^2 + 1}
  \lec \frac{1}{\be^{(1)}}  \frac{1}{\be^{(2)}}. 
\]
Inserting this in~\eqref{mono1} we see
\[
  \gamma(\be^{(1)}) - \gamma(\be^{(2)}) =
  \left( \be^{(1)} - \be^{(2)} \right)\left( 4 \log\left( \frac{1}{\be^{(1)}} \right)
  + O(1)  + O \left(\frac{\be^{(2)}}{\be^{(1)}} \right) \right)
  \gec \left( \be^{(1)} - \be^{(2)} \right) \log\left( \frac{1}{\be^{(1)}} \right).
\]
\end{proof}

\subsubsection{Refined free resolvent estimates}
\label{refinement}

First we prove the refined version~\eqref{refined} of the 
$p=0$ case of~\eqref{freeest1} for $f = h \not \in L^2$.
\begin{proof}
Using~\eqref{h2} and~\eqref{highrho},
\[
\begin{split}
  \| R_0(\be) h \|_X &\lec \| \rho \left( R_0(\be) h \right)^{\Tilde{}} \|_{L^2}
  = \left\| \frac{\rho}{\rho^2 + \be^2} \tilde h \right\|_{L^2}
  \lec \left\| \frac{1}{\rho^2 + \be^2} \right\|_{L^2_{\leq 1}} 
  + \left\| \frac{1}{\rho^2 + \be^2} \right\|_{L^\infty_{\geq 1}} \\
  &= \left( \int_0^1 \frac{ \rho d \rho}{(\rho^2 + \be^2)^2} \right)^{\frac{1}{2}}
  + \frac{1}{1 + \be^2} \lec \frac{1}{\be}.
\end{split}
\]
\end{proof}

Next we prove the $L^2$ estimate~\eqref{refined2} for the free
resolvent acting on well-localized functions. 
\begin{proof}
Using the basic Fourier-Bessel estimates 
\[
  |J_1(y)| \lec 1 \; \implies \; |\tilde f (\rho)| \lec \| f \|_{L^1}, \qquad  
  |J_1(y)| \lec y \; \implies \; |\tilde f(\rho)| \lec \rho \| r f \|_{L^1},
\]
we have
\[
\begin{split}
  \| R_0(\be) f \|_{L^2}^2 &= \left \| \frac{\tilde f}{ \rho^2 + \be^2 } \right\|_{L^2}^2
  \lec \| r f \|_{L^1}^2 \int_0^1 \frac{\rho^2\; \rho \; d \rho}{(\rho^2 + \be^2)^2} +
  \| f \|_{L^1}^2 \int_1^\infty \frac{ \rho \; d \rho}{(\rho^2 + \be^2)^2} \\
  &\lec \log \left( \frac{1}{\be} \right) \| r f \|_{L^1}^2 + \| f \|_{L^1}^2,
\end{split}  
\]
and~\eqref{refined2} follows.
\end{proof}

\subsubsection{Estimates for $(H + \be^2)^{-1}$}
\label{fullresolvent}

Since $H + \be^2 = -\tilde \Delta + \be^2 - W$,
\begin{equation}  \label{factor}
  (H + \be^2)^{-1} = (-\tilde \Delta + \be^2 - W)^{-1}
  = \left((-\tilde \Delta + \be^2)(I -R_0(\be) W) \right)^{-1}
  = (I - R_0(\be) W)^{-1} R_0(\be)
\end{equation}
where
\[
   R_0(\be)  = (-\tilde \Delta + \be^2)^{-1} , \qquad \tilde \Delta = \Delta_r - \frac{1}{r^2}
\]
denotes the free resolvent. 

Let us consider first the $\be \to 0$ limit of $I - R_0(\be) W$ 
\[
  (I - R_0(\be)  W) |_{\be = 0} = I - G_0 W
\]
where $ G_0 = (-\tilde \Delta)^{-1}$ has the explicit form 
\[
   G_0 f (r) = \frac{1}{2r} \int_0^r s^2 f(s) ds + \frac{r}{2} \int_r^\infty f(s) ds =
  \int_0^\infty \frac{1}{2} \min \left(\frac{s}{r}, \frac{r}{s}\right) \; f(s) \; s ds.  
\]
We summarize its mapping properties on the space $X$.
For this purpose, it is convenient to consider $X$ as a Hilbert space
with inner-product
\[
  \langle f, \; g \rangle_X := \langle f_r, \; g_r\rangle + \langle \frac{1}{r} f, \; \frac{1}{r} g \rangle 
  = \int_0^\infty \left( f_r(r) g_r(r) + \frac{1}{r^2} f(r) g(r) \right) r dr, 
\]
and to identify the dual space $X^*$ with $X$ via this inner-product. 
\begin{lem} \label{mapping}
We have:
\begin{enumerate}
\item
$I -  G_0 W : X \to X$ is Fredholm;
\item
$I -  G_0 W$ is self-adjoint (with respect to $\langle \cdot, \; \cdot \rangle_X$);
\item
$\ker(I - G_0 W) = \langle h \rangle$;
\item
$ran(I - G_0 W) = X \cap (Wh)^{\perp} := \{ f \in X \; | \; \langle Wh, \; f \rangle = 0 \}$.
\item
$I -  G_0 W : X \cap (r W)^{\perp} \to X \cap (W h)^{\perp}$
is bijective;
\item
$\| r^{p+1} \p_r  G_0 W f \|_{L^2} + \| r^p  G_0 W f \|_{L^2} \lec \| f \|_X,
\quad -2 < p < 0$.
\end{enumerate}
\end{lem}
\begin{proof}
\begin{enumerate}
\item
The first statement will follow from the fact that $ G_0 W: r L^2 \to X$ is
compact. We first show $ G_0 W : r L^2 \to r L^2$ is compact. Equivalently, we show
$\frac{1}{r} G_0 W r : L^2 \to L^2$ is compact. Indeed, it is Hilbert-Schmidt,
since its integral kernel is square integrable:
\[
\begin{split}
  \int_0^\infty s ds \int_0^\infty r dr 
  \left(\frac{1}{r} \frac{1}{2}   \min \left(\frac{s}{r}, \frac{r}{s} \right)
  W(s) s  \right)^2 &= \frac{1}{4}
  \int_0^\infty W^2(s) s^2 \; s ds \left( \frac{1}{s^2} \int_0^s r dr + s^2 \int_s^\infty \frac{1}{r^4} r dr
  \right) \\
  & = \frac{1}{4} \int_0^\infty W^2(s) s^2 \; s ds < \infty.
\end{split}   
\]  
A very similar calculation shows that the operator
\[
  \p_r G_0 W:  f (r) \mapsto  -\frac{1}{2 r^2} \int_0^r s W(s) f(s) \; s ds + 
  \frac{1}{2} \int_r^\infty \frac{1}{s} W(s) f(s) \; s ds 
\]
is compact (in fact Hilbert-Schmidt) from $r L^2$ to $L^2$.
\item
Compute, for $f$, $g$ $\in C_0^\infty((0,\infty))$:
\[
\begin{split}
  \langle (I - G_0 W) f, \; g \rangle_X &=
  \langle (I -  G_0 W) f, \; -\tilde \Delta g \rangle =
  \langle f, \; (I - W  G_0 (-\tilde \Delta) g \rangle \\
  &= \langle f, \; -\tilde \Delta (I -  G_0 W) g \rangle 
  = \langle f, \; (I -  G_0 W) g \rangle_X.
\end{split}
\]
\item
This is a computation.
Clearly $ h \in \ker(I -  G_0 W)$, since
\[
  (I -  G_0 W) h = h - h = 0.
\]
Conversely, if $X \ni f =  G_0 W f$, then, since 
$X \subset C((0,\infty))$, 
$f \in C^2((0,\infty))$, and then
\[
  0 = -\tilde \Delta f + W f = H f.
\] 
Since $h$ and $k$ are independent solutions of this second-order ODE,
while $k \not\in X$, it follows that $f \in \langle h \rangle$. 
\item
Since $I -  G_0 W$ is Fredholm, its range is the orthogonal 
complement (with respect to $\langle \cdot, \; \cdot \rangle_X$) of the
kernel of its adjoint. Since
$\ker(I -  G_0 W)^* = \ker(I -  G_0 W) = \langle h \rangle$,
\[
  ran(I -  G_0 W) = \{ f \in X \; | \; 0 = \langle h, \; f \rangle_X =
  \langle -\tilde \Delta h, \; f\rangle  = \langle Wh, \; f\rangle \; \}.
\]
\item
Since we have identified the kernel and range, the last statement follows
from the facts that $r W \in X$, and
\begin{equation}  \label{nonzero}
  \langle r W, \; h \rangle = \int_0^\infty r \frac{h^2(r)}{r^2} h(r) \; r dr =
  \int_0^\infty \frac{h^3}{r} \; r dr = 2 \not= 0.
\end{equation}
\item
Using the explicit form of $ G_0$, and $W(r) \lec (1 + r^2)^{-2}$,
we have
\[
  r \leq 1 \; \implies \; 
  | G_0 W f| \lec \frac{1}{r} \| f \|_{L^\infty} \left( \int_0^r s^2 ds \right)
  + r \| f \|_{L^\infty} \left( \int_0^\infty W(s) ds \right) 
  \lec r \| f \|_{L^\infty},
\] 
while
\[
  r \geq 1 \; \implies \;
   | G_0 W f| \lec \frac{1}{r} \| f \|_{L^\infty} \left( \int_0^\infty s^2 W(s) ds \right)
   + r \| f \|_{L^\infty} \left( \int_r^\infty \frac{1}{s^4} ds \right) \lec \frac{1}{r} \| f \|_{L^\infty},
\]
and so for $-2 < p < 0$,
\[
  \| r^p  G_0 W f \|_{L^2} \lec \| r^p r (1 + r^2)^{-1} \|_{L^2} \| f \|_{L^\infty}
  \lec \| f \|_{L^\infty} \lec \| f \|_X.
\] 
A similar calculation gives, again for $-2 < p < 0$,
\[
  \| r^{p+1} \p_r  G_0 W \|_{L^2} \lec 
  \| r^{p+1} (1 + r^2)^{-1} \|_{L^2} \| f \|_{L^\infty} \lec \| f \|_{L^\infty} \lec \| f \|_X.
\] 
\end{enumerate}
\end{proof}
We will therefore denote by $(I - G_0 W)^{-1}$ the (bounded)
inverse operator
\begin{equation}  \label{invdef}
  (I -  G_0 W)^{-1} : X \cap (hW)^{\perp} \to X \cap (r W)^{\perp}.
\end{equation}
The particular choice $(r W)^{\perp}$ of subspace for the 
range of the inverse is what allows for estimate~\eqref{resest1} below.
 
We return now to consider $I - R_0(\be)  W$ for $\be > 0$.
We first confirm that it is indeed invertible on $X$ -- essentially just 
by manipulating the factorization~\eqref{factor}, and exploiting the
fact that $H \geq 0$, and hence $H + \be^2$ is invertible:  
\begin{lem}
For $\be > 0$, $I - R_{0}(\be) W : X \to X$
is bounded and bijective, hence 
\[
  (I - R_0(\be)  W)^{-1} : X \to X
\]
exists (and is bounded).
\end{lem} 
\begin{proof}
Boundedness follows from
\[
  f \in X \; \implies \; \| R_0(\be)  W f \|_X \lec \| r W f \|_{L^2} \lec \| f \|_{L^\infty}
  \lec \| f \|_X, 
\] 
using~\eqref{freeest1} with $p=1$.

For surjectivity: given $f \in X$, set
\[
  g := f + (H+\be^2)^{-1} W f.
\]
As a self-adjoint operator on $L^2$, $H \geq 0$, and so
$(H + \be^2)^{-1} : L^2 \to D(H) \subset X$ boundedly. 
Since $f \in X \subset L^\infty$, $W f \in L^2$, and hence $g \in X$.
And just compute
\[
\begin{split}
  (1 - R_0(\be)  W) g &= f - R_0(\be)  W f + (H + \be^2)^{-1} W f
  - R_0(\be)  (-H - \be^2 - \tilde \Delta + \be^2)(H + \be^2)^{-1} W f \\
  &=  f - R_0(\be)  W f + (H + \be^2)^{-1} W f + R_0(\be)  W f
  - (H + \be^2)^{-1} W f = f.
\end{split}
\]

For injectivity: if $f \in X$ satisfies $(1 - R_0(\be)  W) f = 0$, then
since $W f \in L^2$,
$f = R_0(\be)  W f \; \in D(R_0(\be) ) = D(H)$, and
\[
  0 = f - R_0(\be)  ( - H - \be^2 - \tilde \Delta + \be^2) f = R_0(\be) (H + \be^2) f,
\]
and so $f = 0$ by the invertibility of $ R_{0}(\be)$ and $H + \be^2$.
\end{proof}

The key result is:
\begin{prop} \label{bound}
If $f \in X \cap (Wh)^{\perp}$, then for $\be$ sufficiently small,
\begin{equation}  \label{resest1}
  \| (1 - R_0(\be)  W)^{-1} f - (1 - G_0 W)^{-1} f \|_X \lec 
  \frac{1}{\log\left( \frac{1}{\be} \right)} \| f \|_X,
\end{equation}
and so in particular
\begin{equation}  \label{resest2}
  \| (1 - R_0(\be)  W)^{-1} f \|_X \lec \| f \|_X
\end{equation}
(uniformly in $\be$), and we have~\eqref{basic}:
\[
  g \perp R_0(\be) Wh \;\; \implies \;\;
  \| (H + \be^2)^{-1} g \|_X \lec \| R_0(\be)  g \|_{X}.
\]
\end{prop}
\begin{proof}
Set
\[
  g := (1 - R_0(\be)  W)^{-1} f - (1 -  G_0 W)^{-1} f \; \in X
\]
so that, using a resolvent identity,
\[
\begin{split}
  (1 -  G_0 W) g &= (1 - R_0(\be)  W + (R_0(\be)  - G_0) W )
  (1 - R_0(\be)  W)^{-1} f - f \\ &=
  - \be^2 R_0(\be)  G_0 W (1 - R_0(\be)  W )^{-1} f \\
  &= - \be^2 R_0(\be)  G_0 W \left( g +  (1 - G_0 W)^{-1} f \right).
\end{split}
\]
Evidently, the right hand side lies in 
$Ran (1 -  G_0 W ) = X \cap (W h)^{\perp}$, so using~\eqref{invdef}
and Lemma~\ref{mapping},
\begin{equation}  \label{g}
  g = -\be^2 (1 -  G_0 W)^{-1} 
  R_0(\be)  G_0 W \left( g +  (1 -  G_0 W)^{-1} f \right)
  + \mu \; h
\end{equation}
for some $\mu \in \R$, and so
\[
  \| g \|_X \lec \be^2 
  \|  R_0(\be)   G_0 W \left( g +  (1 -  G_0 W)^{-1} f \right) \|_X
  + |\mu|.
\]
Then using~\eqref{freeest3} for any $-1 < p < 0$, the last statement of
Lemma~\ref{mapping}, and~\eqref{invdef},
\[
  \| g \|_X \lec \be^2 \; \be^{p-1}  \left( \| g \|_X + \| f \|_X \right) + |\mu|
\]
and so for $\be$ small enough,
\begin{equation}  \label{g2}
  \| g \|_X \lec \be^{1+p} \| f \|_X + |\mu|.
\end{equation}
Finally, we bound $\mu$ by taking the ($L^2$) inner-product of \eqref{g}
with $r W$, and using~\eqref{invdef}, 
\[
  \langle rW, \; h \rangle \; \mu = -\langle rW, \; g\rangle = -\langle rW, \; (1 - R_0(\be)  W)^{-1} f \rangle 
\]
and so by~\eqref{nonzero},
\[
  | \mu | \lec |(r W, \; \hat g)|, \qquad 
  \hat g := (1 - R_0(\be)  W)^{-1} f  = g + (1 -  G_0 W)^{-1} f.
\]
Now since $f \in X \cap (W h)^{\perp}$, using the relation
$ G_0 (W h) = h$, a resolvent identity, and estimate~\eqref{choice},
\[
\begin{split}
  0 &= \langle  Wh, \; f \rangle = \langle Wh, \; (1 - R_0(\be)  W) \hat g \rangle  =
  \langle (1 - W R_0(\be) ) Wh, \; \hat g \rangle \\
  &= \langle (1 - W G_0) Wh + W(  G_0 - R_0(\be)  ) W h, \; \hat g \rangle  \\
  &= \be^2 \langle  W R_0(\be)  G_0 W h, \; \hat g \rangle
  = \be^2 \langle  R_0(\be)  h, \; W \hat{g}\rangle \\
  &= \be^2 \left( \langle r W, \; \hat g \rangle \log \left( \frac{1}{\be} \right) 
  + O(\| \hat g \|_{L^\infty})   \right)
\end{split}
\]
from which follows
\[
  | \mu| \lec |\langle r W, \; \hat g\rangle | \lec \frac{1}{\log(1/\be)} \| \hat g \|_{X}
  \lec \frac{1}{\log(1/\be)} \left( \| g \|_{X} + \| f \|_X \right) .
\]
Returning to~\eqref{g2}, we have (for any $-1 < p < 0$),
\[
  \| g \|_X \lec \be^{1 + p} \| f \|_X + \frac{1}{\log(1/\be)} \left( \| g \|_{X} + \| f \|_X \right),
\]
and~\eqref{resest1} follows.

Finally,~\eqref{resest2} is a consequence of~\eqref{resest1} and~\eqref{invdef},
and then~\eqref{basic} follows from the factorization~\eqref{factor}.
\end{proof}

\subsubsection{Estimates for $(\hat H + \be^2)^{-1}$}
\label{fullresolvent'}

Here we prove the estimates for the resolvent  $(\hat H + \be^2)^{-1}$
used in Section~\ref{regests} for the higher-regularity estimate.

First we prove Proposition~\ref{L1}, by comparing $(\hat H + \be^2)^{-1}$
to $\hat H^{-1}$.
\begin{proof}  
Since $\hat H \geq 0$, 
$\hat H + \be^2$ is an invertible operator on $L^2$, so
\[
  \zeta := (\hat H + \be^2)^{-1} f \in D(\hat H) \subset L^2
\]
exists, and
\begin{equation} \label{hatH2}
  (\hat{H} + \beta^2) \zeta =  f.
\end{equation}
Note that by second-order ODE regularity,
$f \in C((0,\infty)) \; \implies \; \zeta \in C^2((0,\infty))$,
so~\eqref{hatH2} holds for all $r \in (0,\infty)$, and
$D(\hat H) \subset X \subset L^\infty$ shows that 
$\zeta$ is also bounded.

We will show~\eqref{L1est} by comparing $(\hat H + \be^2)^{-1}$ to 
$\hat H^{-1}$, which has a simple explicit form. Indeed,
it is easily checked that 
\[
  u(r) = \frac{2}{r h(r)} = \frac{r^2 + 1}{r^2}, \qquad
  v(r) = \frac{ (r^2 + 1) \ln(r^2 + 1)}{r^2} - 1
\]
are homogeneous solutions: $\hat H u = \hat H v = 0$.
It is also easily checked that:
\begin{equation}  \label{uvprops}
  u(r) > 0, \quad v(r) > 0, \qquad
  |u(r)| \lesssim \left\{ \begin{array}{lr}
  \frac{1}{r^2} &  r \to 0 \\ 1 & r \to \infty \end{array}  \right. ,
  \qquad |v(r)| \lesssim \left\{ \begin{array}{lr}
   r^2 &   r \to 0 \\ \log r  &  r \to \infty \end{array} \right. . 
 \end{equation}
Since $f \in L^1 \cap C((0,\infty))$, by~\eqref{uvprops},
\begin{equation}  \label{hatHinv}
  ( \hat H^{-1} f )(r) := 
  u(r)\int_{0}^{r}v(s)f(s) \; s ds + v(r) \int_{r}^{\infty} u(s)f(s) \; s ds 
\end{equation}
is well-defined for $r \in (0,\infty)$, and by the variation of parameters formula, 
$\hat H^{-1} f \in C^2((0,\infty))$ satisfies
\begin{equation}  \label{hatH}
  \hat{H} ( \hat H^{-1} f ) = f 
\end{equation}
for $r \in (0,\infty)$.

\begin{lem} \label{positivity}
If $0 \leq f \in L^1 \cap L^2 \cap C((0,\infty))$, then
$\hat H^{-1} f \geq 0$ and $(\hat H^{-1} + \be^2)^{-1} f \geq 0$.
\end{lem}
\begin{proof}
The non-negativity of $\hat H^{-1} f$ follows immediately from the
expression~\eqref{hatHinv}, and the positivity of $u$ and $v$ (see~\eqref{uvprops}).
Non-negativity of $\zeta := (\hat H + \be^2)^{-1} f$ follows from the maximum principle:
since $\zeta \in D(\hat H) \subset X$, 
it is bounded and vanishes as $r \to 0+$ and $r \to \infty$.
Using the equation~\eqref{hatH2}, if for some $\hat r \in (0,\infty)$,
$0 > \min_r \zeta(r) = \zeta(\hat r)$, then
\[
  0 \leq f( \hat r) = \left( (\hat{H} + \beta^2) \zeta \right)(\hat r)
   = -\Delta_r \zeta (\hat r) + \hat{V}(\hat r) \zeta(\hat{r}) < 0,
\]
since $\hat{V} > 0$, a contradiction . 
\end{proof}
\begin{lem}
\begin{equation}  \label{boundL1}
  \left| \left( (\hat{H} + \beta^2 )^{-1} f \right)(r) \right| 
  \le \left( \hat{H}^{-1} | f | \right) (r).
\end{equation}
\end{lem}
\begin{proof}
Write $f(r) = f^{+}(r) - f^{-}(r)$, $f^{+} := \max\{f, 0\}$, $ f^{-} := \max\{-f, 0\}$,
$f^{\pm} \in L^1 \cap L^2 \cap C((0,\infty))$. Set
\[
  \zeta^{\pm} := (\hat H + \be^2)^{-1} f^{\pm} \in D(\hat H)  \cap C^2((0,\infty)), \qquad
  \hat \zeta^{\pm} := H^{-1} f^{\pm}.
\] 
By Lemma~\ref{positivity}, $\hat \zeta^{\pm} \geq 0$
and $\zeta^{\pm} \geq 0$.
Then using the equations~\eqref{hatH2} and~\eqref{hatH}, 
\[
  (\hat H + \be^2)(\hat \zeta^{\pm} - \zeta^{\pm}) =
  f^{\pm} + \be^2 \hat \zeta^{\pm} - f^{\pm} = \be^2 \hat \zeta^{\pm} \geq 0,
\]
and so by Lemma~\ref{positivity} again,
$\hat \zeta^{\pm} - \zeta^{\pm} \geq 0$. That is,
\[
\hat{H}^{-1} f^{\pm} \ge (\hat{H} + \beta^2)^{-1} f^{\pm} \ge 0,
\] 
and so
\[
\begin{split}
  |(\hat{H} + \beta^2)^{-1}f | &= |(\hat{H} + \beta^2)^{-1}(f^{+} - f^{-})| \le
  (\hat{H} + \beta^2)^{-1}f^{+} + (\hat{H} + \beta^2)^{-1} f^{-} \\
  &\le \hat{H}^{-1} f^{+} + \hat H^{-1} f^{-} = \hat{H}^{-1}(f^+ + f^-)
  = \hat H^{-1} |f|,
\end{split}
\]
which is~\eqref{boundL1}.
\end{proof}

It follows from the pointwise bound~\eqref{boundL1} that
\begin{equation}  \label{relate}
  \| (\hat{H} + \be)^{-1} f \|_{L^\infty \cap r^2 L^1_{\leq 1}} \lesssim  
  \|  \hat{H}^{-1} f \|_{L^\infty \cap r^2 L^1_{\leq 1}},
\end{equation}
and so it remains to consider $\hat H^{-1}$.
\begin{lem}
\begin{equation} \label{hatHinvbound}
  \| \hat H^{-1} f \|_{L^\infty \cap r^2 L^1_{\leq 1}}  \lesssim  
  \| f \|_{L^{1,\log}}
\end{equation}
\end{lem}
\begin{proof}
We estimate directly using the explicit form~\eqref{hatHinv}
and the properties~\eqref{uvprops}. First the pointwise bound:
for $ r \le 2$,
\[ 
  |\hat H^{-1} f (r)| \lesssim  
  \frac{1}{r^2}\int_{0}^{r} s^2 |f(s)| s ds + r^2 \int_{r}^{2} \frac{1}{s^2} |f(s)| s ds 
  + r^2 \int_{2}^{\infty} |f(s)| s ds  \lesssim \|f\|_{L^1} \leq \| f \|_{L^{1,\log}},
\]
and for $r \ge 2$,
\[
\begin{split}
  | \hat H^{-1} f (r) | &\lesssim 
  \int_{0}^{2} s^2 |f(s)| s ds + \int_{2}^{r} \log(s) |f(s)| s ds + 
  \log r \int_{r}^{\infty} |f(s)| s ds \\
  & \lec \| f \|_{L^1} + \| f \|_{L^{1,\log}}  \lec \| f \|_{L^{1,\log}},
\end{split}
\]
which together give $\| \hat H^{-1} f \|_{L^\infty} \lec \| f \|_{L^{1,\log}}$.
For the $r^2 L^1_{\leq 1}$ bound: using Tonelli's theorem,
\[
\begin{split}
  \int_{0}^{1} \frac{1}{r^2} | \hat H^{-1} f (r) | r dr 
  & \lesssim \int_{0}^{1} r dr \frac{1}{r^4} \int_{0}^{r} s^2 |f(s)| s ds 
  + \int_{0}^{1} r dr \int_{r}^{1} \frac{1}{s^2} |f(s)| s ds 
  + \int_{0}^{1} r dr \int_{1}^{\infty} |f(s)| s ds\\
  & = \int_0^1 s^2 |f(s)| s ds \int_s^1 \frac{d r}{r^3}
  + \int_0^1 \frac{1}{s^2} |f(s)| s ds \int_0^s r dr + \frac{1}{2} \int_1^\infty |f(s)| s ds \\
  & \lec \int_0^1 |f(s)| s ds + \int_1^\infty |f(s)| s ds = \| f \|_{L^1},
\end{split}
\]
and so $\| \hat H^{-1} f \|_{r^2 L^1_{\leq 1}} \lec \| f \|_{L^1} \lec \| f \|_{L^{1,\log}}$,
giving~\eqref{hatHinvbound}.
\end{proof}

Finally,~\eqref{L1est} follows from~\eqref{hatHinvbound} and~\eqref{relate}.
\end{proof}

Next, we prove Proposition~\ref{L2prop} by direct integration.
\begin{proof}
Given $f \in L^2$, set 
\[
  \eta := (\hat H + \be^2)^{-1} f \; \in D(\hat H). 
\]
Taking the inner product of $\eta$ with
\[ 
  (\hat{H} + \beta^2) \eta = (-\Delta_r + \hat{V} + \beta^2) \eta = f ,
\]
we get
\[
  \int_{0}^{\infty} \left( \eta_r^{2} + \hat V \eta^2 + \beta^2 \eta^2 \right) r dr  
  = (f, \eta) \lesssim \|\eta\|_{L^2} \|f\|_{L^2},
\]
which, since $\hat V > 0$, gives the $L^2$ estimate
\begin{equation} \label{L2bound}
  \|\eta\|_{L^2}\lesssim \frac{1}{\beta^2} \|f\|_{L^2}
\end{equation}
and then also 
\begin{equation}  \label{otherbound}
  \|\eta_r\|_{L^2} + \|\sqrt{\hat{V}}\eta\|_{L^2} \lesssim  
  \|\eta\|^{\frac{1}{2}}_{L^2} \|f\|^{\frac{1}{2}}_{L^2} \lesssim 
  \frac{1}{\beta} \|f\|_{L^2}.
\end{equation}

We use~\eqref{L2bound} and~\eqref{otherbound} to estimate $\| \eta \|_{L^\infty}$.
For $r \leq 1$,  $\frac{1}{r} \lesssim \sqrt{\hat{V}(r)}$, and so
\begin{equation}\label{interval_1}
  \eta^2(r) = 2 \int_{0}^{r}\eta_r(s) \frac{\eta(s)}{s} s ds \lesssim \|\eta_r\|_{L^2} 
  \|\sqrt{\hat{V}}\eta\|_{L^2} \lesssim \frac{1}{\beta^2}\|f\|^2_{L^2}.
\end{equation}
For $ r \geq \frac{1}{\beta}$
\begin{equation}\label{interval_2}
  \eta^2(r) = -2 \int_{r}^{\infty} \frac{1}{s}\eta_r(s) \eta(s) s ds \lesssim 
  \frac{1}{r}\|\eta_r\|_{L^2} \|\eta\|_{L^2} \lesssim \frac{1}{\beta^2}\|f\|^2_{L^2}.
\end{equation}
For $1 \leq  r \leq \frac{1}{\beta}$,
\begin{equation}\label{interval_3}
  |\eta(r) - \eta(1) | \lesssim \int_{1}^{r} |\eta_r(s)| \frac{1}{s} s ds \lesssim \|\eta_r\|_{L^2} \log^{\frac{1}{2}}(r) \lesssim \|\eta_r\|_{L^2}\log^{\frac{1}{2}}(\frac{1}{\beta})
\lesssim \frac{\log^{\frac{1}{2}}(\frac{1}{\beta})}{\beta}\|f\|_{L^2}
\end{equation}
Combining~\eqref{interval_1},~\eqref{interval_2} and~\eqref{interval_3} gives
\begin{equation}  \label{Linftybound}
  \|\eta\|_{L^\infty} \lesssim \frac{1}{\beta} \log^{\frac{1}{2}} (\frac{1}{\beta}) \|f\|_{L^2}.
\end{equation}

Now we estimate $\| \eta \|_{r^2 L^1_{\leq 1}}$. Since
\[ 
  \hat{H} = -\Delta_r + \frac{4}{r^2}  - V, \; \mbox { with } \; 
  V  = \frac{4}{r^2 + 1} \; \mbox{  bounded },
\]
Using~\eqref{L2bound} again,
\[
  \left\| \frac{\eta}{r^2} \right\|_{L^2} \lec \| (-\Delta_r + \frac{4}{r^2}) \eta \|_{L^2} 
  \lesssim \| f + V \eta + \beta^2 \eta \|_{L^2}
  \lec \| f \|_{L^2} + \| \eta \|_{L^2} \lec \frac{1}{\be^2} \|f\|_{L^2},
\]
and then H\"older,
\[
  \int_{0}^{\be} \frac{|\eta|}{r^2} r dr \lesssim 
  \be \|\frac{\eta}{r^2}\|_{L^2}  \lesssim \frac{1}{\beta} \|f\|_{L^2}.
\]
For the remaining interval, use the estimate~\eqref{Linftybound} of $\|\eta\|_{L^\infty}$:
\[
  \int_{\be}^{1} \frac{|\eta|}{r^2} r dr \lesssim 
  \|\eta\|_{L^\infty} \log(\frac{1}{\be} ) \lesssim 
  \frac{1}{\beta} \log^{\frac{3}{2}}(\frac{1}{\beta}) \|f\|_{L^2} 
\]
Combining the last two bounds gives
\begin{equation} \label{r2L1bound}
  \|\eta\|_{r^2 L^1_{\leq 1}} \lesssim 
  \frac{1}{\beta} \log^{\frac{3}{2}}(\frac{1}{\beta}) \|f\|_{L^2}.
\end{equation}
Then~\eqref{boundL2} follows from~\eqref{Linftybound} and~\eqref{r2L1bound}.
\end{proof}

\subsection{Computation of some integrals} 

\begin{equation}  \label{h^2/r^2}
  \int_0^\infty \frac{h^2}{r^2} r dr = \int_0^\infty \frac{4 r}{(r^2+1)^2} r dr
  = -\frac{2}{r^2+1} \big |^\infty_0 = 2 
\end{equation}
\[
\begin{split}
  \int_0^\infty \frac{1}{r} \hat{h} h^3 r dr 
  &= 8 \int_0^\infty \frac{(r^2-1)r^3}{(r^2+1)^4} dr
  = 4 \int_0^\infty \left( \frac{1}{(r^2+1)^2} - \frac{3}{(r^2+1)^3} + 
  \frac{2}{(r^2+1)^4} \right) 2 r dr \\
  &= 4 \left( 1 - \frac{3}{2} + \frac{2}{3} \right) = \frac{2}{3}
\end{split}
\]
\[
\begin{split}
  \int_0^\infty h^4 r dr &= 16 \int_0^\infty \frac{r^4}{(r^2+1)^4} r dr
  = 8 \int_0^\infty \left( \frac{1}{(r^2+1)^2} - \frac{2}{(r^2+1)^3} 
  + \frac{1}{(r^2+1)^4} \right) 2 r dr \\
  &= 8 \left( 1 - 1 + \frac{1}{3}  \right) = \frac{8}{3}
\end{split}
\]
\[
\begin{split}
  \int_0^\infty \frac{1}{r^2} h^4 r dr &= 
  16 \int_0^\infty \frac{r^2}{(r^2+1)^4} r dr
  = 8 \int_0^\infty \left( \frac{1}{(r^2+1)^3} - \frac{1}{(r^2+1)^4} \right) 2 r dr \\
  &= 8 \left( \frac{1}{2} - \frac{1}{3}  \right) = \frac{4}{3}
\end{split}
\]
\[
\begin{split}
  \int_0^\infty \frac{1}{r^2} h^3 r dr &=
  8 \int_0^\infty \frac{r^2}{(r^2+1)^3} dr
  = 2 \int_0^{\frac{\pi}{2}} \sin^2(2 \th) d \th = \frac{\pi}{2}
\end{split}
\]
\begin{equation}  \label{h3/r}
\begin{split}
  \int_0^\infty \frac{1}{r} h^3 r dr &=
  4 \int_0^\infty \frac{r^2}{(r^2+1)^3} 2 r dr
  = 4 \int_0^\infty \left( \frac{1}{(r^2+1)^2} - \frac{1}{(r^2+1)^3} \right) 2 r dr \\
  &= 4 \left(  1 - \frac{1}{2} \right) = 2.
\end{split}
\end{equation}
\[
\begin{split}
  \int_0^\infty \frac{1}{r^3} h^3 r dr &=
  4 \int_0^\infty \frac{1}{(r^2+1)^3} 2 r dr
  = 4 \cdot \frac{1}{2} = -2,
\end{split}
\]
\[
\begin{split}
  \int_0^\infty \frac{1}{r^3} \hat{h} h^3 r dr 
  &= 4 \int_0^\infty \frac{(r^2-1)}{(r^2+1)^4} 2 r dr
  = 4 \int_0^\infty \left( \frac{1}{(r^2+1)^3} - \frac{2}{(r^2+1)^4} \right) 2 r dr \\
  &= 4 \left( \frac{1}{2} - \frac{2}{3} \right) = \frac{2}{3}
\end{split}
\]


\section*{Acknowledgement}
This research was partially supported by the first author's NSERC Discovery Grant
03847-18.

\bibliography{references}
\bibliographystyle{ieeetr}

\end{document}